\documentclass{amsart}
\usepackage{graphicx}
\usepackage{graphics}
\usepackage{psfrag}
\usepackage[active]{srcltx}
\usepackage{amssymb}
\usepackage{amscd}

\newtheorem{theorem}{Theorem}[section]
\newtheorem{cor}[theorem]{Corollary}
\newtheorem{lm}[theorem]{Lemma}

\theoremstyle{definition}
\newtheorem{df}[theorem]{Definition}

\newtheorem{con}[theorem]{Convention}

\theoremstyle{remark}

\numberwithin{equation}{section}

\newif\ifShowLabels
\ShowLabelstrue
\newdimen\theight
\newcommand\TeXref[1]{%
     \leavevmode\vadjust{\setbox0=\hbox{{\tt
               \quad\quad #1}}%
                    \theight=\ht0
                         \advance\theight by \dp0
                              \advance\theight by \lineskip
                                   \kern -\theight \vbox to
                                             \theight{\rightline{\rlap{\box0}}%
                                                   \vss}%
                                                         }}%
\newcommand{\labelp}[1]{\label{#1}%
    \ifShowLabels \TeXref{{#1}} \fi}
\ShowLabelsfalse

\newcommand\bt{\beta} 
\newcommand\dlt{\delta} 
\newcommand\gm{\gamma}  
\newcommand\g{\gamma}  %
\newcommand\kp{\kappa}  
\newcommand\lph{\alpha}
\newcommand\ps{\psi}
\newcommand\vph{\varphi}
\newcommand\x{\xi}
\newcommand\wi{\xy}
\newcommand\wj{\yz}
\newcommand\wk{\xz}
\newcommand\wu{u}
\newcommand\wv{v}
\newcommand\ww{w}
\newcommand\wx{x}
\newcommand\wy{y}
\newcommand\wz{z}
\newcommand\kk{K/K}  
\newcommand\kh{K/H}  
\newcommand\gxhi{G_\wx/H_\wi}
\newcommand\gyki{G_\wy/K_\wi}

\newcommand\h[1]{H_{#1}}
\newcommand\hs[2]{H_{#1,#2}}
\newcommand\ks[2]{K_{#1,#2}}
\newcommand\e[1]{e_{#1}}
\newcommand\vphi[1]{\vph_{#1}}
\newcommand\pair[2]{(#1,#2)}
\newcommand\kai[1]{\kappa_{#1}}
\newcommand\G[1]{G_{#1}}
\newcommand\gsq[2]{\G{#1}\times \G {#2}}

\newcommand\R{R} 
\newcommand\mo{^{-1}}
\newcommand\seq{\subseteq}
\newcommand\ssm{^{\scriptscriptstyle\smile}}
\newcommand\scir{\raise2pt\hbox{$\,\scriptscriptstyle\circ\,$}}
\newcommand\tsbigc{\textstyle \bigcup\limits}
\newcommand\tbigcup{\textstyle \bigcup}
\newcommand\tbigcap{\textstyle \bigcap}
\newcommand\tsum{\textstyle\sum}
\newcommand\idd[1]{id_{#1}}
\newcommand\id[1]{id_{#1}}
\newcommand\co{\textnormal{,}\ }     
\newcommand\opar{\textnormal{(}}     
\newcommand\cpar{\textnormal{)}}     
\newcommand\f[1]{{\mathfrak {#1}}}

\newcommand\smbcomma{\,,}
\newcommand{\mc}[1]{\mathcal{#1}}
\renewcommand\k[1]{K_{#1}}
\renewcommand\a{\alpha}
\renewcommand\r[2]{R_{{#1},{#2}}}
\newcommand\per{\textnormal{\myspace.\ }}
\newcommand{\myspace}{{\hspace*{.5pt}}}
\newcommand\qeddef{\qed}
\newcommand{\comma}{\textnormal{,}\ }
\newcommand\xy{{xy}}
\newcommand\cs[2]{{#1}_{#2}}
\newcommand\xx{{xx}}
\newcommand\ex[1]{e_{#1}}
\newcommand\po{\textnormal{.}\ }     
\newcommand\yx{{yx}}
\renewcommand\b{\beta}
\newcommand\rp{\!\mid\!}
\newcommand\wwz{{wz}}
\newcommand\varnot{\varnothing}
\newcommand\yz{{yz}}
\newcommand\xz{{xz}}
\newcommand\hvphs[1]{\hat{\vph}_{#1}}
\newcommand\ra{relation algebra}
\newcommand\cra[2]{\mathfrak{#1}\myshortspace[\mathcal{#2}\,]}
\newcommand{\myshortspace}{{\hspace*{.1pt}}}
\newcommand\craset[2]{#1\myshortspace[\mathcal{#2}\,]}

\newcommand{\refL}[1]{Lemma~\ref{L:#1}}
\newcommand{\refD}[1]{Definition~\ref{D:#1}}
\newcommand{\refC}[1]{Corollary~\ref{C:#1}}
\newcommand{\refT}[1]{Theorem~\ref{T:#1}}
\newcommand{\refS}[1]{Section~\ref{S:#1}}

\newcommand{\refF}[1]{Figure~\ref{F:#1}}
\newcommand\cc[3]{C_{#1#2#3}}
\newcommand\trip[3]{(#1,#2,#3)}
\newcommand\ez[1]{\mc E_{#1}}
\newcommand\hh{\h\xy\scir\h\xz}
\newcommand\diff{\sim}
\newcommand\conv{{}^{\scriptstyle\smallsmile}}
\newcommand\ident{1\mynegspace\textnormal{\rq}}
\newcommand{\mynegspace}{{\hspace*{-.5pt}}}
\newcommand\yy{{yy}}
\newcommand\www{{ww}}
\newcommand\refEq[1]{(\ref{Eq:#1})}
\newcommand\refCo[1]{Convention~\ref{Co:#1}}
\newcommand\uv{{uv}}
\newcommand\zx{{zx}}
\newcommand\zy{{zy}}
\newcommand\zw{{zw}}
\newcommand\quadr[4]{(#1,#2,#3,#4)}
\newcommand\xw{{xw}}
\newcommand\dd{D_1}
\newcommand\ddd{D_2}
\newcommand\yw{{yw}}
\newcommand\wwv{{wv}}
\newcommand\ttrip[3]{\tau_{{#1}{#2}{#3}}}
\newcommand\vphih[1]{{\hat\vph}_{#1}}
\newcommand\wwy{{wy}}
\newcommand\wwx{{wx}}
\newcommand\pq{{pq}}
\newcommand\pr{{pr}}
\newcommand\vp{\varphi}
\newcommand\grp[1]{G_{#1}}
\newcommand\ho[2]{\varphi_{#1#2}}
\newcommand\hll[2]{H_{{#1}{#2}}}
\newcommand\hr[2]{K_{#1#2}}
\newcommand\lz[1]{L_{#1}}
\newcommand\ze{0}
\newcommand\on{1}
\newcommand\relprodd{\,\vert\,}
\newcommand\qr{{qr}}
\newcommand\qp{{qp}}
\newcommand\pw{{pw}}
\newcommand\pps{{ps}}
\newcommand\ppt{{pt}}
\newcommand\crah[2]{\mathfrak{#1}\myshortspace[\bar{\mathcal{#2}}\,]}
\newcommand\rs[2]{\R_{#1,#2}}
\newcommand\al{a}
\newcommand\pt[1]{p_{\mc #1}}
\newcommand\vth{\vartheta}
\newcommand\st{{st}}
\newcommand\xu{{xu}}
\newcommand\vz{{vz}}
\newcommand\cras[2]{\mathfrak{#1}\myshortspace[{#2}\,]}

\begin{document}

\title[Coset relation algebras]{Coset relation algebras}
\author{Hajnal Andr\'eka and Steven Givant}%
\address{Hajnal Andr\'eka\\Alfr\'ed R\'enyi Institute of
Mathematics\\Hungarian Academy of Sciences\\Re\'altanoda utca
13-15\\ Budapest\\ 1053 Hungary}\email{andreka.hajnal@renyi.mta.hu}
\address{Steven Givant\\Mills College\\5000 MacArthur Boulevard, Oakland,
CA 94613}\email{givant@mills.edu}
\thanks{This research was
partially supported  by Mills College and  the Hungarian National
Foundation for Scientific Research, Grants T30314 and T35192.}

\begin{abstract} A \textit{measurable relation algebra} is a relation
algebra in which the identity element is a sum of atoms that can be
measured in the sense that the ``size" of each such atom can be
defined in an intuitive and reasonable way (within the framework of
the first-order theory of relation algebras). A large class of
concrete measurable set relation algebras, using systems of groups
and coordinated systems of isomorphisms between quotients of the
groups, is constructed in \cite{giv1}. This class of \textit{group
relation algebras} is not large enough to prove that every
measurable relation algebra is isomorphic to a group relation
algebra and hence is representable.

In the present article, the class of examples of measurable relation
algebras is considerably extended by adding one more ingredient to
the mix: systems of cosets that are used to ``shift" the operation
of relative multiplication.  It is shown that, under certain
additional hypotheses on the system of cosets, each such
\textit{coset relation algebra} with a shifted operation of relative
multiplication is an example of a measurable relation algebra. We
also show that the class of coset relation algebras does contain
examples of measurable relation algebras that are not representable
as set relation algebras. In later articles, it will be shown that
the class of coset relation algebras is adequate to the task of
describing all measurable relation algebras in the sense that every
atomic measurable relation algebra is essentially isomorphic to a
coset relation algebra (see \cite{ga}), and the class of group
relation algebras is similarly adequate to the task of representing
all measurable relation algebras in which the associated groups are
finite and cyclic (see \cite{ag}). An extended abstract for this
series of papers is \cite{ga02}.
\end{abstract}
\maketitle

\section{Introduction}\labelp{S:sec1}

In \cite{giv1},  a subidentity  element $x$---that is to say, an
element below the identity element---of a relation algebra is
defined to be \textit{measurable} if it is an atom and if the square
$x;1;x$ is a sum of  functional elements, that is to say, the sum of
elements that satisfy a characteristic property of relations that
are functions, namely, that the composition of the converse of the
relation with the relation itself is included in the identity
relation. The number of non-zero functional elements below the
square $x;1;x$ gives the \textit{measure}, or the \textit{size}, of
the atom $x$.  A relation algebra is said to be \textit{measurable}
if the identity element is the sum of measurable atoms.
 The group relation algebras  constructed in \cite{giv1}  are examples of measurable
relation algebras. It turns out, however, that they are not the
only examples of measurable relation algebras.

In this paper, a more general class of examples of measurable
relation algebras is constructed.  The algebras are obtained from
group relation algebras by ``shifting" the relational composition
operation by means of coset multiplication, using an auxiliary
system of  cosets. For that reason, we have called them
\textit{coset relation algebras}. By using this new construction, we
show that not all measurable relation algebras are representable. In
fact, as hinted in the proof, the class of coset relation algebras includes
infinitely many mutually
non-isomorphic, non-representable relation algebras. These are new
examples of non-representable relation algebras, with a completely
different underlying motivation than the examples that have appeared
so far in the literature.

These non-representable examples show that it was necessary to
broaden the class of group relation algebras, all of which are
representable, in order to get a representation theorem for all
measurable relation algebras. Indeed, the new class is broad enough
for representing all measurable relation algebras, as is shown in
\cite{ga}. It will be shown in \cite{ag} that if the groups $\G x$
constructed in an atomic, measurable relation algebra $\f A$ are all
finite and cyclic, then $\f A$ is essentially isomorphic to a full
group relation algebra. These theorems together provide far-reaching
generalizations of the atomic case of Maddux's representation
theorem for pair-dense relation algebras in \cite{ma91}. An extended
abstract describing these results and their interconnections
was published by the authors in \cite{ga02}.  The reader might find it
helpful to consult that article in order to get a
overview  of the program and its motivation.

In the next section of this paper,   the principal results
concerning group relation algebras are reviewed.  In the third
section, a system  of shifting cosets is introduced, and a new
operation of multiplication is defined with the help of these
cosets.  Characterizations are given in   the fourth section of when
the resulting algebra is a measurable relation algebra. A concrete
example of such a measurable coset relation algebra that, as it
turns out, is not representable, is given in the fifth section. The
final section  of the paper contains a decomposition theorem for
coset relation algebras that is similar to the decomposition theorem
for group relation algebras proved in \cite{giv1}.  Except for basic
facts about groups, this article is intended to be largely
self-contained. Readers who wish to learn more about the subject of
relation algebras are recommended to look at one or more of the
books Hirsch-Hodkinson\,\cite{hh02}, Maddux\,\cite{ma06}, or
Givant\,\cite{giv18}, \cite{giv18b}.

\section{Group relation algebras}\labelp{S:sec2}

For the convenience of the reader, here is a summary of the
essential notions and results from \cite{giv1} that will be needed
in this paper. Fix a system
\[G=\langle
\G x:x\in I\,\rangle\] of  groups $\langle \G
x\smbcomma\scir\smbcomma\mo\smbcomma \e x\rangle$ that are
pairwise disjoint, and an associated system
\[\varphi=\langle\vph_{xy}:\pair x y\in \mc E\,\rangle\] of
quotient isomorphisms.  Specifically, we require that  $\mc E$ be an equivalence
relation on the index set $I$, and for each pair $\pair x y$ in
$\mc E$, the function $\vphi {xy}$  be an isomorphism from a
quotient group of $\G x$ to a quotient group of $\G y$. Call
\[\mc F=\pair G \varphi\]  a \textit{group pair}.  The set    $I$ is
the \textit{group index set},  and the equivalence relation $\mc
E$ is the
 (\textit{quotient}) \textit{isomorphism index set}, of $\mc F$.
The normal subgroups of $\G x$ and $\G y$ from which the quotient
groups are constructed are uniquely determined by
 $\vphi\wi$, and will be denoted  by $\h\wi$ and $\k\wi$
respectively, so that $\vphi\wi$ maps $\gxhi$ isomorphically onto
$\gyki$.

The elements of the quotient group  $\gxhi$ are cosets, and hence
complexes (sets) of group elements.  As such they obey the standard
laws of group theory. Multiplication of cosets and unions of cosets
is an associative operation for which the normal subgroup $\h\wi$ is
the identity element that commutes with every other coset (and every
union of cosets).  Every coset has an inverse, and the operation of
forming inverses of cosets satisfies the first and second involution
laws: the inverse of the inverse of a coset is the original coset,
and the inverse of the composition of two cosets is the composition
of the inverses, in the reverse order.

For a fixed enumeration $\langle \hs\wi \g:\g<\kai\wi\rangle$
(without repetitions) of the cosets of $\h\wi$ in $\G x$, the
isomorphism $\vphi\wi$ induces a \textit{corresponding}, or
\textit{associated},  coset system
 of $\k\wi$ in $\G y$,  determined by the rule
\[\ks\wi \g=\vphi\wi(\hs\wi \g)\]
for each $\gm<\kai\wi$.
 In what follows, it is always assumed that the given coset
systems for $\h\wi$ in $\G x$ and for $\k\wi$ in $\G y$ are
associated in this manner. Furthermore, it is assumed that the
first elements of the coset systems are always the normal
subgroups themselves, so that
\[\hs\wi 0 =\h\wi\qquad\text{and}\qquad\ks\wi 0 =\k\wi.\]
\begin{df} \labelp{D:compro}  For each pair $\pair
x y$  in $\mc E$ and each $\a<\kai\wi$, define a  binary relation
$\r\wi \a $ by\[ \r{\wi}{\lph}= \tbigcup_{\gm < \kp_\wi}
H_{\wi,\gm}\times \vphi\wi[H_{\wi,\gm}\scir H_{\wi,\lph}]=
\tbigcup_{\gm < \kp_\wi} H_{\wi,\gm}\times (K_{\wi,\gm}\scir
K_{\wi,\lph})\per\] \qeddef\end{df}

\begin{lm}[Partition Lemma] \labelp{L:i-vi}  The
relations $\r\wi \a $\comma for $\a<\kai\xy$\comma are non-empty
and partition the set $\gsq x y$\per
\end{lm}

Let $U$ be the union of the disjoint system of groups,  and $E$ the
equivalence relation on $U$ induced by the isomorphism index set
$\mc E$, \[U=\tbigcup\{\cs Gx:x\in I\}\qquad\text{and}\qquad
E=\tbigcup\{\cs Gx\times\cs Gy:\pair xy\in \mc E\}\per\]  Take $A$
to be the collection of unions of all possible sets of the relations
of the form $\r\xy\lph$ for $\pair xy$ in $\mc E$ and
$\lph<\kai\xy$.   It turns out that $A$ is always the universe of a
complete and atomic Boolean set algebra.

\begin{theorem}[Boolean Algebra Theorem]\labelp{T:disj} The set $A$
is the universe of a complete\comma atomic Boolean algebra of
subsets of $E$\per The atoms in $A$ are the distinct relations
$\r\xy\lph$ for $\pair xy$ in $\mc E$ and $\lph<\kai\xy$\comma and
the distinct elements in $A$ are the unions of distinct sets of
atoms\per
\end{theorem}

\renewcommand\wi{\xx}

The set $A$ does not automatically  contain  the identity relation
$\id U$, so it is important to characterize when $\id U$ does belong
to $A$.

\begin{theorem}[Identity Theorem]\labelp{T:identthm1} For each element $x$ in
$I$\co the following conditions are equivalent\per
\begin{enumerate}
\item[(i)]
The identity relation $\idd {\G x}$ on $\G\wx$ is in $A$\per
\item[(ii)] $\r\xx 0=\idd {\G x}$\per
\item[(iii)]$\vphi\xx$ is the identity
automorphism of $\G\wx/\{\ex \wx\}$\po
\end{enumerate}
Consequently\co the set $A$ contains the identity relation $\id U$
on the base set $U$ if and only if \textnormal{(iii)}  holds for
each $\wx$ in $I$\po
\end{theorem}

\renewcommand\wi{\xy}
\renewcommand\wj{\yx}

Similarly, the set $A$ is not automatically closed under the
operation of converse.

\begin{theorem}[Converse Theorem]\labelp{T:convthm1}  For each
pair $\pair x y$  in $\mc E$\co the following conditions are
equivalent\per
\begin{enumerate}
\item[(i)] There are an $\a<\kai \xy$ and a $\b<\kai\yx$ such that $\r\xy\lph\mo=\r\yx\bt$\per
\item[(ii)] For every $\a<\kai \xy$ there is a $\b<\kai\yx$ such that $\r\xy\lph\mo=\r\yx\bt$\per
\item[(iii)]$\vphi\xy\mo=\vphi\yx$.
\end{enumerate}
Moreover\comma if one of these conditions holds\comma then we may assume that $\kai \yx=\kai\xy$\comma
and the index $\b$ in \textnormal{(i)} and \textnormal{(ii)} is uniquely determined by $
\hs\xy\lph\mo=\hs\xy\bt$\per The set $A$ is closed under converse if and only
 if \textnormal{(iii)} holds for all    $\pair x y$ in $\mc E$\po
\end{theorem}

\begin{con}
\labelp{Co:con11} Suppose $A$ is closed under converse. If a pair
$\pair x y$ is in $\mc E$, then $ \h\yx=\k\xy$, and therefore any
coset system for $\h\yx$ is also a coset system for $\k\xy$\per
Since the enumeration $\langle \hs \yx \g:\g<\kai \yx\rangle$ of
the cosets of $\h \yx$ can be freely chosen,
 we can and  always shall choose it so that
$\kai \yx=\kai \xy$ and
 $\hs \yx\g=\ks \xy\g$ for $\g<\kai \xy$\per  It then follows from  the
Converse Theorem that $\ks \yx \g = \hs \xy \g$ for $\g<\kai
\xy$\per
\end{con}

\renewcommand\wj{\yz}
Finally, the set $A$ is not in general closed under relational
composition, except when the composition is empty.

\renewcommand\wi{\xy}
\renewcommand\wj{\wwz}

\begin{lm}\labelp{L:emptycomp}  If $\pair \wx\wy$ and $ \pair w z$ are in $\mc E$\co and if $y\ne
w$\co then
\[\r\wi \a \rp \r\wj \b=\varnot
\]
for all $\lph<\kai\wi$ and $\bt <\kai\wj$\po
\end{lm}

\renewcommand\wi{\xy}
\renewcommand\wj{\yz}
\renewcommand\wk{\xz}

The most important case regarding the composition of two atomic
relations is when $y=w$.

\begin{theorem}[Composition Theorem]\labelp{T:compthm}  For
all pairs $\pair x y$ and $\pair y z$  in $\mc E$\co the following
conditions are equivalent\per
\begin{enumerate}
\item[(i)] The relation
$\r\xy 0 \rp\r\yz 0$ is in $A$\per
\item[(ii)] For each
$\lph<\kai\xy$ and each $\bt<\kai\yz$\comma the relation
$\r\xy\lph\rp\r\yz\bt$ is in $A$\per
\item[(iii)] For each
$\lph<\kai\xy$ and each $\bt<\kai\yz$\comma
\[\r \xy \a \rp \r \yz\b=\tbigcup\{\r \xz \g:
\hs \xz \g \seq \vphi \xy\mo[ \ks\xy \a\scir\hs \yz \b]\}\per\]
\item[(iv)]$ \h\xz\seq\vphi\xy\mo[\k\xy\scir\h\yz]$ and
$\hvphs\xy\rp\hvphs\yz=\hvphs\xz$\comma where
$\hvphs\xy$ and $\hvphs\xz$ are the mappings induced by
$\vphi\xy$ and $\vphi\xz$ on the quotient of $\G\wx$ modulo the
normal subgroup $\vphi\xy\mo[\k\xy\scir\h\yz]$\comma while
$\hvphs\yz$ is the isomorphism induced by $\vphi\yz$ on the
quotient of $\G\wy$ modulo the normal subgroup
$\k\xy\scir\h\yz$\per \end{enumerate} Consequently\co the set $A$
is closed under relational composition if and only if
\textnormal{(iv)} holds for all pairs  $\pair x y$ and $\pair y z$
in $\mc E$\po
\end{theorem}

\begin{cor}\labelp{C:compequiv}  If the set $A$ contains the
identity relation\comma then for any pairs $\pair x y$ and $\pair y
z$ in $\mc E$\co the following conditions are equivalent\per
\begin{enumerate}
\item[(i)] $\r\wi\a\rp\r\wj\b$ is in $A$ for some $\a<\kai\wi$ and
some $\b<\kai\wj$\per
\item[(ii)]$\r\wi\a\rp\r\wj\b$ is in $A$ for all $\a<\kai\wi$ and
all $\b<\kai\wj$\per
\end{enumerate}
\end{cor}

Putting together  the preceding theorems yields a characterization,
 purely in terms of the quotient isomorphisms, of when a
group pair gives rise to a complete and atomic set relation
algebra.

\begin{df}\labelp{D:cosfra} A \textit{group} \textit{frame}  is a group pair
\[\mc F=(\langle \G x:x\in
I\,\rangle\smbcomma\langle\vph_\xy:\pair x y\in \mc E\,\rangle)\]
satisfying the following  \textit{frame conditions} for all pairs
$\pair xy$ and $\pair yz$ in $\mc E$\per
\begin{enumerate}
\item[(i)]
$\vphi {xx}$ is the identity automorphism of $\G x/\{\e x\}$ for
all $x$\per
\item [(ii)]
$\vphi \yx=\vphi \xy\mo$\per
\item[(iii)] $\vphi \xy[\h \xy\scir\h\xz] =\k
\xy\scir\h \yz$ and $\vphi \yz[\k \xy\scir\h \yz] =\k
\xz\scir\k\yz$\per
\item[(iv)] $\hvphs \xy\rp\hvphs \yz=\hvphs \xz$\po
\end{enumerate}\qed\end{df}

Given a group frame \ $\mc F$, let $A$  be the collection of all
possible unions of relations of the form $\r\wi \a$ for $\pair x
y$ in $\mc E$ and $\a <\kai\wi$.  Call $A$ the set of
\textit{frame relations} constructed from $\mc F$.

\begin{theorem}[Group Frame Theorem]\labelp{T:closed} If $\mc F$ is
a group frame\co then the set  of frame relations
constructed from $\mc F$ is the universe of a complete\co
atomic\co measurable set relation algebra   with   base set and
unit \[U=\tbigcup  \{\G x:x\in I\}\qquad\text{and}\qquad
E=\tbigcup\{\cs Gx\times\cs Gy:\pair xy\in\mc E\}\]
respectively\per The atoms in this algebra are the relations of
the form $\r\wi \a$\comma and the subidentity atoms are the
relations of the form $\r\xx 0$\per The measure of $\r\xx 0$ is
just the cardinality of the group $\cs Gx$\per
\end{theorem}

The  theorem justifies the following definition.

\begin{df}\labelp{D:gradef} Suppose that $\mc F$ is a group frame\po The
 set \ra\ constructed from $\mc F$ in   Group Frame \refT{closed} is
 called the (\textit{full}) \textit{group relation algebra} on $\mc F$ and is
 denoted by $\cra G F$ \opar and its universe by  $\craset G
 F$\cpar\po A \textit{general group relation algebra} is defined
 to be an algebra that is embeddable into a full group relation
 algebra\po \qed
\end{df}


\section{Coset  Systems}\labelp{S:8}

\renewcommand\wi{\xy}
\renewcommand\wj{\yz}
\renewcommand\wk{\xz}

Group relation algebras by themselves are not  sufficient to
represent all measurable relation algebras as will be seen in
\refS{sec5}. However, it is shown in \cite{ga} that if the operation
of composition in a group relation algebra is changed slightly, then
the resulting class of new algebras   is sufficient to represent all
measurable relation algebras. We call these new algebras
\textit{coset relation algebras}.

The operation of relative multiplication in a coset relation algebra is a kind of ``shifted"
relational composition. To accomplish this shifting, it is necessary to add one more
ingredient to a group pair  $\mc F=\pair G\vph$, namely a system of cosets
\[\langle \cc x y z:\trip x y z\in\ez 3\rangle\comma
\]
where $\ez 3$ is the set of all triples $\trip xyz$ such that the
pairs $\pair xy$ and $\pair yz$ are in $\mc E$, and for each such
triple, the set $\cc x y z$ is a coset of the normal subgroup $\hh$
in $\G\wx$.  Call the resulting triple \[\mc F=\trip G \vph C\] a
\textit{group triple}.

Define a new binary multiplication operation $\,\otimes\,$ on the
pairs of atomic relations in the  Boolean algebra $A$ of \refT{disj}
as follows.

\begin{df}\labelp{D:smult}  For pairs
$\pair \wx\wy$ and $\pair \wy \wz$ in $\mc E$\comma put
\[\r \xy \a \otimes \r \yz\b=\tbigcup\{\r \xz \g:
\hs \xz \g \seq \vphi \xy\mo[ \ks\xy \a\scir\hs \yz \b]\scir\cc x
y z\}\] for all $\lph<\kai \xy$ and all $\bt<\kai\yz$\comma and
for  all other  pairs $\pair \wx\wy$ and $\pair \ww \wz$ in $\mc
E$ with $\wy\neq \ww$, put
\[\r\xy\lph\otimes\r\wwz\bt=\varnothing
\]
for all $\lph<\kai \xy$ and $\bt<\kai\wwz$\per
 Extend $\,\otimes\,$ to all of $A$ by requiring it to distribute over
 arbitrary unions.  This means that for all subsets $ X$ and $ Y$ of the set of atoms in $A$
\[(\tbigcup X)\otimes(\tbigcup Y) = \tbigcup\{\r\xy\lph\otimes\r\wwz\bt:
\r\xy\lph\in  X \text{ and } \r\wwz
\bt\in Y\}\per\]
 \qed
\end{df}

Comparing  the formula defining $\r\xy\lph\otimes\r\yz\bt$ in
\refD{smult}  with the value of the relational composition
$\r\xy\lph\rp\r\yz\bt$ given in Composition \refT{compthm}(iii), it
is clear that they are very similar in form. In the first case,
however, the coset $\vphi\xy\mo[\ks\xy\lph\scir\hs\yz\bt] $ of the
composite group $\h\xy\scir\h\xz$ has been shifted, through coset
multiplication by $\cc  x y z$\comma to another  coset of
$\h\xy\scir\h\xz$, so that in general the value of the
$\,\otimes\,$-product and the value of relational composition on a
given pair of atomic relations will be different, except in certain
cases, for example, the case in which the value is the empty set.

Observe that the product $\r\xy\lph\otimes\r\wwz\bt$ is, by
definition, a union of atomic relations in $A$ and is therefore
itself a member of $A$.  Since $\,\otimes\,$ is extended to all of
$A$ so as to be completely distributive over unions, and since $A$
is closed under arbitrary unions, it follows that $A$ is
automatically closed under the operation $\,\otimes\,$.  It is not
necessary to impose any special conditions on the quotient
isomorphisms to ensure this closure, as was the case for relative
multiplication in group relation algebras. However, to ensure that
$A$ contains the identity relation and is closed under converse, it
is still necessary to require conditions (i) and (ii) from
\refD{cosfra}.
Conditions (iii) and (iv) in \refD{cosfra} ensure that $A$ is closed
under relational composition. In order to get a class of algebras
large enough to represent all measurable relation algebras, it is
necessary to weaken condition (iv), but condition (iii) can be
retained. In fact, condition (iv) of \refD{cosfra} has  to be
changed only slightly, as can be seen in \refD{semiframedef} below.

Every element of a group  induces an \textit{inner automorphism} of
the group.  In particular, the coset $\cc x y z$\comma which is an
element of the quotient group
\[\G\wx/(\h\xy\scir\h\xz)\comma\]
induces an inner automorphism $\tau_{xyz}$ of the quotient group
that is defined by
\[\tau_{xyz}(D) = \cc x y z\mo\scir D\scir\cc x y z
\]
for every coset $D$ of $\h\xy\scir\h\xz$\per  This automorphism
coincides with the identity automorphism of the quotient group just
in case the coset $\cc x y z$ is in the center of the
 quotient group, that is to say, just in case
\[\cc x y z\scir D
= D\scir \cc x y z\] for every coset $D$ of $\hh $\per

\begin{df}\labelp{D:semiframedef} A group triple
\[\mc F=\trip G \vph C
\]
is a \textit{pre-semi-frame} if  the following three
conditions are satisfied\per
\begin{enumerate}
\item[(i)]
$\vphi {xx}$ is the identity automorphism of $\G x/\{\e x\}$ for
all $x$ in $I$\per
\item [(ii)]
$\vphi \yx=\vphi \xy\mo$ whenever $ \pair x y$ is in $\mc E$\per
\item[(iii)]  $\vphi \xy[\h
\xy\scir\h \xz]=\k \xy\scir\h \yz$ whenever $\trip x y z$ is in $\ez
3$\per %
\end{enumerate}
It is a \textit{semi-frame} if\comma in addition\comma the following
fourth condition  is also satisfied.
\begin{enumerate}
\item[(iv)] $\hvphs \xy\rp\hvphs \yz=\tau_{xyz}\rp \hvphs
\xz$ whenever $\trip x y z$ is in $\ez 3$\per
\end{enumerate} Conditions (i)--(iv) are called the \emph{semi-frame conditions}\per\qed
\end{df}

In condition (iv) of this  definition, it is understood that
$\hvphs\xy$, $\hvphs\yz$, and $\hvphs\xz$ are the induced
isomorphisms described in Composition \refT{compthm}. They  are well
defined by semi-frame condition (iii).

If the group triple  $\mc F$ is a pre-semi-frame, then the  Boolean
set algebra $A$ contains the identity relation on its base set (by
Identity \refT{identthm1}), and is closed under converse (by
Converse \refT{convthm1}) and under $\,\otimes\,$ (by \refD{smult}).
Consequently, it is permissible to form the algebra
\[\cra C F =\langle
A\smbcomma\cup\smbcomma\diff
\smbcomma\otimes\smbcomma\mo\smbcomma\id U\rangle\per
\]
Of  course, $\cra CF$ need not be a relation algebra, that is to
say, an abstract algebra of the form \[\f A=( A\smbcomma
+\smbcomma -\smbcomma ;\smbcomma\,\conv\smbcomma\ident)  \]  in
which the following axioms are valid.
\begin{enumerate}
 \item [(R1)] $r+s=s+r$\per \item
[(R2)] $r+(s+t)=(r+s)+t$\per \item [(R3)] $-(-r+s) +
-(-r+-s)=s$\per
\item[(R4)]
$r;(s;t)=(r;s);t$\per
\item[(R5)]
$r;\ident=r$\per
\item[(R6)] $r\conv\conv=r$\per
\item[(R7)]
$(r;s)\conv=s\conv;r\conv$\per
 \item[(R8)] $(r+s);t=r;t +s;t$\per
\item[(R9)]
$(r+s)\conv=r\conv +s\conv$\per
\item[(R11)] $(r;s)\cdot t=0$\quad implies\quad $(r\conv;t)\cdot
s=0$\per
\end{enumerate} (On the basis of the other axioms, (R11) is equivalent
to the original law (R10) that Tarski used as the tenth axiom---see,
for example, Definition 2.1 in Givant\,\cite{giv18}. Consequently,
we will not refer to (R10) again.)

Certain relation algebraic axioms are, however, automatically valid
in $\cra CF$.  For example, the \textit{Boolean axioms} (R1)--(R3)
are all valid, because the Boolean part of $\cra CF$ is a complete
and atomic Boolean set algebra. The \textit{first involution law}
(R6) involves only the operation of converse, so it is valid in
$\cra CF$. The operation $\,\otimes\,$ is distributive over
arbitrary unions, as is the operation of converse, so the
\textit{distributive axioms for relative multiplication  and
converse over addition}\comma (R8) and (R9) respectively, are valid
in $\cra CF$.

Each of  the remaining four axioms\comma the \textit{associative law
for relative multiplication} (R4), the \textit{identity law} (R5),
the \textit{second involution law} (R7), and the \textit{cycle law}
(R11) may fail in $\cra C F$. It is therefore   important to
impose conditions on the coset system of a pre-semi-frame that
characterize when each of these axioms does hold in $\cra CF$. This
task is simplified by certain observations. Three of the axioms,
namely (R4), (R5), and (R7), are equations, and one of them, namely
(R11), is an implication between two equations of the form
$\sigma=0$.  Each of the equations involved is \textit{positive} in
the sense that its terms are constructed from variables and constant
symbols using only the operation symbols for addition,
multiplication, relative multiplication, and converse. In
particular, there is no occurrence of the operation symbol for
complement.  Each of the axioms is also \textit{regular} in the
sense that no variable occurs more than once on either side of an
equation. It is a well-known result that positive, regular
equations,  and   implications between positive, regular equations of
the form $\sigma=0$, hold  in an atomic relation algebra (or in any
Boolean algebra with completely distributive operators) just in case
they hold for all atoms (see, for example, Corollaries 19.26 and
19.28 in Givant\,\cite{giv18b}). Thus, to verify that any one of
these axioms holds in $\cra C F$ under certain hypotheses on the
coset system, it suffices to verify that it holds for all atomic
relations.

We begin with  a lemma that says    equalities between
unions of atomic relations are equivalent to the corresponding coset
equalities.

\begin{lm}\labelp{L:auxlemma} Let $\mc F$ be a pre-semi-frame\comma and
$\trip x y z$ a triple  in $\ez 3$\per If  $D_0$ and $D_1$ are
each unions of cosets of $\hh$\comma  then the following
conditions are equivalent\per
\begin{enumerate}
\item[(i)]
$ D_0=D_1$\per
\item[(ii)] $\tbigcup\{\r\xz\gm:\hs\xz\gm\seq
D_0\}=\tbigcup\{\r\xz\x:\hs\xz\x\seq D_1\}$\per
\end{enumerate}
\end{lm}

\begin{proof} Condition (i) obviously implies
(ii). To establish the reverse implication, assume $D_0\neq D_1$\per
There must then be a coset $M$ of the subgroup $\hh$ that is
included  in one of the unions, say $D_0$, but not the other,
$D_1$\per It follows that $M$ must be disjoint from each of the
cosets in $D_1$, since two cosets of a subgroup are either equal or
disjoint. In particular,   each coset $\hs\xz\gm$ of $\h\xz$ that is
included in $M$  must be disjoint from $D_1$\comma so  the
corresponding relation $\r\xz\gm$, which is included in the
left-hand side of (ii), by assumption, must be disjoint from the
right-hand side of (ii), by  Partition \refL{i-vi}.
\end{proof}

Turn now to  the task of finding necessary and sufficient conditions
for various relation algebraic laws to hold in the algebra $\cra C
F$, and begin with the identity law (R5). This law is positive and
regular, so it suffices to characterize when it holds for all atomic
relations in $\cra CF$.

\begin{theorem}[Identity Law Theorem]\labelp{T:identthm2} Let $\mc F$ be
a pre-semi-frame\comma and
$\pair x y$ a pair in  $\mc E$\po  The following conditions are
equivalent\per
\begin{enumerate}\item[(i)] $
\r\xy \a\otimes \id U = \r\xy \a$ for some $\a<\kai\xy$\per
\item[(ii)] $
\r\xy \a\otimes \id U = \r\xy \a$ for all $\a<\kai\xy$\per
\item[(iii)] $\r\xy \a\otimes \r\yy 0 = \r\xy \a$ for some
$\a<\kai\xy$\per
\item[(iv)] $
\r\xy \a\otimes \r\yy 0 = \r\xy \a$ for all $\a<\kai\xy$\per
\item[(v)] $\cc x y y=\h\xy$\per
\end{enumerate}
Consequently, the identity law holds in the  algebra $\cra C F$ if
and only if \textnormal{(v)} holds for all pairs $\pair x y$ in $\mc
E$\po
\end{theorem}
\begin{proof}   Identity \refT{identthm1} and
semi-frame condition (i) imply that
\begin{equation*}
\id U =\tsbigc_{\ww\in I}\r\www 0\per \end{equation*} Therefore,
\begin{equation*}\tag{1}\labelp{Eq:identlaw2}
\r\xy\a\otimes\id U=\tsbigc_{\ww\in I}\r\xy\a\otimes\r\www
0=\r\xy\a\otimes \r\yy 0\comma\end{equation*} by the
distributivity of $\,\otimes\,$ over arbitrary unions, and the
fact  that \[\r\xy\a\otimes\r\www 0=\varnot\]  whenever
$\ww\neq\wy$. The equivalences of (i) with (iii), and of (ii) with
(iv), are  immediate consequences of \refEq{identlaw2}.

We show the equivalence of (iii) and (v), from which it follows
trivially that conditions (iii), (iv), and (v) are all equivalent.
We have by \refD{smult}, the convention that $\hs\yy 0=\{e_y\}$, and
semi-frame condition (ii) and the convention that
$\ks\xy\a=\vphi\xy(\hs\xy\a)$. Now, (iii) holds, by \refL{auxlemma}
just in case $\hs\xy\a\scir\cc x y y = \hs\xy\a$, and this last
equality holds just in case $\cc x y y =\h\xy$, which is just
condition (v). This establishes the equivalence of conditions
(iii)--(v), and hence of all five conditions, in the statement of
the theorem.

The identity law holds in $\cra C F$ just in case it holds for all
atoms $\r\xy\a$\per Apply  the equivalence of (ii) and (v) in the
the statement of the theorem to conclude that the identity law holds
in $\cra CF$ just in case $\cc x y y=\h\xy$ for all pairs $\pair x
y$ in $\mc E$\per
\end{proof}

Take up now the task of characterizing when the cycle law (R11)
holds. It suffices to characterize when this implication holds for
atoms, and for atoms $r$, $s$, and $t$, the implication is
equivalent to the following \textit{atomic form of the cycle law}:
\[s\le r\conv;t\qquad\text{implies}\qquad t\le r;s\per\]

\begin{theorem}[Cycle Law Theorem]\labelp{T:clawc} Let $\mc F$ be a pre-semi-frame\comma and
 $\trip x y z$  a triple in $\ez 3$\po  The
following conditions are equivalent\per
\begin{enumerate}
 \item[(i)]  If
$\r\yz\b\seq\r\xy\a\mo\otimes\r\xz\g$\comma then
$\r\xz\g\seq\r\xy\a\otimes\r\yz\b$\comma for some
$\a<\kai\xy$\comma $\b<\kai\yz$\comma and $\g<\kai\xz$\per
\item[(ii)]  If
$\r\yz\b\seq\r\xy\a\mo\otimes\r\xz\g$\comma then
$\r\xz\g\seq\r\xy\a\otimes\r\yz\b$\comma for all
$\a<\kai\xy$\comma $\b<\kai\yz$\comma and $\g<\kai\xz$\per
\item[(iii)]$\vphi\xy[\cc x y z]=\cc y x z\mo$\per
\end{enumerate} Consequently\comma the cycle law holds in
 the algebra $\cra C F$ just in case \textnormal{(iii)} holds for all triples
 $\trip x y z$ in $\ez 3$\per
\end{theorem}

\begin{proof}
Fix indices $\a<\kai\xy$\comma $\b<\kai\yz$\comma and
$\g<\kai\xz$\comma with the goal of establishing the equivalence
of conditions (i) and (iii). Choose $\dlt<\kai\xy$ so that
\begin{align*}\hs\xy\a\mo&=\hs\xy\dlt\comma\tag{1}\labelp{Eq:clawb1}\\
\intertext{and observe that}
\r\xy\a\mo&=\r\yx\dlt\comma\tag{2}\labelp{Eq:clawb2.1}\\
\intertext{by semi-frame condition (ii) and Converse
\refT{convthm1}. Semi-frame condition (ii) and  \refCo{con11}
imply that} \vphi\xy\mo&=\vphi\yx\tag{3}\labelp{Eq:clawb2}\\
\intertext{and}
\ks\yx\dlt&=\hs\xy\dlt\per\tag{4}\labelp{Eq:clawb3}
\end{align*} Combine
\refEq{clawb1}--\refEq{clawb3}, and use the definition of
$\,\otimes\,$, to arrive at
\begin{align*}
\r\xy\a\mo\otimes\r\xz\g&=\r\yx\dlt\otimes\r\xz\g\\
&=\tbigcup\{\r\yz\x:\hs\yz\x\seq\vphi\yx\mo[\ks\yx\dlt\scir\hs\xz\g]\scir\cc
y x z\}\\
&=\tbigcup\{\r\yz\x:\hs\yz\x\seq\vphi\yx\mo[\hs\xy\dlt\scir\hs\xz\g]\scir\cc
y x z\}\\
&=\tbigcup\{\r\yz\x:\hs\yz\x\seq\vphi\xy[\hs\xy\a\mo\scir\hs\xz\g]\scir\cc
y x z\}\per
 \end{align*} It follows from this string of
equalities and Partition \refL{i-vi}  that the inclusion
\begin{equation*}\tag{5}\labelp{Eq:clawb4}
\r\yz\b\seq\r\xy\a\mo\otimes\r\xz\g \end{equation*} is equivalent
to the inclusion
\begin{equation*}\tag{6}\labelp{Eq:clawb5}
\hs\yz\b\seq\vphi\xy[\hs\xy\a\mo\scir\hs\xz\g]\scir\cc  y x z\per
\end{equation*}

A completely analogous argument shows  that the inclusion
\begin{equation*}\tag{7}\labelp{Eq:clawb6}\r\xz\g\seq\r\xy\a\otimes\r\yz\b
\end{equation*}
is equivalent to the  inclusion
\begin{equation*}\tag{8}\labelp{Eq:clawb7}
\hs\xz\g\seq\vphi\xy\mo[\ks\xy\a\scir\hs\yz\b]\scir\cc x y z\per
\end{equation*}

We now  transform  \refEq{clawb5} in a series of steps. Multiply
each side of \refEq{clawb5}  on the left by the coset $\ks\xy\a$
to obtain the equivalent  inclusion
\begin{equation*}\tag{9}\labelp{Eq:clawb8}
\ks\xy\a\scir\hs\yz\b\seq\ks\xy\a\scir\vphi\xy[\hs\xy\a\mo\scir\hs\xz\g]\scir\cc
y x z\per
 \end{equation*} Notice that the
right side  of \refEq{clawb8} is  a coset  of  $\kh$\per (For
example, $\cc y x z$ is a coset of $\h\yx\scir\h\yz$, which  is
equal to $\k\xy\scir\h\yz$\per  Also, $\hs\xy\a\mo\scir\hs\xz\g$ is
a coset of $\hh$\comma and $\hvphs\xy$ maps cosets of $\hh$ to
cosets of $\kh$\comma so $\vphi\xy[\hs\xy\a\mo\scir\hs\xz\g]$ is  a
coset of $\k\xy\scir\h\yz$. Finally, the product of two cosets of
$\kh$ with the coset $\ks\xy\a$ of $\k\xy$ is again a coset of
$\kh$.) The left side of \refEq{clawb8} is also a coset of $\kh$.
Since two cosets of the same group are either equal or disjoint,
the inclusion in \refEq{clawb8} is  equivalent to the equality
\begin{equation*}\tag{10}\labelp{Eq:clawb9}
\ks\xy\a\scir\hs\yz\b=\ks\xy\a\scir
\vphi\xy[\hs\xy\a\mo\scir\hs\xz\g]\scir\cc  y x z\per
 \end{equation*}
Observe that
\begin{align*} \ks\xy\a\scir \vphi\xy[\hs\xy\a\mo\scir\hs\xz\g]&=
\vphi\xy[\hs\xy\a]\scir \vphi\xy[\hs\xy\a\mo\scir\hs\xz\g]\\ &=
\vphi\xy[\hs\xy\a \scir \hs\xy\a\mo\scir\hs\xz\g]\\
&=\vphi\xy[\h\xy\scir\hs\xz\g]\comma\end{align*} by the definition
of  $\ks\xy\a$ (which implies that $\vphi \xy[\hs\xy\a]=\ks\xy
\a$)\comma the isomorphism properties of $\vphi\xy$\comma and the
 laws of group theory. Equation \refEq{clawb9} can
therefore be rewritten in the form
\begin{equation*}\tag{11}\labelp{Eq:clawb10}
\ks\xy\a\scir\hs\yz\b=\vphi\xy[\h\xy\scir\hs\xz\g]\scir\cc  y x
z\per\end{equation*} Apply $\vphi\xy\mo$ to  both sides  of
\refEq{clawb10}, and use the isomorphism properties of $\vphi
\xy\mo$, to obtain
\begin{align*}\tag{12}\labelp{Eq:clawb20}
\vphi\xy\mo[\ks\xy\a\scir\hs\yz\b]&=\vphi\xy\mo[\vphi\xy[\h\xy\scir\hs\xz\g]\scir\cc
y x z]\\
&=\vphi\xy\mo[\vphi\xy[\h\xy\scir\hs\xz\g]\myspace]\scir\vphi\xy\mo[\cc
y x z]\\ &=\h\xy\scir\hs\xz\g\scir\vphi\xy\mo[\cc  y x z]\per
\end{align*}
Now $\cc y x z$ is a coset of $\h\yx\scir\h\yz$, which, in turn,
is equal to $\k\xy\scir\h\yz$\comma and   $\vphi\xy $ maps the
group $\G\wx/(\h\xy\scir\h\xz)$ isomorphically to the group
$\G\wy/(\k\xy\scir\h\yz)$, so the inverse image $\vphi\xy\mo[\cc y
x z] $ must be a coset of $\h\xy\scir\h\xz$\per  Consequently,
\[\h\xy\scir\vphi\xy\mo[\cc y x z]=\vphi\xy\mo[\cc y x z]\comma
\]
so that \refEq{clawb20} reduces to
\begin{equation*}\tag{13}\labelp{Eq:clawb11}
\vphi\xy\mo[\ks\xy\a\scir\hs\yz\b]= \hs\xz\g\scir\vphi\xy\mo[\cc y
x z]\per\end{equation*} Summarizing,  inclusion \refEq{clawb5},
and hence also inclusion \refEq{clawb4}, is equivalent to equation
\refEq{clawb11}.

We now subject equation \refEq{clawb7} to similar, but simpler,
transformations.  Multiply each side of \refEq{clawb7} on the
right by $\cc x y z\mo$, and use the laws of group theory, to
obtain
\begin{equation*}\tag{14}\labelp{Eq:clawb12}
\hs\xz\g\scir\cc x y z\mo\seq\vphi\xy\mo[\ks\xy\a\scir\hs\yz\b]\per
\end{equation*}
Each side of this  inclusion is a coset of $\hh$\per Since two
 cosets of the same group are equal or disjoint, the inclusion in
\refEq{clawb12} is equivalent to the equation
\begin{equation*}\tag{15}\labelp{Eq:clawb13}
\hs\xz\g\scir\cc x y z\mo=\vphi\xy\mo[\ks\xy\a\scir\hs\yz\b]\per
\end{equation*}
Therefore,   inclusion \refEq{clawb7}, and hence also inclusion
\refEq{clawb6}, is equivalent to  equation \refEq{clawb13}.

Combine the results of the last two paragraphs to arrive at the
following conclusion: inclusion \refEq{clawb4} implies inclusion
\refEq{clawb6} just in case equation \refEq{clawb11} implies
equation \refEq{clawb13}. Compare \refEq{clawb11} with
\refEq{clawb13}: the former implies the latter just in case
\[\hs\xz\g\scir\vphi\xy\mo[\cc  y x
z]=\hs\xz\g\scir\cc x y z\mo\comma
\]
or, equivalently, just in case
\begin{equation*}\tag{16}\labelp{Eq:clawb115}
\vphi\xy\mo[\cc  y x z]=\cc x y z\mo\per\end{equation*}  Form the coset
inverse of both sides  of \refEq{clawb115}, and  apply the
isomorphism properties of $\vphi\xy\mo$, to rewrite
\refEq{clawb115} as
\begin{equation*}\tag{17}\labelp{Eq:clawb116}
\vphi\xy\mo[\cc y x z\mo]=\cc  x y z\per\end{equation*} Apply
$\vphi\xy$ to both sides of \refEq{clawb116} to arrive at the
equivalent equation
\begin{equation*}\tag{18}\labelp{Eq:clawb14}
\vphi\xy[\cc x y z]=\cc  y x z\mo\per \end{equation*}

It has been shown that the implication from \refEq{clawb4} to
\refEq{clawb6} for fixed $\a$, $\b$, and $\g$,  is equivalent  to
\refEq{clawb14}. This means that conditions (i) and (iii) in the
statement of the theorem are equivalent.  Since the formulation of
(iii) does not involve any of the three indices $\a$, $\b$, and
$\g$, it follows that (iii) implies (i) for each such triple of
indices, and hence  (iii) implies (ii). The implication from (ii)
to (i) is immediate.

The cycle law holds in $\cra CF$ just in case it holds for all
atoms. Consider such  a triple of  atoms \[\r\xy\a\comma\qquad
\r\wwz\b\comma\qquad \r\uv\g\comma\] we want to show
\[\r\wwz\b\subseteq\r\xy\a\mo\otimes\r\uv\g\qquad\text{implies}\qquad\r\uv\g\subseteq\r\xy\a\otimes\r\wwz\b\per\]
If $\wy=\ww$ and $u=\wx$ and $v=\wz$\co then the atomic form of the
cycle law holds for the triple  just in case $\vphi\xy[\cc x y
z]=\cc y x z\mo$\comma by the equivalence of conditions (ii) and
(iii) in the first part of the theorem.

Assume $\wy\neq\ww$ or $u\neq\wx$ or $v\neq \wz$. We show that the
law holds trivially, because the left side of the implication
reduces to the empty relation. Choose $\x<\kai\xy$ such that
\begin{align*}
\hs\xy\a\mo&=\hs\xy\x\comma\\ \intertext{and observe that}
\r\xy\a\mo&=\r\yx\x\comma\tag{19}\labelp{Eq:clawa1}
\end{align*}
by Converse \refT{convthm1}.  Consequently,
\[\r\xy\a\mo\otimes\r\uv\g=\r\yx\x\otimes \r\uv\g\seq\G\wy\times\G v\comma
\]
by \refEq{clawa1}, the definition of $\,\otimes\,$, and Partition
\refL{i-vi}. On the other hand, the relation $\r\wwz\b$ is included
in $\G\ww\times\G\wz$\comma by Partition \refL{i-vi}.
  The hypothesis that $\ww\neq\wy$ or $\wz\neq v$ implies that the two
Cartesian products \[\G\wy\times\G
v\qquad\text{and}\qquad\G\ww\times\G\wz\] are disjoint, since
distinct groups in the given group system are assumed to be
disjoint. It follows that
\[\r\wwz\b\cap (\r\xy\a\mo\otimes\r\uv\g)\seq (\G\ww\times\G\wz)\cap
(\G\wy\times\G v) =\varnothing\per
\]
Since $\r\wwz\b$ is non-empty, this argument shows that the
antecedent of the implication does not hold, so the entire
implication  must be true. If $u\neq \wx$, then
\[\r\xy\a\mo\otimes\r\uv\g=\r\yx\x\otimes \r\uv\g=\varnothing\comma
\]
by \refEq{clawa1} and the definition of $\,\otimes\,$, so again the
antecedent of the asserted implication is false, which means that
the entire implication is true.
\end{proof}

The next two characterization theorems make use of semi-frame
condition (iv).  We begin with an auxiliary lemma. Notice that (i)
of the lemma coincides with semi-frame condition (iv) stated for the
triple $(x,y,z)$ in $\ez 3$.

\begin{lm}\labelp{L:sfimage} Suppose that $\mc F$ is a
pre-semi-frame\comma and $\trip xyz$ a triple in $\ez 3$. The
following are equivalent\per
\begin{enumerate}
\item[(i)] If $Q$ is a union of cosets of the subgroup
$\h\xy\scir\h\xz$ in $\G\wx$\comma then
\[\vphi\yz[\vphi\xy[Q]\myspace]= \vphi\xz[\cc x y z\mo\scir Q\scir \cc x y z]\per
\]
\item[(ii)] If $Q$ is a union of
cosets of the subgroup $\k\xy\scir\h\yz$ in $\G\wy$\comma then
\[\vphi\xz\mo[\vphi\yz[Q]\myspace]=\cc x y
z\mo\scir\vphi\xy\mo[Q]\scir \cc x y z\per
\]
\item [(iii)]  If $Q$ is a union of
cosets of the subgroup $\k\xz\scir\k\yz$ in $\G\wz$\comma then
\[\cc x y z\scir \vphi\xz\mo[Q]=\vphi\xy\mo[\vphi\yz\mo[Q]\myspace]\scir
\cc x y z\per
\]
\end{enumerate}
\end{lm}

\begin{proof} Assume (i). To prove (ii), let $Q$ be a union of
cosets of $\k\xy\scir\h\yz$. By semi-frame condition (iii), which
holds by the assumption that $\mc F$ is a pre-semi-frame, we have
that $\vphi\xy\mo[Q]$ is a union of cosets of $\h\xy\scir\h\xz$.
Substitute $\vphi\xy\mo[Q]$ in place of $Q$ in (i) to get
\begin{equation*}\tag{1}\labelp{Eq:auxlem1}
\vphi\yz[\vphi\xy[\vphi\xy\mo[Q]]]\myspace=\vphi\xz[\cc x y
z\mo\scir \vphi\xy\mo[Q]\scir \cc x y z]\per
\end{equation*}
On both sides of \refEq{auxlem1} there is a union of cosets of
$\k\xz\scir\k\yz$, again by semi-frame condition (iii).
Apply $\vphi\xz\mo$ to both sides of (i) to obtain
\begin{equation*}\tag{2}\labelp{Eq:auxlem2}
\vphi\xz\mo[\vphi\yz[\vphi\xy[\vphi\xy\mo[Q]]]\myspace]=
\vphi\xz\mo[\vphi\xz[\cc x y z\mo\scir \vphi\xy\mo[Q]\scir \cc x y
z]]\per
\end{equation*}
Use the inverse property of functions to obtain (ii) from
\refEq{auxlem2}. (Notice that the symbol ${}\mo$ is being used two
different ways: to denote the inverse functions of the isomorphisms
$\vphi\xy$ and $\vphi\xz$, and to denote the group inverse of the
coset $\cc xyz$. The two different meanings  of this particular
symbol are standard, and should not cause the reader any confusion.)

In a similar way, to get (iii) from (ii), let $Q$ be a union of
cosets of $\k\xz\scir\k\yz$. Substitute $\vphi\yz\mo[Q]$ in place of
$Q$ in (ii), multiply both sides by $\cc xyz$ on the left, and use
the inverse property of functions to arrive at (iii).

To get (i) from (iii), let $Q$ be a union of cosets of
$\h\xy\scir\h\xz$. In (iii), substitute  $\vphi\xz[\cc xyz\mo\scir
Q\scir\cc xyz]$ in place of $Q$, and use the inverse property of
functions, to get
\begin{equation*}\tag{3}\labelp{Eq:auxlem3}
\cc xyz\scir \cc xyz\mo\scir Q\scir\cc xyz =
\vphi\xy\mo[\vphi\yz\mo[\vphi\xz[\cc xyz\mo\scir Q\scir\cc
xyz]]]\scir \cc x y z\per
\end{equation*}
Multiply both sides with $\cc xyz\mo$ on the right, and use the
inverse property for groups to get
\begin{equation*}\tag{4}\labelp{Eq:auxlem4}
Q = \vphi\xy\mo[\vphi\yz\mo[\vphi\xz[\cc xyz\mo\scir Q\scir\cc
xyz]]]\per
\end{equation*}
Finally, apply $\vphi\xy$ and then $\vphi\yz$ to both sides of
\refEq{auxlem4} and use the inverse property of functions to get (i)
from \refEq{auxlem4}.
\end{proof}

Turn next to the second involution law.  As before, it suffices to
characterize when the equation holds for pairs of atoms in $\cra C
F$.

\begin{theorem}[Second Involution Law Theorem]\labelp{T:invollawb} Let $\mc F$ be a semi-frame\comma and
$\trip x y z$  a triple  in $\ez 3$\per  The following conditions
are equivalent\per
\begin{enumerate}
\item[(i)]$(\r\xy \a\otimes\r\yz \b)\mo=\r\yz \b\mo\otimes\r\xy
\a\mo$ for some $\a<\kai\xy$ and some $\b<\kai\yz$\per
\item[(ii)]$(\r\xy \a\otimes\r\yz \b)\mo=\r\yz \b\mo\otimes\r\xy
\a\mo$ for all $\a<\kai\xy$ and all $\b<\kai\yz$\per
 \item[(iii)]$\vphi\xz[\cc x y z]=\cc z y x\mo$\per 
\end{enumerate} Consequently\comma the second involution law holds in the algebra $\cra C F$
just in case \textnormal{(iii)}
 holds for all triples $\trip x y z$   in $\ez 3$\per
\end{theorem}
\begin{proof}    Fix $\a<\kai\xy$ and
$\b<\kai\yz$\comma with the goal of showing that conditions (i) and
(iii) are equivalent. The first step is to work out concrete
formulas for the expressions on the left and right sides of
condition (i).  The definition of $\,\otimes\,$ gives
\begin{equation*}\tag{1}\labelp{Eq:inlawb1}
\r\xy\a\otimes\r\yz\b=\tbigcup\{\r\xz\gm:\hs\xz\gm\seq\vphi\xy\mo[\ks\xy\a\scir\hs\yz\b]\scir\cc
x y z\}\per
\end{equation*}
Form the relational converse of both sides of \refEq{inlawb1}, and
apply the distributivity of converse over arbitrary unions, to
obtain
\begin{equation*}
(\r\xy\a\otimes\r\yz\b)\mo=\tbigcup\{\r\xz\gm\mo:\hs\xz\gm\seq\vphi\xy\mo[\ks\xy\a\scir\hs\yz\b]\scir\cc
x y z\}\per\end{equation*} This last equation is equivalent to the
equation
\begin{equation*}\tag{2}\labelp{Eq:inlawb2}
(\r\xy\a\otimes\r\yz\b)\mo=\tbigcup\{\r\xz\gm\mo:\hs\xz\gm\mo
\seq(\vphi\xy\mo[\ks\xy\a\scir\hs\yz\b]\scir\cc x y
z)\mo\}\comma\end{equation*} by the first involution law for groups
(which says that $(g\mo)\mo=g$ for every element $g$ in a group).
Converse \refT{convthm1} asserts that
\[\r\xz\gm\mo=\r\zx\x\qquad\text{just in
case}\qquad\hs\xz\gm\mo=\hs\xz\x\per
\]
Substitute the right  side of each of these equations into the right
side of \refEq{inlawb2} to  arrive at
\begin{equation*}\tag{3}\labelp{Eq:inlawb3}
(\r\xy\a\otimes\r\yz\b)\mo=\tbigcup\{\r\zx\x:\hs\xz\x
\seq(\vphi\xy\mo[\ks\xy\a\scir\hs\yz\b]\scir\cc x y
z)\mo\}\per\end{equation*} Use the second involution law for groups
(which says that $(g\scir h)\mo=h\mo\scir g\mo$ for all elements $g$
and $h$ in a group)  and the isomorphism properties of $\vphi\xy$ to
see that
\begin{align*}(\vphi\xy\mo[\ks\xy\a\scir\hs\yz\b]\scir\cc
x y z)\mo&= \cc x y
z\mo\scir(\vphi\xy\mo[\ks\xy\a\scir\hs\yz\b])\mo\\ &= \cc x y
z\mo\scir\vphi\xy\mo[(\ks\xy\a\scir\hs\yz\b)\mo]\\ &= \cc x y
z\mo\scir\vphi\xy\mo[\hs\yz\b\mo\scir\ks\xy\a\mo]\per\end{align*}
  Replace the first term by the last term in the
right  side of \refEq{inlawb3} to conclude that
\begin{equation*}\tag{4}\labelp{Eq:inlawb4}
(\r\xy\a\otimes\r\yz\b)\mo=\tbigcup\{\r\zx\x:\hs\xz\x \seq\cc x y
z\mo\scir\vphi\xy\mo[\hs\yz\b\mo\scir\ks\xy\a\mo]\}\per
\end{equation*}

The next step is to work out an analogous expression for  the right
side of (i). Choose $\rho<\kai\xy$ and $\eta<\kai\yz$ so that
\begin{equation*}\tag{5}\labelp{Eq:inlawb5}
\ks\xy\a\mo=\ks\xy\rho\qquad\text{and}\qquad\hs\yz\b\mo=\hs\yz\eta\per
\end{equation*}
Apply semi-frame condition (ii) and Converse \refT{convthm1} to
obtain
\begin{equation*}\tag{6}\labelp{Eq:inlawb6}
\r\xy\a\mo=\r\yx\rho\qquad\text{and}\qquad
\r\yz\b\mo=\r\zy\eta\per\end{equation*}  Use \refEq{inlawb6} and the
definition of $\,\otimes\,$ to get
\begin{align*}\tag{7}\labelp{Eq:inlawb7.0}\r\yz\b\mo\otimes\r\xy\a\mo&=\r\zy\eta\otimes \r\yx\rho\\
&=\tbigcup\{\r\zx\gm:\hs\zx\gm\seq\vphi\zy\mo[\ks\zy\eta\scir\hs\yx\rho]\scir\cc
z y x\}\per
\end{align*}
\refCo{con11} and \refEq{inlawb5} yield
\begin{equation*}\tag{8}\label{Eq:inlawb8.0}
\ks\zy\eta=\hs\yz\eta=\hs\yz\b\mo\quad\text{and}\quad\hs\yx\rho
=\ks\xy\rho=\ks\xy\a\mo\per
\end{equation*}
Combine \refEq{inlawb7.0} and \refEq{inlawb8.0} to arrive at
\begin{equation*}\tag{9}\labelp{Eq:inlawb7}
\r\yz\b\mo\otimes\r\xy\a\mo=\tbigcup\{\r\zx\gm:\hs\zx\gm\seq
\vphi\zy\mo[\hs\yz\b\mo\scir\ks\xy\a\mo]\scir\cc z y x\}\per
\end{equation*}
Apply the isomorphism $\vphi\zx$ to both sides of the inclusion
\begin{align*} \hs\zx\gm&\seq
\vphi\zy\mo[\hs\yz\b\mo\scir\ks\xy\a\mo]\scir\cc z y
x\tag{10}\labelp{Eq:inlawb8}\\ \intertext{to obtain the equivalent
inclusion} \vphi\zx[\hs\zx\gm]&\seq
\vphi\zx[\vphi\zy\mo[\hs\yz\b\mo\scir\ks\xy\a\mo]\scir\cc z y
x]\per\tag{11}\labelp{Eq:inlawb9}
 \end{align*}
Use the definition of the coset $\ks\zx\gm$ as the image of the
coset $\hs\zx\gm$ under the isomorphism $\vphi\zx$, and then use
\refCo{con11}, to rewrite the  left side  of \refEq{inlawb9}   as
\begin{equation*} \vphi\zx[\hs\zx\gm]=\ks\zx\gm=\hs\xz\gm\per \tag{12}\labelp{Eq:inlawb12.0}
\end{equation*}
The right  side of \refEq{inlawb9} may also be rewritten in the
following way:
\begin{align*}
\vphi\zx[\vphi\zy\mo[\hs\yz\b\mo\scir\ks\xy\a\mo]\scir\cc
z y
x]&=\vphi\zx[\vphi\zy\mo[\hs\yz\b\mo\scir\ks\xy\a\mo]\myspace]\scir\vphi\zx[\cc
z y x]\tag{13}\labelp{Eq:inlawb13.0}\\
&=\vphi\xz\mo[\vphi\yz[\hs\yz\b\mo\scir\ks\xy\a\mo]\myspace]\scir\vphi\zx[\cc
z y x]\\ &=\cc x y
z\mo\scir\vphi\xy\mo[\hs\yz\b\mo\scir\ks\xy\a\mo]\scir\cc x y
z\scir\vphi\zx[\cc z y x]\per
\end{align*} The first equality uses the isomorphism property of
$\vphi\zx$\comma the second uses semi-frame condition 
(ii) which says that
\[\vphi\zx=\vphi\xz\mo\qquad\text{and}\qquad\vphi\yz=\vphi\zy\mo\comma\]
and the third equality uses \refL{sfimage}(ii) (with
$\hs\yz\b\mo\scir\ks\xy\a\mo$ in place of $Q$). Combine
\refEq{inlawb12.0} with \refEq{inlawb13.0} to conclude that the
inclusion in \refEq{inlawb9},  and consequently also the one in
\refEq{inlawb8}, is equivalent to the inclusion
\begin{equation*}\tag{14}\labelp{Eq:inlawb10.0}
\hs\xz\gm\seq\cc x y
z\mo\scir\vphi\xy\mo[\hs\yz\b\mo\scir\ks\xy\a\mo]\scir\cc x y
z\scir\vphi\zx[\cc z y x]\per\end{equation*} Use the equivalence
between \refEq{inlawb8} and \refEq{inlawb10.0} to rewrite
\refEq{inlawb7}  as
\begin{multline*}\tag{15}\labelp{Eq:inlawb10}
\r\yz\b\mo\otimes\r\xy\a\mo=\\ \tbigcup\{\r\zx\gm:\hs\xz\gm\seq\cc x
y z\mo\scir\vphi\xy\mo[\hs\yz\b\mo\scir\ks\xy\a\mo]\scir\cc x y
z\scir\vphi\zx[\cc z y x]\}\per \end{multline*}

It follows from \refEq{inlawb4} and \refEq{inlawb10} that the
equation in (i) holds just in case the right  side of
\refEq{inlawb4} is equal to the right  side of \refEq{inlawb10}. The
right  sides of \refEq{inlawb4} and \refEq{inlawb10} are equal just
in case the cosets
\begin{gather*}\cc x y
z\mo\scir\vphi\xy\mo[\hs\yz\b\mo\scir\ks\xy\a\mo]\tag{16}\labelp{Eq:inlawb12}\\
\intertext{and} \cc x y
z\mo\scir\vphi\xy\mo[\hs\yz\b\mo\scir\ks\xy\a\mo]\scir\cc x y
z\scir\vphi\zx[\cc z y x]\tag{17}\labelp{Eq:inlawb13}
\end{gather*}
are equal, by \refL{auxlemma}. (Notice that \refEq{inlawb12} and \refEq{inlawb13}
really are cosets of $\hh$.  In more detail, each of  the factors in
\refEq{inlawb12} and \refEq{inlawb13} is a coset of $\hh$, so the
composition of these factors is also a coset of $\hh$\per For
example, $\hs\yz\b\mo\scir\ks\xy\a\mo$ is a coset of
$\k\xy\scir\h\yz$\comma and $\hvphs\xy$ maps the group $\G\wx/(\hh)$
isomorphically onto the group $\G\wy/(\kh)$, so the inverse image
$\vphi\xy\mo[\hs\yz\b\mo\scir\ks\xy\a\mo]$ must be a coset of
$\hh$\per  The isomorphism $\hvphs\zx$, which coincides with
$\hvphs\xz\mo$, maps the group $\G\wz/(\kk)$ isomorphically onto the
group $\G\wx/(\hh)$\comma and $\cc z y x$ is a coset of
$\h\zy\scir\h\zx=\kk$, so the image $\vphi\zx[\cc z y x]$ must be a
coset of $\hh$.)   The cosets in  \refEq{inlawb12} and
\refEq{inlawb13} are equal just in case
\begin{gather*} \hh=\cc x
y z\scir\vphi\zx[\cc z y x]\comma\\ \intertext{or, put another way,
they are equal just in case} \vphi\zx[\cc z y x]\mo=\cc x y
z\comma\tag{18}\labelp{Eq:inlawb14} \end{gather*} by the
cancellation law for groups. Rewrite \refEq{inlawb14} as
\begin{equation*}\tag{19}\labelp{Eq:inlawb19.0}
\vphi\zx[\cc z y x\mo]=\cc x y z\comma
\end{equation*} using the isomorphism properties of $\vphi\zx$\comma and then apply  the
inverse $\vphi\xz$ of the isomorphism $\vphi\zx$ to both sides of
\refEq{inlawb19.0} to obtain the equivalent equation
\begin{equation*}\tag{20}\labelp{Eq:inlawb15}
\cc z y x\mo=\vphi\xz[\cc x y z]\per \end{equation*}  Combine these
various equivalences to conclude that (i) holds if and only if
\refEq{inlawb15} holds, that is to say, if and only if (iii) holds.

It has been shown that (i) and (iii) are equivalent for any fixed
$\a$ and $\b$.  Since (iii) does not involve $\a$ and $\b$, it may
be concluded that (iii) implies (i) for any $\a$ and $\b$, and hence
(iii) implies (ii).  The implication from (ii) to (i) is trivial.

The form of the second involution law as a positive, regular
equation implies that it holds in $\cra CF$ just in case it holds
for all atoms $\r\xy\a$ and $\r\wwz\b$ in $\cra CF$\per  If
$\wy=\ww$, then the law holds for the given pair of atoms just in
case $\vphi\xz[\cc x y z]=\cc z y x\mo$\comma by the equivalence of
conditions (ii) and (iii) established above.

Assume $\wy\neq\ww$. We show that the second involution law holds
automatically for the given pair of atoms. Indeed, choose $\gm$ and
$\dlt$ so that
\begin{align*}
\hs\xy\a\mo=\hs\xy\gm\qquad&\text{and}\qquad
\hs\wwz\b\mo=\hs\wwz\dlt\per\\ \intertext{Semi-frame condition (ii)
and  Converse \refT{convthm1} imply that}
\r\xy\a\mo=\r\yx\gm\qquad&\text{and}\qquad \r\wwz\b\mo=\r\zw\dlt\per
\end{align*}
Combine this with the definition of $\,\otimes\,$ under the
assumption that $y\neq w$ to obtain
\begin{gather*}
\r\wwz\b\mo\otimes\r\xy\a\mo
=\r\zw\dlt\otimes\r\yz\gm=\varnothing\tag{21}\labelp{Eq:inlawa2}\\
\intertext{and} (\r\xy\a\otimes\r\wwz\b)\mo=\varnothing\mo
=\varnothing\per\tag{22}\labelp{Eq:inlawa3}\end{gather*} Since the
right sides of \refEq{inlawa2} and \refEq{inlawa3} are equal, so are
the left sides.
\end{proof}

Turn finally to the task of characterizing when the associative
law for relative multiplication holds in an algebra $\cra CF$.
Again, it suffices to characterize when it holds for atoms. It is
helpful to introduce a bit of notation.
  Let $\ez 4$ denote  the set of quadruples $\quadr x y z w$ such that the pairs
$\pair x y$, $\pair x z$, and $\pair x w$ are all in $\mc E$, or,
equivalently, such that the triples $\trip x y z$ and $\trip x z w$
are in $\ez 3$\per

\begin{theorem}[Associative Law Theorem]\labelp{T:alawc}  Let $\mc F$ be a semi-frame\comma  and
$\quadr x y z w$  a quadruple in $\ez 4$\per  The following
conditions are equivalent\per
\begin{enumerate}
 \item[(i)] $(\r\xy\a\otimes\r\yz\b)\otimes\r\zw\g=
\r\xy\a\otimes(\r\yz\b\otimes\r\zw\g) $ for  some
$\a<\kai\xy$\comma $\b<\kai\yz$ and $\g<\kai\zw$\per
 \item[(ii)]  $(\r\xy\a\otimes\r\yz\b)\otimes\r\zw\g=
\r\xy\a\otimes(\r\yz\b\otimes\r\zw\g) $ for  all
$\a<\kai\xy$\comma $\b<\kai\yz$ and $\g<\kai\zw$\per
 \item[(iii)] $\cc x y z\scir\cc x z w=\vphi\yx[\cc y z w\scir\h\yx]\scir\cc x y w$\per
 \end{enumerate}
Consequently\comma the associative law for $\,\otimes\,$ holds in
the algebra $\cra C F$ just in case \textnormal{(iii)} holds for all quadruples
$\quadr x y z w$ in $\ez 4$\per
\end{theorem}

\begin{proof}
Fix some $\a<\kai\xy$\comma $\b<\kai\yz$\comma and $\g<\kai\zw$,
with goal of establishing the equivalence of (i) and (iii). The
first task is to compute and simplify an expression for
\begin{equation*}\tag{1}\labelp{Eq:alawb1}
(\r\xy\a\otimes\r\yz\b)\otimes\r\zw\g\per
\end{equation*}
The definition of $\,\otimes\,$ implies that
\begin{equation*}\tag{2}\labelp{Eq:alawb1.1}
 \r\xy\a\otimes\r\yz\b=
\tbigcup\{\r\xz\x:\hs\xz\x\seq\vphi\xy\mo[\ks\xy\a\scir\hs\yz\b]\scir\cc
x y z\}\per
 \end{equation*}
Form the product, in  the sense of $\,\otimes\,$, on both sides of
\refEq{alawb1.1} on the right with $\r\zw\g$\comma and use the
distributivity of $\,\otimes\,$ over arbitrary unions, to see that
\refEq{alawb1} is equal to the union
\begin{equation*}\tag{3}\labelp{Eq:alawb2}
\tbigcup\{\r\xz\x\otimes\r\zw\g:\hs\xz\x\seq\vphi\xy\mo[\ks\xy\a\scir\hs\yz\b]\scir\cc
x y z\}\per \end{equation*}  The definition  of $\,\otimes\,$ also
yields
\begin{equation*}\tag{4}\labelp{Eq:alawb3}
 \r\xz\x\otimes\r\zw\g=
\tbigcup\{\r\xw\rho:\hs\xw\rho\seq\vphi\xz\mo[\ks\xz\x\scir\hs\zw\g]\scir\cc
x z w\}
 \end{equation*}
for each $\x$. Write
\begin{equation*}\tag{5}\labelp{Eq:alawb4}
\dd= \vphi\xy\mo[\ks\xy\a\scir\hs\yz\b]\scir\cc x y
z\comma\end{equation*} and observe that $\dd$ is a coset of the
normal subgroup $\h\xy\scir\h\xz$ in $\G \wx$.  Combine
\refEq{alawb4} with \refEq{alawb2} and \refEq{alawb3} to arrive at
the equality of \refEq{alawb1} with
\begin{equation*}
\tbigcup\{\r\xw\rho:\hs\xw\rho\seq\vphi\xz\mo[\ks\xz\x\scir\hs\zw\g]\scir\cc
x z w\text{ for some }\hs\xz\x\seq\dd\}\per
 \end{equation*}
This union may  be rewritten as
\begin{equation*}\tag{6}\labelp{Eq:alawb5}
\tbigcup\bigl\{\r\xw\rho:\hs\xw\rho\seq\tbigcup\{\vphi\xz\mo[\ks\xz\x\scir\hs\zw\g]\scir\cc
x z w:\hs\xz\x\seq\dd\}\bigr\}\per
\end{equation*} In more detail, the sets
\[\vphi\xz\mo[\ks\xz\x\scir\hs\zw\g]\scir\cc
x z w\comma
\] for various $\xi$, are cosets of $\h\xz\scir\h\xw$ (since $\vphi\xz$
induces an isomorphism from $\G\wx/(\h\xz\scir\h\xw)$ to
$\G\wz/(\k\xz\scir\h\zw)$)\comma and any coset $\hs\xw\rho$ of
$\h\xw$ that is contained in a union of cosets of
$\h\xz\scir\h\xw$  must be contained entirely within one of these
cosets. It follows that \refEq{alawb1} and \refEq{alawb5} are
equal.

We now transform   \refEq{alawb5} in a series of steps.  First,
\begin{align*}
\tbigcup\{\ks\xz\x:\hs\xz\x\seq\dd\}
&=\tbigcup\{\vphi\xz[\hs\xz\x]:\hs\xz\x\seq\dd\}
\tag{7}\labelp{Eq:alawb6}\\
&=\vphi\xz[\,\tbigcup\{\hs\xz\x:\hs\xz\x\seq\dd\}]\\
&=\vphi\xz[\dd]\comma
\end{align*} by the
definition  of $\ks\xz\x$ as the image of $\hs\xz\x$ under the
mapping $\vphi\xz$\comma the distributivity of function images
over unions, and the fact that $\dd$ is the union of the set of
cosets of $\h\xz$ that are included in it, by \refEq{alawb4} and
the remark following \refEq{alawb4}. Therefore
\begin{multline*}\tbigcup\{\vphi\xz\mo[\ks\xz\x\scir\hs\zw\g]\scir\cc
x z w:\hs\xz\x\seq\dd\}\\
\begin{split}
&=\tbigcup\{\vphi\xz\mo[\ks\xz\x\scir\hs\zw\g]:\hs\xz\x\seq\dd\}\scir\cc
x z w\\ &=\vphi\xz\mo[\,\tbigcup\{
\ks\xz\x\scir\hs\zw\g:\hs\xz\x\seq\dd\}]\scir\cc x z w\\
&=\vphi\xz\mo[\,\tbigcup\{
\ks\xz\x\scir\k\xz\scir\hs\zw\g:\hs\xz\x\seq\dd\}]\scir\cc x z w\\
&=\vphi\xz\mo[\,\tbigcup\{
\ks\xz\x:\hs\xz\x\seq\dd\}\scir\k\xz\scir\hs\zw\g]\scir\cc x z w\\
&=\vphi\xz\mo[\vphi\xz[\dd]\scir\k\xz\scir\hs\zw\g]\scir\cc x z
w\\
&=\vphi\xz\mo[\vphi\xz[\dd]\myspace]\scir\vphi\xz\mo[\k\xz\scir\hs\zw\g]\scir\cc
x z w\\ &=\dd\scir\vphi\xz\mo[\k\xz\scir\hs\zw\g]\scir\cc x z w\\
&= \vphi\xy\mo[\ks\xy\a\scir\hs\yz\b]\scir\cc x y
z\scir\vphi\xz\mo[\k\xz\scir\hs\zw\g]\scir\cc x z w\comma
\end{split}\end{multline*} by the distributivity
of coset composition over arbitrary unions, the distributivity of
inverse function images over arbitrary unions, the fact that
$\k\xz$ is the identity element for its group of cosets, the
distributivity of coset composition over arbitrary unions,
\refEq{alawb6}, the isomorphism property of $\vphi\xz\mo$\comma
the fact that $\vphi\xz$ and $\vphi\xz\mo$ are inverses of one
another (by semi-frame condition (ii)), and the definition of
$D_1$ in \refEq{alawb4}\per

Recall that $\cc x y z$ is a coset of $\hh$. The latter is the
identity element of the quotient group $\G\wx/(\h\xy\scir\h\xz)$,
and also the  image of $\kk$ under the inverse isomorphism
$\vphi\xz\mo$\per Consequently,
\begin{align*} \cc x y z\scir\vphi\xz\mo[\k\xz\scir\hs\zw\g]&=\cc x y z\scir
\hh\scir\vphi\xz\mo[\k\xz\scir\hs\zw\g]\\ &=\cc x y z\scir
\vphi\xz\mo[\kk]\scir\vphi\xz\mo[\k\xz\scir\hs\zw\g]\\ &=\cc x y
z\scir\vphi\xz\mo[\kk \scir \k\xz\scir\hs\zw\g]\\ &=\cc x y
z\scir\vphi\xz\mo[\kk \scir \hs\zw\g]\\ &=\cc x y
z\scir\vphi\xz\mo[\kk \scir\k\yz\scir \hs\zw\g]\\
&=\vphi\xy\mo[\,\vphi\yz\mo[\kk \scir\k\yz\scir
\hs\zw\g]\myspace]\scir\cc x y z\\
&=\vphi\xy\mo[\,\vphi\yz\mo[\kk]
\scir\vphi\yz\mo[\k\yz\scir\hs\zw\g]\myspace]\scir\cc x y z\\
&=\vphi\xy\mo[\kh\scir\vphi\yz\mo[\k\yz\scir\hs\zw\g]\myspace]\scir\cc
x y z\per
\end{align*} The sixth equality uses \refL{sfimage}(iii) (with $\kk
\scir\k\yz\scir \hs\zw\g $ in place of $Q$), the seventh the
isomorphism property of $\vphi\yz\mo$\comma and the eighth the
fact that $\vphi\yz$ maps $\kh$ to $\kk$\per

Combine the last two strings of equalities with the isomorphism
properties of $\vphi\xy\mo$, and the fact that $\k\xy\scir\h\yz$
is the identity element of the quotient group
$\G\wy/(\k\xy\scir\h\yz)$,  to arrive at
\begin{multline*}
\tbigcup\{\vphi\xz\mo[\ks\xz\x\scir\hs\zw\g]\scir\cc x z
w:\hs\xz\x\seq\dd\}\\
\begin{split}
&= \vphi\xy\mo[\ks\xy\a\scir\hs\yz\b]\scir\cc x y
z\scir\vphi\xz\mo[\k\xz\scir\hs\zw\g]\scir\cc x z w\\
&=\vphi\xy\mo[\ks\xy\a\scir\hs\yz\b]\scir\vphi\xy\mo[\kh\scir\vphi\yz\mo[\k\yz\scir\hs\zw\g]\myspace]\scir\cc
x y z\scir\cc x z w\\ &=\vphi\xy\mo[\ks\xy\a\scir\hs\yz\b
\scir\kh\scir\vphi\yz\mo[\k\yz\scir\hs\zw\g]\myspace]\scir\cc x y
z\scir\cc x z w\\ &=\vphi\xy\mo[\ks\xy\a\scir\hs\yz\b
\scir\vphi\yz\mo[\k\yz\scir\hs\zw\g]\myspace]\scir\cc x y
z\scir\cc x z w\per\\
\end{split}\end{multline*}
Conclusion: \refEq{alawb5} may be rewritten as the inclusion
\begin{equation*}\tag{8}\labelp{Eq:alawb7}
\tbigcup\{\r\xw\rho:\hs\xw\rho\seq\vphi\xy\mo[\ks\xy\a\scir\hs\yz\b
\scir\vphi\yz\mo[\k\yz\scir\hs\zw\g]\myspace]\scir\cc x y
z\scir\cc x z w\}\comma
\end{equation*}
so \refEq{alawb1} and \refEq{alawb7} are equal.

The next task is to  work out  an analogous expression for
\begin{equation*}\tag{9}\labelp{Eq:alawb8}
\r\xy\a\otimes(\r\yz\b\otimes\r\zw\g)
\end{equation*}
in an analogous fashion\per  Write
\begin{equation*}\tag{10}\labelp{Eq:alawb9}
\ddd= \vphi\yz\mo[\ks\yz\b\scir\hs\zw\g]\scir\cc
 y z w\per \end{equation*}
The definition of $\,\otimes\,$ and \refEq{alawb9} imply that
\begin{equation*}\tag{11}\labelp{Eq:alawb10.1}
\r\yz\b\otimes\r\zw\g=\tbigcup\{\r\yw\x:\hs\yw\x\seq\ddd\}\per
\end{equation*} Form the $\otimes$ product, on both sides of this
equation on the left with $\r\xy\a$\comma and use the
distributivity of $\,\otimes\,$ over arbitrary unions, to see that
\refEq{alawb8} is equal to
\begin{equation*}\tag{12}\labelp{Eq:alawb12.1}
\tbigcup\{\r\xy\a\otimes\r\yw\x:\hs\yw\x\seq\ddd\}\per
\end{equation*} Since
\begin{equation*}
\r\xy\a\otimes\r\yw\x=\tbigcup\{\r\xw\rho:\hs\xw\rho\seq\vphi\xy\mo[\ks\xy\a\scir\hs\yw\x
]\scir\cc x y w\}\comma
\end{equation*} by the definition of $\,\otimes\,$,
it follows that \refEq{alawb12.1}, and hence also \refEq{alawb8},
is equal to
\begin{equation*}
\tbigcup\{\r\xw\rho:\hs\xw\rho\seq\vphi\xy\mo[\ks\xy\a\scir\hs\yw\x]\scir\cc
x y w\text{ for some }\hs\yw\x\seq\ddd\}\per
 \end{equation*}
This union can be rewritten as
\begin{equation*}
\tbigcup\bigl\{\r\xw\rho:\hs\xw\rho\seq
\tbigcup\{\vphi\xy\mo[\ks\xy\a\scir\hs\yw\x]\scir\cc x y
w:\hs\yw\x\seq\ddd\}\bigr\}\comma\end{equation*} and therefore
also as
\begin{equation*}\tag{13}\labelp{Eq:alawb10}
\tbigcup \{\r\xw\rho:\hs\xw\rho\seq
\vphi\xy\mo[\ks\xy\a\scir\ddd]\scir\cc x y w\}\per\end{equation*}
(This last step uses the distributivity of coset compositions and
of inverse function images over arbitrary unions.)  Use the
identity element property for $\k\yz$ with respect to  its cosets,
the    isomorphism properties of $\vphi\yz\mo$ on cosets and
unions of cosets of $\k\yz$\comma and   the definition of
$\ks\yz\b$ to write
\begin{align*}
\vphi\yz\mo[\ks\yz\b\scir\hs\zw\g]&=
\vphi\yz\mo[\ks\yz\b\scir\k\yz\scir\hs\zw\g]\tag{14}\labelp{Eq:alawb11}\\
&=\vphi\yz\mo[\ks\yz\b]\scir\vphi\yz\mo[\k\yz\scir\hs\zw\g]\\
&=\hs\yz\b\scir\vphi\yz\mo[\k\yz\scir\hs\zw\g]\per
\end{align*}
 It  follows that
\begin{multline*}
\vphi\xy\mo[\ks\xy\a\scir\ddd]\scir\cc x y w\\
\begin{split}
&=\vphi\xy\mo[\ks\xy\a\scir\vphi\yz\mo[\ks\yz\b\scir\hs\zw\g]\scir\cc
 y z w]\scir\cc
x y w\\
&=\vphi\xy\mo[\ks\xy\a\scir\hs\yz\b\scir\vphi\yz\mo[\k\yz\scir\hs\zw\g]\scir\cc
 y z w]\scir\cc
x y w\\
&=\vphi\xy\mo[\ks\xy\a\scir\k\xy\scir\hs\yz\b\scir\vphi\yz\mo[\k\yz\scir\hs\zw\g]\scir\cc
 y z w]\scir\cc
x y w\\
&=\vphi\xy\mo[\ks\xy\a\scir\hs\yz\b\scir\vphi\yz\mo[\k\yz\scir\hs\zw\g]\scir\k\xy\scir\cc
 y z w]\scir\cc
x y w\\
&=\vphi\xy\mo[\ks\xy\a\scir\hs\yz\b\scir\vphi\yz\mo[\k\yz\scir\hs\zw\g]\myspace]\scir
\vphi\xy\mo[\k\xy\scir\cc
 y z w]\scir\cc
x y w\comma
\end{split}\end{multline*}
by \refEq{alawb9}, \refEq{alawb11}, the identity element
properties of $\k\xy$ with respect to its cosets, the fact that
$\k\xy$ is a normal subgroup of $\G y$ and therefore commutes with
the other sets, and the isomorphism properties of $\vphi\xy\mo$.
In this regard, observe that the complex product $\k\xy\scir\cc
 y z w$ is a union of cosets of $\k\xy$ (this was the point of introducing $\k\xy$ into  the
fourth expression), and of course so is
\[\ks\xy\a\scir\hs\yz\b\scir\vphi\yz\mo[\k\yz\scir\hs\zw\g]\]
(since the coset $\ks\xy\a$ is present in the complex product), so
the isomorphism property of $\vphi\xy\mo$  for unions of cosets of
$\k\xy$ really is applicable.

This last string  of equalities shows that \refEq{alawb10} may be
rewritten in the form
\begin{multline*}\tag{15}\labelp{Eq:alawb12}
\tbigcup \{\r\xw\rho:\hs\xw\rho\seq\\
\vphi\xy\mo[\ks\xy\a\scir\hs\yz\b\scir\vphi\yz\mo[\k\yz\scir\hs\zw\g]\myspace]\scir
\vphi\xy\mo[\k\xy\scir\cc
 y z w]\scir\cc
x y w\}\comma\end{multline*} so \refEq{alawb8} is equal to
\refEq{alawb12}.

It has been shown that \refEq{alawb1} is equal to  \refEq{alawb7},
and \refEq{alawb8} to  \refEq{alawb12}.  It follows that
\refEq{alawb1} and  \refEq{alawb8} will be equal, that is to say,
condition (i) of the theorem will hold,
 just in case \refEq{alawb7} and  \refEq{alawb12} are equal.
According to \refL{auxlemma}, the unions \refEq{alawb7} and
\refEq{alawb12} are equal just in case the corresponding cosets
\begin{gather*}\vphi\xy\mo[\ks\xy\a\scir\hs\yz\b
\scir\vphi\yz\mo[\k\yz\scir\hs\zw\g]\myspace]\scir\cc x y
z\scir\cc x z w \intertext{and}
\vphi\xy\mo[\ks\xy\a\scir\hs\yz\b\scir\vphi\yz\mo[\k\yz\scir\hs\zw\g]\myspace]\scir
\vphi\xy\mo[\k\xy\scir\cc
 y z w]\scir\cc x y w
\end{gather*}
are equal.  Apply the cancellation law  for the quotient group $\G
\wx/(\hh\scir\h\xw)$ to conclude that these two  cosets are equal
if and only if
\begin{equation*}\tag{16}\labelp{Eq:alawb14}
\cc x y z\scir\cc x z w=\vphi\xy\mo[\k\xy\scir\cc
 y z w]\scir\cc
x y w\per
\end{equation*}

To justify this application of the cancellation law, it must be
shown that the relevant factors, namely
\begin{equation*}\tag{17}\labelp{Eq:alawb15}
\vphi\xy\mo[\ks\xy\a\scir\hs\yz\b\scir
\vphi\yz\mo[\k\yz\scir\hs\zw\g]\myspace]\comma
\end{equation*}
\begin{equation*}\tag{18}\labelp{Eq:alawb16}
\cc x y z\scir\cc x z w\comma
\end{equation*}
and
\begin{equation*}\tag{19}\labelp{Eq:alawb17}
\vphi\xy\mo[\k\xy\scir\cc
 y z w]\scir\cc
x y w\comma
\end{equation*}
really are all cosets in $\G\wx$ of the normal subgroup
\begin{equation*}\tag{20}\labelp{Eq:alawb18}
\h\xy\scir\h\xz\scir\h\xw\per
\end{equation*}
Begin with \refEq{alawb15}.  Observe that $\k\yz\scir\hs\zw\g$ is
a coset  of $\k\yz\scir\h\zw$\comma so its  inverse image under
$\vphi\yz$ is a  coset  of $\h\yz\scir\h\yw$\per The complex
product $\ks\xy\a\scir\hs\yz\b $ is a coset of $\kh$\comma so the
product
\begin{equation*}\tag{21}\labelp{Eq:alawb21.1}
\ks\xy\a\scir\hs\yz\b\scir\vphi\yz\mo[\k\yz\scir\hs\zw\g]
\end{equation*}
is a coset of the group $\k\xy\scir\h\yz\scir\h\yz\scir\h\yw$,
which coincides with the group
\begin{equation*}\tag{22}\labelp{Eq:alawb21.2}
\k\xy\scir\h\yz\scir\h\yw\per
\end{equation*}
  Applying
$\vphi\xy\mo$ to \refEq{alawb21.1} gives \refEq{alawb15}. Applying
it to \refEq{alawb21.2} gives
\begin{equation*}
\vphi\xy\mo[\k\xy\scir\h\yz\scir\h\yw]\per
\end{equation*}
Since
\begin{align*}
 \vphi\xy\mo[\k\xy\scir\h\yz\scir\h\yw]&=
\vphi\xy\mo[\k\xy\scir\k\xy\scir\h\yz\scir\h\yw]\\
&=\vphi\xy\mo[\k\xy\scir\h\yz\scir\k\xy\scir\h\yw]\\
&=\vphi\xy\mo[\k\xy\scir\h\yz]\scir\vphi\xy\mo[\k\xy\scir\h\yw]\\
&=(\hh)\scir(\h\xy\scir\h\xw)\\
&=\h\xy\scir\h\xy\scir\h\xz\scir\h\xw\\
&=\h\xy\scir\h\xz\scir\h\xw\comma
\end{align*} and since \refEq{alawb21.1} is a coset of
\refEq{alawb21.2}, it may be concluded that \refEq{alawb15} is a
coset of \refEq{alawb18}, as claimed.

Turn now to \refEq{alawb16}.  By assumption, $\cc x y z$ is a
coset of the subgroup $\hh$\comma and $\cc x z w$ is a coset of
the subgroup $\h\xy\scir\h\xw$, so the product coset
\refEq{alawb16} is a coset of the product subgroup, which is
\refEq{alawb18}.

Consider, finally, \refEq{alawb17}.  By assumption, $\cc y z w$ is
a coset of $\h\yz\scir\h\yw$, so the product $\k\xy\scir\cc y z w$
is a coset of $\k\xy\scir\h\yz\scir\h\yw$\per  It follows that the
inverse image
\begin{equation*}\tag{23}\labelp{Eq:alawb19}
\vphi\xy\mo[\k\xy\scir\cc y z w]
\end{equation*}
 is a coset of the inverse image
$\vphi\xy\mo[\k\xy\scir\h\yz\scir\h\yw] $\per It was shown above
that this inverse image coincides with \refEq{alawb18}, so
\refEq{alawb19} is a coset of \refEq{alawb18}.   The set $\cc x w
y$ is a coset of $\h\xw\scir\h\xy$\comma by assumption, so the
product of $\cc x w y$ with \refEq{alawb19} is a coset of the
product of $\h\xw\scir\h\xy$ with  \refEq{alawb18}. This  last
product reduces to \refEq{alawb18}, so \refEq{alawb17} is a coset
of \refEq{alawb18}.

We carry out one final transformation of \refEq{alawb14}.
Semi-frame condition (ii) says that $\vphi\yx$ is the inverse of
$\vphi\xy$\comma and consequently   $\k\xy$ coincides with the
subgroup $\h\yx$\comma by \refCo{con11}. Also, the subgroup
$\h\yx$ is normal. Consequently, equation \refEq{alawb14}  may be
rewritten in the form
\begin{equation*}
\cc x y z\scir\cc x z w=\vphi\yx[\cc y z w\scir\h\yx]\scir\cc x y
w\comma
\end{equation*} which is just the equation in condition (iii).

It has been demonstrated that condition (i) holds for the fixed
$\a$, $\b$, and $\g$ just in case the equation in condition (iii)
holds. Since the formulation of (iii) does not involve any of the
three given indices, it follows that (iii) implies (i) for each
such triple of indices, and therefore (iii) implies (ii).  The
implication from (ii) to (i) is obvious.

The associative law holds in $\cra CF$ just in case it holds for all
atoms.  Consider a triple of atoms \[\r\xy\a\comma\qquad
\r\wwz\b\comma\qquad \r\uv\g\per\]    If $\wy=\ww$ and  $\wz= u$,
then the law holds for the triple of atoms just in case
\[\cc x y z\scir\cc x z w=\vphi\yx[\cc y z w\scir\h\yx]\scir\cc x y w\comma\] by
the equivalence of conditions (ii) and (iii) in the first part of
the theorem.

If $\wy\neq\ww$ or if $\wz\neq u$, then the associative law holds
automatically for this triple, since both sides reduce to the empty
relation. Indeed, if $\wy\neq\ww$\comma then
\[\r\xy\a\otimes\r\wwz\b=\varnothing\comma\]
by the definition of $\,\otimes\,$, and consequently
\begin{align*}
(\r\xy\a\otimes\r\wwz\b)\otimes\r\uv\g&=\varnothing\comma\tag{24}\labelp{Eq:alawa1}\\
\intertext{again, by the definition of $\,\otimes\,$. If also
$\wz\neq u$, then a similar argument shows that}
\r\xy\a\otimes(\r\wwz\b\otimes\r\uv\g)&=\varnothing\per\tag{25}\labelp{Eq:alawa2}
 \end{align*}
In this case, associativity holds by \refEq{alawa1} and
\refEq{alawa2}.

If $\wz=u$, then the argument is  slightly  more involved.  In this
case,
\begin{equation*}
\r\wwz\b\otimes\r\uv\g
=\tbigcup\{\r\wwv\x:\hs\wwv\x\seq\vphi\wwz\mo[\ks\wwz\b\scir\hs\uv\g]\scir\cc
w z v\}\comma\tag{26}\labelp{Eq:alawa3}
\end{equation*}
by the definition of $\,\otimes\,$, and therefore
\begin{multline*}
\r\xy\a\otimes(\r\wwz\b\otimes\r\uv\g)=\\
\tbigcup\{\r\xy\a\otimes\r\wwv\x: \hs\wwv\x\seq
\vphi\wwz\mo[\ks\wwz\b\scir\hs\uv\g]\scir\cc w z v\}\comma
\end{multline*}
by \refEq{alawa3} and the  distributivity of the operation
$\,\otimes\,$ over arbitrary unions. Each of the relations
$\r\xy\a\otimes\r\wwv\x$ in this union is empty, by the definition
of $\,\otimes\,$, since we have assumed that $\wy\neq\ww$.  It
follows that \refEq{alawa2} holds in this case as well. Compare
\refEq{alawa2} with \refEq{alawa1} to arrive at the desired
conclusion for the case  $\wy\neq w$. The case  $\wz\neq\wu$ is
treated in an analogous fashion.
\end{proof}

The next corollary says that semi-frame condition (iv) is necessary
for $\cra CF$ to be a relation algebra.

\begin{cor}[Semi-frame Corollary]\labelp{C:semiframe}
Assume that $\mc F$ is a pre-semi-frame. If either the Second
Involution Law or the Associative Law holds in the algebra $\cra
CF$, then $\mc F$ is a semi-frame.
\end{cor}

\begin{proof}
Assume that the Second Involution Law holds in $\cra CF$. Semi-frame
condition (iv) was used only once in the proof of \refT{invollawb},
when  \refL{sfimage}(ii) was applied to justify the third equality
in \refEq{inlawb13.0}. Omitting that step,  the proof shows that
\refT{invollawb}(i) holds just in case the cosets in
\refEq{inlawb12} and the modified \refEq{inlawb13} of that proof are
equal, that is to say, just in case
\begin{equation*}\tag{1}\labelp{Eq:semiframe1}
\cc xyz\mo\scir\vphi\xy\mo[Q] =
\vphi\xz\mo[\vphi\yz[Q]]\scir\vphi\zx[\cc zyx]\comma
\end{equation*}
where $Q$ is $\hs\yz\b\mo\scir\ks\xy\a\mo$. From the assumption that
the Second Involution Law holds, it follows that \refEq{semiframe1}
holds for all $\a,\b$, that is to say, for all cosets $Q$ of
$\h\yz\scir\k\xy$. Take $Q=\h\yz\scir\k\xy$ and use semi-frame
condition (iii) to obtain $\vphi\xy\mo[Q] =
\vphi\xz\mo[\vphi\yz[Q]]$. Substitute the left side of this equality
for the right side in \refEq{semiframe1}, and use the cancellation
law for groups (and the fact that for this choice of $Q$, the
inverse image $\vphi\xy\mo(Q)$ is a normal subgroup of $\G x$, and
hence commutes with $\vphi\zx(\cc zyx))$ to reduce
\refEq{semiframe1}   to
\begin{equation*}\tag{2}\labelp{Eq:semiframe2}
\cc xyz\mo = \vphi\zx[\cc zyx]\per%
\end{equation*}
Substitute the left side of \refEq{semiframe2} for the right side in
equation \refEq{semiframe1}, and then multiply both sides of the
resulting equation by $\cc xyz$ on the right to arrive at
\refL{sfimage}(ii), which is equivalent to \refL{sfimage}(i). Thus,
\refL{sfimage}(i)   holds for all triples $\trip xyz$ in $\ez 3$,
which is just what semi-frame condition (iv) expresses.

Assume now that the Associative Law holds in $\cra CF$. The
derivation of semi-frame condition (iv) is similar to the preceding
one. Semi-frame condition (iv) was used only once in the proof of
\refT{alawc}, when  \refL{sfimage}(iii) was applied to justify the
sixth equality in the transformation of the expression $\cc
xyz\scir\vphi\xz\mo[\k\xz\scir\hs\zw\g]$. If we use the
``half-transformed'' expression that we get without using
\refL{sfimage}(iii), in place of the one in step \refEq{alawb7} of
that proof, we get the   term
\begin{equation*}\tag{3}\labelp{Eq:semiframe3}
\vphi\xy\mo[\ks\xy\a\scir\hs\yz\b]\scir\cc
xyz\scir\vphi\xz\mo[\k\xz\scir\k\yz\scir\hs\zw\g]\scir\cc xzw\per
\end{equation*}
\refT{alawc}(i) is equivalent to the equality of \refEq{semiframe3}
and the term in \refEq{alawb12} of that proof, that is to say, to the term
\begin{equation*}\tag{4}\labelp{Eq:semiframe4}
\vphi\xy\mo[\ks\xy\a\scir\hs\yz\b\scir
\vphi\yz\mo[\k\yz\scir\hs\zw\g]]\scir\vphi\xy\mo[\k\xy\scir\cc
yzw]\scir\cc xyw\per
\end{equation*}
Multiply  the two terms on the left by
$\vphi\xy[\ks\xy\a\scir\hs\yz\b]$, use isomorphism property of
$\vphi\xy\mo$, and write $Q$ in place of $\k\yz\scir\hs\zw\g$ to get
that \refT{alawc}(i) is equivalent to the equation
\begin{equation*}\tag{5}\labelp{Eq:semiframe5}
\cc xyz\scir\vphi\xz\mo[\k\xz\scir Q]\scir\cc xzw =
\vphi\xy\mo[\vphi\yz\mo[Q]]\scir\vphi\xy\mo[\k\xy\scir\cc
yzw]\scir\cc xyw\per
\end{equation*}
The assumption that the associative law holds implies that \refEq{semiframe5} holds for all cosets $Q$
of $\k\yz\scir\h\zw$. In particular, it holds for $\k\yz\scir\h\zw$,
from which it follows that
\begin{equation*}\tag{6}\labelp{Eq:semiframe6}
\cc xyz\scir\cc xzw = \vphi\xy\mo[\k\xy\scir\cc yzw]\scir\cc xyw\per
\end{equation*}
Substitute the left side of \refEq{semiframe6} for the right side in
its occurrence on the right side of \refEq{semiframe5}, and then
cancel the occurrence of $\cc xzw$ on the right  of both sides of
the resulting equation, to get \refL{sfimage}(iii). The desired
conclusion now follows just as in the  previous paragraph.
\end{proof}

Coset relation algebras are generalizations of group relation
algebras, since each group relation algebra may be viewed as a coset
relation algebra.  In more detail, let $\mc F=\pair G\vph$ be a
group frame, and put $\bar{\mc F} = ( G, \vph, C)$ where $\cc x y z
= \h\xy\scir \h\xz$ for each triple $\trip xyz$ in $\ez 3$. It is
easy to see that the algebras $\f G[\mc F]$ and $\f C[\bar{\mc F}]$
are equal. In \refS{sec5}, it will be shown the class of coset
relation algebras is a proper extension of the class of group
relation algebras: there exist coset relation algebras that are not
group relation algebras.

We conclude the present section with two lemmas that concern the
relationship between these two constructions. The first lemma
characterizes when the operation $\,\otimes\,$ gives the same result
as relational composition.

\begin{lm}\labelp{L:otimes.1} Let  $\mc F$
be a semi-frame\per The following conditions are equivalent for all
triples    $\trip x y z$  in $\ez 3$\per
\begin{enumerate}\item[(i)] $\r\xy \lph\otimes\r\yz \bt=\r\xy \lph\rp\r\yz
\bt $ for  some $\lph<\kai\xy$ and some $\bt<\kai\yz$\per
\item[(ii)] $\r\xy \lph\otimes\r\yz \bt=\r\xy \lph\rp\r\yz
\bt $ for  all $\lph<\kai\xy$ and all $\bt<\kai\yz$\per
\item[(iii)] $\cc x y z =\hh$\per
\end{enumerate}
\end{lm}

\begin{proof} Assume first that  condition (iii) holds, with the goal of
establishing (ii). Clearly,
\begin{equation*}\tag{1}\labelp{Eq:otimes2}
\vphi\xy\mo[\ks\xy\lph\scir\hs\yz\bt]\scir\cc x y
z=\vphi\xy\mo[\ks\xy\lph\scir\hs\yz\bt]\comma\end{equation*} because
$\hh$ is the identity element in its group of  cosets. For the same
reason, the inner automorphism $\tau$ of $\G\wx/(\hh)$ determined by
the coset $\cc x y z$ is  the identity automorphism. Semi-frame
condition (iv) therefore reduces to
\begin{equation*}\tag{2}\label{Eq:otimes2.1}
\hvphs\xy\rp\hvphs\yz=\tau\rp\hvphs\xz =\hvphs\xz\per
\end{equation*}
Use \refEq{otimes2.1} and the implication from (iv) to (iii) in the
Composition Theorem to obtain
\begin{equation*}\tag{3}\labelp{Eq:otimes1}\r \xy \a \rp \r
\yz\b=\tbigcup\{\r \xz \g: \hs \xz \g \seq \vphi \xy\mo[ \ks\xy
\a\scir\hs \yz \b]\}
\end{equation*}
for all $\lph<\kai\xy$ and   $\bt<\kai\yz$\per  (The first
hypothesis in condition (iv) is satisfied because of semi-frame
condition (iii).) Use \refD{smult} to get
\begin{equation*}\tag{4}\labelp{Eq:otimes0}\r \xy \a \otimes \r
\yz\b=\tbigcup\{\r \xz \g: \hs \xz \g \seq \vphi \xy\mo[ \ks\xy
\a\scir\hs \yz \b]\scir\cc x y z\} \end{equation*} for all
$\a<\kai\xy$ and all $\b<\kai\yz$\per
 Combine \refEq{otimes1}, \refEq{otimes0}, and
\refEq{otimes2} to arrive at (ii).

The implication from (ii) to (i) is obvious. To establish the
implication from (i) to (iii),  let $\a<\kai\xy$ and $\b<\kai\yz$ be
fixed indices such that (i) holds. Since the universe $A$ of the
algebra $\cra CF$ is closed under the operation $\,\otimes\,$, the
composition $\r \xy \a \rp \r \yz\b$ must belong to $A$. Apply
\refC{compequiv}   to see that this composition must belong to $A$
for every choice of $\a<\kai\xy$ and $\b<\kai\yz$. Invoke the
Composition Theorem to obtain \refEq{otimes1}. Use \refD{smult} to
get \refEq{otimes0}. Combine \refEq{otimes1} and \refEq{otimes0}
with the assumption in (i) to arrive at
\begin{multline*}  \tbigcup\{\r \xz \g:
\hs \xz \g \seq \vphi \xy\mo[ \ks\xy\a \scir\hs \yz\b ]\}\\
=\tbigcup\{\r \xz \g: \hs \xz \g \seq \vphi \xy\mo[ \ks\xy\a
\scir\hs \yz\b ]\scir\cc x y z\}\tag{5}\labelp{Eq:otimes4}
\end{multline*} for the  $\a $ and $\b$ chosen so  that (i) holds.
Apply \refL{auxlemma} to \refEq{otimes4} to obtain \refEq{otimes2}.
(To check  that \refL{auxlemma} really is applicable, observe that
the inverse image $\vphi\xy\mo[\ks\xy\lph\scir\hs\yz\bt]$ of the
coset $\ks\xy\lph\scir\hs\yz\bt$   of $\kh$   is a coset of $\hh$,
because $\hvphs\xy$ maps $\G\wx/(\hh)$ isomorphically to
$\G\wy/(\kh)$\per Also, $\cc xyz$ is a coset of $\hh$, by
assumption. Consequently, the composition
\[\vphi \xy\mo[ \ks\xy\a \scir\hs \yz\b ]\scir\cc x y z\] is a
coset of $\hh$.) The only element of a (quotient) group that leaves
another element of the group unchanged under group composition is
the identity element, by the cancellation law for groups.
Consequently, it follows from \refEq{otimes2} that $\cc x y z$ must
coincide with the identity element $\hh$ of the quotient group\per
\end{proof}

In general,  the composition $\r\xy \a\rp\r\yz \b$ does not belong
to the algebra $\cra CF$.  Fortunately, it is possible  to
characterize when it does belong.

\begin{lm}\labelp{L:compclosed} Let $\mc F$ be a
semi-frame\per The following conditions are equivalent for all
triples $\trip xyz$ in   $\ez 3$\per
\begin{enumerate}\item[(i)] $\r\xy \a\rp\r\yz \b$ is in
$\cra CF$ for some $\a<\kai\xy$ and some $\b<\kai\yz$\per
\item[(ii)]  $\r\xy \a\rp\r\yz \b$ is in
$\cra CF$ for all $\a<\kai\xy$ and all $\b<\kai\yz$\per
\item[(iii)] $\cc x y z$ is in the center of the  group
$\G\wx/(\hh)$\per
\end{enumerate}
\end{lm}

\begin{proof}
The equivalence of (i) and (ii) is proved in \refC{compequiv}.  To
establish the implication from (iii) to (ii), assume that $\cc x y
z$ is in the center of $\G\wx/(\hh)$\per The inner automorphism
$\tau$ determined by $\cc x y z$ is then the identity automorphism,
so    semi-frame condition (iv) for the given triple $\trip x y z$
reduces to
\begin{equation*}\tag{1}\labelp{Eq:compclosed1}
\hvphs\xy\rp\hvphs\yz=\hvphs\xz\per
\end{equation*}
Keeping in mind semi-frame condition (iii), we see that the
conditions in part (iv) of the  Composition Theorem are satisfied
for the triple $\trip x y z$.  By the implication from (iv) to (ii)
in that theorem, the composition $\r\xy\a\rp\r\yz\b$ must be in the
universe $A$ of the algebra $\cra CF$ for  all $\a$ and $\b$.

 To establish  the implication from (ii) to (iii), assume that $\r\xy \a\rp\r\yz \b$ is
in $A$ for all $\a$ and   $\b$.  It follows from the Composition
Theorem that \refEq{compclosed1} holds. By assumption, $\mc F$ is a
semi-frame, so
\begin{equation*}\tag{2}\labelp{Eq:compclosed2}
\hvphs\xy\rp\hvphs\yz=\tau\rp\hvphs\xz\comma
\end{equation*}
with $\tau$ denoting $\tau_{xyz}$. Comparing \refEq{compclosed1} and
\refEq{compclosed2}, it is clear that
\[\hvphs\xz=\tau\rp\hvphs\xz\per
\]
Form the relational composition of  each side of this equation with
$\hvphs\xz\mo$ on the  right to  see that $\tau$ is the identity
automorphism of the quotient group $\G\wx/(\hh)$\per This can only
happen if $\cc x y z$ is in the center of the quotient group,
because $\tau $ is the inner automorphism determined by $\cc x y
z$\per
\end{proof}

\section{Coset Semi-frames}\labelp{S:sec4}

In the preceding section, necessary and sufficient conditions are
given for the algebra $\cra CF$ constructed from a coset
semi-frame $\mc F$ to satisfy the identity law,  the second
involution law, the cycle law, and the associative law, and hence
to be a relation algebra. We single out the coset semi-frames that
satisfy these conditions.

\begin{df}\labelp{D:costra}A coset semi-frame
\[\mc F=(\langle
\G x:x\in I\,\rangle\smbcomma \langle\vph_{xy}:\pair x y\in \mc
E\,\rangle\smbcomma \langle \cc x y z:\trip x y z\in\ez
3\rangle)\] is said to \textit{satisfy the coset conditions} if
the following equations hold for all pairs $\pair xy$ in $\mc E$\comma all
triples $\trip xyz$ in $\ez 3$\comma and all quadruples $\quadr
xyzw$ in $\ez 4$ respectively\per
\begin{enumerate}
\item[(i)]
$\cc x y y =\h\xy$\per
\item[(ii)]
$\vphi\xz[\cc x y z]=\cc z y x\mo$\per
\item[(iii)]
$\vphi\xy[\cc x y z]=\cc y x z\mo$\per
\item[(iv)]
$\cc x y z\scir\cc x z w=\vphi\yx[\cc y z w\scir\h\yx]\scir\cc x y
w$\per
\end{enumerate} These
are called the \textit{coset conditions} for the identity law, the
second involution law, the cycle law, and the associative law
respectively.
\end{df}

The results in the previous section lead to the following theorem,
which is one of the main results of this paper.

\begin{theorem}[Coset Semi-frame Theorem]\labelp{T:costhm}
If  a coset semi-frame  $\mc F$  satisfies the coset
conditions\comma then the algebra $\cra CF$ constructed from $\mc
F$ is a complete and atomic  measurable relation algebra with base
set and unit \[U=\tbigcup  \{\G x:x\in I\}\qquad\text{and}\qquad
E=\tbigcup\{\cs Gx\times\cs Gy:\pair xy\in\mc E\}\]
respectively\per The atoms in this algebra are the relations of
the form $\r\xy \a$ for pairs $\pair xy$ in $\mc E$\comma and the
subidentity atoms are the relations of the form $\r\xx 0$ for elements $x$
in $I$\per The measure of $\r\xx 0$ is just the cardinality of the
group $\G\wx$\per
\end{theorem}
\begin{proof} The algebra $\cra CF$ is a complete and atomic
Boolean algebra of binary relations containing the identity
relation $\id U$, and closed under the set-theoretic operation of
converse and under the operation $\,\otimes\,$, by the definition
of a semi-frame, the assumption that $\mc F$ is a semi-frame, and
 Boolean Algebra \refT{disj}, Identity \refT{identthm1},
Converse \refT{convthm1}, and the definition of $\,\otimes\,$. The
Boolean axioms (R1)--(R3), the first involution law (R6), and the
two distributive laws (R8) and (R9) are valid in $\cra CF$, by
\refT{disj} and the remarks following \refD{semiframedef}\per The
associative law (R4), the identity law (R5), the second involution
law (R7), and the cycle law (R11) are also valid in $\cra CF$, by
Associative Law \refT{alawc}, Identity Law \refT{identthm2}, Second
Involution Law \refT{invollawb}, and Cycle Law \refT{clawc}
respectively, because $\mc F$ is assumed to satisfy the coset
conditions. Consequently, $\cra CF$ is a complete and atomic
relation algebra in which the universe consists of binary relations,
and all operations except the one for relative multiplication,
coincide with the standard set-theoretic operations of set relation
algebras.

The  atoms of the algebra $\cra CF$ are the relations of the form
$\r\xy\a$\comma and the subidentity atoms are the relations of the
form $\r\xx 0$\comma by  \refL{i-vi}, \refT{disj}, and the
construction of $\cra CF$. The identity relation $\id U$ is the
disjoint union of the subidentity atoms $\r\xx 0$, by
\refT{identthm1} and semi-frame condition (i).

 To prove that each subidentity atom $\r\xx 0$ is measurable, with
 measure the cardinality of the group $\G\wx$, it must
 be shown that the square
\begin{equation*}\tag{1}\labelp{Eq:costhm0}
\r\xx 0 \otimes E \otimes\r\xx 0
\end{equation*}
  is a union of $\kai\xx$ non-zero
 functional atoms.  The unit $E$ may be written in the form
\begin{equation*}\tag{2}\labelp{Eq:costhm1}
  E=\tbigcup\{\cs Gy\times\cs Gz:\pair
yz\in\mc E\}=\tbigcup\{\r\yz\a:\pair yz\in\mc E\text{ and
}\a<\kai\yz\}\comma
\end{equation*}
by \refL{i-vi} and \refT{disj}.  Consequently,
\begin{align*}\tag{3}\labelp{Eq:costhm2} \r\xx
0\otimes E \otimes \r\xx 0&= \r\xx 0\otimes
(\tbigcup\{\r\yz\a:\pair yz\in\mc E\text{ and }\a<\kai\yz\})
\otimes \r\xx 0\\ &= \tbigcup\{ \r\xx 0\otimes \r\yz\a \otimes
\r\xx 0:\pair yz\in\mc E\text{ and }\a<\kai\yz\}\comma
\end{align*} by \refEq{costhm1} and the distributivity of
$\,\otimes\,$ over arbitrary unions. If $x\neq y$ or $x\neq z$, then
\begin{equation*}\tag{4}\labelp{Eq:costhm3}
\r\xx 0\otimes\r\yz\a\otimes\r\xx 0=\varnot\comma
\end{equation*}
by the definition of the operation $\,\otimes\,$. On the other
hand, if $x=y$ and $x=z$, then
\begin{multline*}\tag{5}\labelp{Eq:costhm4}
\r\xx 0\otimes\r\yz\a\otimes\r\xx 0=\r\xx
0\otimes\r\xx\a\otimes\r\xx 0\\=\r\xx 0\rp\r\xx\a\rp\r\xx 0 =
\r\xx\a\per
\end{multline*}
The first equality uses the assumptions on $y$ and $z$. The second
equality  uses the assumption that $\mc F$ satisfies the coset
condition for the identity law, together with \refL{otimes.1} and
\refT{identthm2}, which ensures that condition (iii) of
\refL{otimes.1}, namely
\begin{equation*}\tag{6}\label{Eq:costhm6.1}
\cc x y z=\cc x x x=\h\xx=\h\xx\scir\h\xx=\h\xy\scir\h\xz\comma
\end{equation*}
 is satisfied.  The third equality   uses the fact
that $\r\xx 0=\id{\G x}$, and $\r\xx\lph$ is a subset of $\G x\times\G x$.
Combine \refEq{costhm2}--\refEq{costhm4}, and use
\refL{i-vi}, to arrive at
\begin{equation*}\tag{7}\labelp{Eq:costhm5}
\r\xx 0\otimes E \otimes\r\xx 0=\tbigcup\{\r\xx
\a:\a<\kai\xx\}=\G\wx\times\G\wx\per
\end{equation*}

Since $\h\xx=\k\xx=\{e_x\}$, the sets $\hs\xx\gm=\ks\xx\gm$ have the
form $\{\cs g\gm\}$, and therefore the relations $\r\xx\lph$ (for
$\lph<\kai\xx$) have the
form\begin{multline*}\tag{8}\labelp{Eq:costhm8}
\r\xx\lph=\tbigcup_\gm\{\hs\xx\gm\times(\ks\xx\gm\scir\ks\xx\lph)\}\\=\tbigcup_\gm\{\{\cs
g\gm\}\times\{\cs g\gm\scir\cs g\lph\}\}=\{\pair{\cs g\gm}{\cs
g\gm\scir\cs g\lph}:\gm<\kai\xx\}\comma
\end{multline*}  which is a function, and in fact a bijection.

It follows from \refEq{costhm5} and \refEq{costhm8} that the square \refEq{costhm0} is the disjoint union  of $\kai\xx$
functions.  Consequently, $\r\xx 0$ is
a measurable atom of measure $\kai\xx$. Combine this with the
observations of the previous paragraph to conclude that the relation
algebra $\cra CF$ is measurable.
\end{proof}

The theorem justifies the following definition.
\begin{df}\label{D:cradef}
  Suppose that $\mc F$ is a coset semi-frame that satisfies the
  coset conditions.  The relation algebra $\cra CF$ constructed from $\mc F$
  in Coset Semi-frame Theorem \textnormal{\ref{T:costhm}} is called
  the (\textit{full}) \textit{coset relation algebra}  on $\mc F$.  A \textit{general coset
  relation algebra} is defined to be an algebra that is embeddable
  into a full coset relation algebra.
\end{df}

The task of verifying  that a given group triple satisfies the
 semi-frame conditions and the coset conditions, and
therefore yields a full coset relation algebra, that is to say, it
yields an example of a measurable relation algebra, can be quite
complicated and tedious. Fortunately, some simplifications are
possible. To describe them, it is helpful to assume that the group
index set $I$ is linearly ordered, say by a relation $\,<\,$.
Roughly speaking, under the assumption of condition (i), condition
(ii) holds in general just in case it holds for each pair $\pair xy$
in $\mc E$ with $x<y$, and similarly for the other semi-frame
conditions.  Similar simplifications are possible for most of the
remaining semi-frame and coset conditions.  Actually, it is possible
to replace coset conditions (i)---(iii) with four simpler conditions
that do not simultaneously involve the formation of a coset inverse
and the application of a quotient isomorphism.

We begin with two lemmas. The first formulates some conditions that
are equivalent to coset condition  (ii) for the second involution
law and coset condition (iii) for the cycle law.

\begin{lm}\labelp{L:cysim} Let $\mc F$ be a semi-frame\comma  and  $\trip u v w  $ a triple
 in $\ez 3$\per Consider the following conditions on the coset
system of $\mc F$\per
\begin{enumerate}
\item[(i)]$\cc x y z\mo=\cc x z y$\quad for all permutations $\trip x y z$ of $\trip u v
w$\per
\item[(ii)] $\vphi \xz[\cc x y z]=\cc z y x\mo$\quad for all permutations $\trip x y z$ of
$\trip u v w$\per
\item[(iii)]$\vphi \xz[\cc x y z]=\cc z x y$\quad for all permutations $\trip x y z$ of $\trip u v
w$\per
\item[(iv)]$\vphi \xy[\cc x y z]=\cc y x z\mo$\quad for all permutations $\trip x y z$ of $\trip u v
w$\per
\item[(v)]$\vphi \xy[\cc x y z]=\cc y z x$\quad for all permutations $\trip x y z$ of $\trip u v
w$\per
\end{enumerate}
Conditions \textnormal{(iii)} and \textnormal{(v)} are
equivalent\per  Any two of conditions
\textnormal{(i)--(iv)}\comma\ and also any two of conditions
\textnormal{(i), (ii), (iv)} and \textnormal{(v),} imply all of
the other conditions\per
\end{lm}

\begin{proof} First, observe that
\begin{equation*}\tag{1}\labelp{Eq:cysim1}
\vphi\yz[\vphi\xy[\cc x y z]]=\vphi\xz[\cc x y z]
\end{equation*}
holds by semi-frame condition (iv), since\[\tau_{xyz}(\cc x y
z)=\cc x y z\mo\scir\cc x y z\scir\cc x y
z  = \cc x y z\per\] Apply $\vphi\zy$ to both sides of
\refEq{cysim1}, and use the fact that $\vphi\zy$ is the inverse of
$\vphi\yz$, by semi-frame condition (ii), to obtain
\begin{equation*}\tag{2}\labelp{Eq:cysim3}
\vphi\zy[\vphi\xz[\cc x y z]]=\vphi\xy[\cc x y z]\per
\end{equation*}
The equivalence of (iii) and (v) is now easy to prove.  If (iii)
holds, then
\begin{equation*}
\cc y z x =\vphi\zy[\cc z x y]=\vphi\zy[\vphi\xz[\cc x y
z]]=\vphi\xy[\cc x y z]\comma
\end{equation*} by (iii) (with $z$, $x$, and $y$ in place of $x$, $y$,
and $z$ respectively), another application of (iii), and
\refEq{cysim3}. On the other hand, if (v) holds, then
\begin{equation*}
\cc z x y  =\vphi\yz[\cc y z x] =\vphi\yz[\vphi\xy[\cc x y
z]]=\vphi\xz[\cc x y z]\comma
\end{equation*}
by (v) (with $y$, $z$, and $x$ in place of $x$, $y$, and $z$
respectively), another application of (v), and \refEq{cysim1}.

The next step is to show that conditions (i) and (ii) imply all of
the remaining conditions. The derivation of (iii) and (iv) from
(i) and (ii) is easy. For (iii), use (ii) and (i) (with $x$ and
$z$ interchanged):
\[\vphi\xz[\cc x y z]=\cc z y x\mo=\cc z x y\per
\] For (iv), first use (i), (ii) (with $y$ and $z$ interchanged),
and (i) (with $y$, $z$, and $x$ in place of
$x$, $y$, and $z$ respectively)  to get
\[\vphi\xy[\cc x y z\mo]=\vphi\xy[\cc x z y]=\cc y z x\mo =\cc y x
z\per
\] Form the coset inverses of the first and last terms, and use
the isomorphism properties of $\vphi\xy$\comma to arrive at (iv). It
has already been shown that (v) follows from (iii), so conditions
(i) and (ii) do imply all of the remaining conditions.

To show that conditions (i) and (iii) imply all of the remaining
conditions, it suffices to derive (ii), by the observations of the
preceding paragraph. Use (iii) and (i) (with $x$ and $z$
interchanged) to obtain
\begin{equation*}
\vphi\xz[\cc x y z]=\cc z x y =\cc z y x\mo\per
\end{equation*}
 Similarly, to show that conditions (i) and (iv) imply all of the
remaining conditions, it suffices to derive (ii). First, use (i),
(iv) (with $y$ and $z$ interchanged), and (i) (with $z$, $x$, and
$y$ in place of $x$, $y$, and $z$ respectively) to obtain
\begin{equation*}
\vphi\xz[\cc x y z\mo]=\vphi\xz[\cc x z y]=\cc z x y\mo=\cc z y
x\per
\end{equation*}  Form the coset inverses of
the first and last terms, and use the isomorphism properties of
$\vphi\xz$\comma to arrive (ii).

To prove that (ii) and (iii) imply all of the remaining
conditions, it suffices to derive (i). Use (iii) and (ii) to get
\[\cc z x y=\vphi\xz[\cc x y z]=\cc z y x\mo\per
\] Interchange $x$ and $z$ to arrive at (i).  Similarly, to prove that
(ii) and (iv) imply all of
the remaining conditions, it suffices to derive (i). Use (ii)
(with $x$ and $y$ interchanged) and the isomorphism properties of
$\vphi\xy$, (iv), \refEq{cysim1}, and (ii) to obtain
\begin{equation*}
\cc z x y =\vphi\yz[\cc y x z\mo]=\vphi\yz[\vphi\xy[\cc x y
z]]=\vphi\xz[\cc x y z]=\cc z y x\mo\per
\end{equation*} Again, interchange $x$ and $z$ to arrive at (i).

Finally, to show that (iii) and (iv) imply the remaining
conditions,  it suffices to derive (i). Use (iii) (with $x$ and
$y$ interchanged) and the isomorphism properties of $\vphi\yz$,
(iv), \refEq{cysim1}, and (iii) to obtain
\begin{equation*}
\cc z y x\mo  =\vphi\yz[\cc y x z\mo]=\vphi\yz[\vphi\xy[\cc x y
z]]=\vphi\xz[\cc x y z]=\cc z x y\per
\end{equation*} As before, interchange $x$ and $z$ to arrive at (i).
\end{proof}

The second lemma facilitates the verification of the second and
third  coset conditions in cases when some of the indices
coincide.
\begin{lm}\labelp{L:idcy} Let $\mc F$ be a semi-frame\per If
$\cc x y z=\hh
$
 for every permutation $\trip x y z$ of a given triple in $\ez 3$\co then
\[\cc x y z\mo=\cc x z y\qquad\text{and}\qquad
\vphi\xy[\cc x y z]=\cc y z x
\]
for every permutation of the given triple\po
\end{lm}
\begin{proof} Assume that
\begin{equation*}\tag{1}\labelp{Eq:idcy1}
\cc x y z=\hh
\end{equation*}
for all  permutations $\trip x y z$ of a given triple in $\ez 3$.
Obviously,
\begin{equation*}
\cc x y z\mo=(\hh)\mo=\h\xz\mo\scir\h\xy\mo=\h\xz\scir\h\xy=\cc x
z y
\end{equation*}
for all such permutations, by \refEq{idcy1},  the second involution law for cosets, the fact that
$\h\xz$ and $\h\xy$ are subgroups of $\G\wx$ and hence closed under
inverses, and \refEq{idcy1}  (with $y$ and $z$ interchanged).
Thus, the first equation in the conclusion  holds.

Semi-frame conditions (ii) and (iii), together with \refCo{con11}
and the fact that $\h\yx\scir\h\yz$ is a subgroup of $\G\wy$, imply
that
\begin{equation*}\tag{2}\labelp{Eq:idcy2}
\vphi\xy[\hh]=\kh=\h\yx\scir\h\yz=(\h\yx\scir\h\yz)\mo\per
\end{equation*}
Consequently,
\begin{equation*}
\vphi\xy[\cc x y z]=\vphi\xy[\hh]=(\h\yx\scir\h\yz)\mo=\cc y x
z\mo=\cc yzx\comma
\end{equation*}
by \refEq{idcy1}, \refEq{idcy2},  \refEq{idcy1} (with $x$ and $y$
interchanged), and the first conclusion of the lemma (with $x$ and
$y$ interchanged). Thus, the second equation in the conclusion
holds.
\end{proof}

The next theorem formulates a set of simplified semi-frame and coset
conditions.

\begin{theorem}\labelp{T:simcos} A group triple $\mc  F$ is a
coset semi-frame that satisfies the first three coset conditions
if and only if the following eight conditions are satisfied\per
\begin{enumerate}
\item[(i)]
$\vphi {xx}$ is the identity automorphism of $\G x/\{\e x\}$ for
every $\wx$ in $I$\per
\item [(ii)]
$\vphi \yx=\vphi \xy\mo$ for every pair $ \pair x y$ in $\mc E$
with $x<y$\per
\item[(iii)] $\vphi \xy[\h
\xy\scir\h \xz]=\k\xy\scir\h\yz$ and $\vphi \yz[\k \xy\scir\h
\yz]=\k\xz\scir\k\yz$  for every triple  $\trip x y z$  in $\ez 3$
with $x<y<z$\per
\item[(iv)] $\hvphs \xy\rp\hvphs \yz=\ttrip x y z\rp\hvphs \xz$ for every triple $\trip x y z$
in $\ez 3$ with $x<y<z$\per
\item[(v)]
$\cc x  x y=\cc x y x=\cc x y y=\h\xy$ for all pairs $\pair x y$
in $\mc E$\per
\item[(vi)] $\cc x y z\mo=\cc x z y$ for all triples $\trip x y z$ in $\ez 3$ with
$\wx$\co $\wy$\co $\wz$ mutually distinct\per
\item[(vii)] $\vphi\xy[\cc x y z]=\cc y z x$ for all triples $\trip x y z$ in $\ez 3$ with $x<y<z$\per
\item[(viii)]$\vphi\xz[\cc x y z]=\cc z x y$ for all triples $\trip x y z$  in  $\ez 3$ with $x<y<z$\per
\end{enumerate}
 If  $\mc F$ is a group triple that satisfies
 conditions \textnormal{(i)--(viii)}\comma then $\mc F$ satisfies the fourth coset condition
 if and only if
\begin{enumerate}
\item[(ix)]$\cc x y z\scir\cc x z w=\vphi\yx[\cc y z w\scir\h\yx]\scir\cc x y w$ for all quadruples
$\quadr x y z w$ in $\ez 4$ with $x<y<z<w$\per
\end{enumerate}
\end{theorem}
\begin{proof} Suppose that a group triple
$\mc F$ satisfies conditions (i)--(viii) of the theorem.  The
proof that semi-frame conditions (i)--(iii) must hold is easy, and
is in fact exactly the same as in the case of the corresponding
simplification of the group frame conditions for group pairs (see
Theorem 4.4 and its proof in \cite{giv1}). The details are
therefore   omitted. Turn  to the verification of semi-frame
condition (iv).

Consider a triple $\trip xyz$ in $\ez 3$, and assume first that
not all of the indices are distinct, say $\wx=\wy$.  The mapping
$\vphi\xy$ is  the identity automorphism of $\G x/\{e_x\} $, by
condition (i), so that
\begin{equation*}%
  \h \xy=\h\xx=\{\cs ex\}=\k\xx=\k\xy\comma\qquad
\h\xz=\h\yz\comma\qquad\k\xz=\k\yz\comma
\end{equation*}
and therefore
\begin{equation*}
  \h \xy\scir\h\xz=\h\xz\comma\quad\k\xy\scir\h\yz=\h\yz=\h\xz\comma\quad
\k\xz\scir\k\yz=\k \yz\scir\k \yz=\k\yz\per
\end{equation*}  It follows
that  the isomorphism $\vphih\xy$ induced by $\vphi\xy$ on
$\G\wx/(\h\xy\scir\h\xz)$ coincides with the identity automorphism
of $\G\wx/\h\xz$, the isomorphism $\vphih\yz$ on
$\G\wy/(\k\xy\scir\h\yz)$ coincides with $\vphi\yz$, and the
isomorphism $\vphih\xz$ induced by $\vphi\xz$ on
$\G\wx/(\h\xy\scir\h\xz)$   coincides with $\vphi\xz$. On the
other hand, the coset that determines the inner automorphism
$\ttrip xyz$ is the subgroup \[\cc xyz=\cc xxz=\h\xz\comma\] by
condition (v), so that $\ttrip xyz$ must be the identity
automorphism of $\G\wx/\h\xz$. Consequently,
\[\vphih\xy\rp\vphih\yz=\vphi\yz=\vphi\xz=\ttrip
xyz\rp \vphih\xz\comma\]so semi-frame condition (iv) holds in this
case.  The cases when $\wy=\wz$ and when $\wx=\wz$ are treated in a
completely analogous fashion.

It remains to consider the case when $\wx$, $\wy$, and $\wz$ are
all distinct.  Condition (vi) of the theorem implies that
\begin{alignat*}{3}
\cc x y z\mo&=\cc x z y,&\qquad \cc y x z\mo&=\cc y z x,&\qquad\cc
z x y\mo&=\cc z y x\comma\tag{1}\labelp{Eq:simpsfr1}\\
\intertext{from which it follows that} \ttrip x y z\mo&=\ttrip x z
y,&\qquad\ttrip  y z x\mo &= \ttrip y x z,&\qquad\ttrip z x
y\mo&=\ttrip z y x\per\tag{2}\labelp{Eq:simpsfr2}
\end{alignat*}
For example, for every coset $D$ in $\G\wy/(\h\yz\scir\h\yx)$, we
have
\begin{multline*}
  \ttrip yxz(\ttrip yzx(D))=\ttrip yxz(\cc yzx\mo\scir D\scir\cc yzx)
  =\cc yxz\mo\scir(\cc yzx\mo\scir D\scir \cc yzx)\scir\cc yxz\\=\cc
  yzx\scir\cc yzx\mo\scir D\scir\cc yzx\scir\cc yzx\mo=D\comma
\end{multline*}
by the definition of $\ttrip yzx$, the definition of $\ttrip yxz$,
the second equation in \refEq{simpsfr1}, and the laws of group
theory.  This argument shows that the composition of $\ttrip yzx$
and $\ttrip yxz$ is the identity function on its domain.  The same
is also true of the reverse composition, so these two inner
automorphisms are the inverses of one another.

The next step is to check that
\begin{equation*}\tag{3}\labelp{Eq:simpsfr21}
\ttrip x y z\rp\hvphs\xy = \hvphs\xy\rp\ttrip y z
x\qquad\text{and}\qquad \ttrip x y
z\rp\hvphs\xz=\hvphs\xz\rp\ttrip z x y\per
\end{equation*} To verify the first equation, consider an
arbitrary coset   $D$ in $\G\wx/(\hh)$\per  The definition of
$\ttrip xyz$, the isomorphism properties of $\vphi\xy$, condition
(vii) of the theorem, and the definition of $\ttrip yzx$ imply
that
\begin{multline*}
\vphi\xy[\ttrip x y z[D]]=\vphi\xy[\cc x y z\mo\scir D\scir \cc x
y z]=\vphi\xy[\cc x y z]\mo\scir\vphi\xy[ D]\scir \vphi\xy[\cc x y
z]\\= \cc y z x \mo\scir\vphi\xy[ D]\scir  \cc y z x= \ttrip y z x
[\vphi\xy[ D]]\per
\end{multline*}
  An analogous argument, using condition (viii) in place of
condition (vii), establishes the second equation in
\refEq{simpsfr21}.

Consider finally the case when all of the indices $x$, $y$, and $z$
are distinct. Assume $x<y<z$, and use condition (iv) of the theorem
to obtain
\begin{equation*}\tag{4}\labelp{Eq:simpsfr23}
\hvphs\xy\rp\hvphs\yz=\ttrip x y z\rp\hvphs\xz\per
\end{equation*}
Compose both sides of this equation on the right  with
$\hvphs\yz\mo$, and on the left with $\ttrip x y z\mo$, to
arrive at
\[\ttrip x y z\mo\rp\hvphs\xy= \hvphs\xz\rp\hvphs\yz\mo\per
\]
The mapping $\ttrip x y z\mo$ coincides with $\ttrip x z y$\comma
by \refEq{simpsfr2}, and  $\hvphs\yz\mo$ coincides with
$\hvphs\zy$, because, as has already been pointed out,
semi-frame condition (ii) is valid in $\mc F$. The previous
equation may therefore be rewritten in the form
\begin{equation*}\tag{5}\labelp{Eq:simpsfr24}
\hvphs\xz\rp\hvphs\zy=\ttrip x z y\rp\hvphs\xy\comma
\end{equation*}
which is a permuted version of  \refEq{simpsfr23} in which the
second and third indices $y$ and $z$ have been transposed. Compose
both sides of \refEq{simpsfr23} on the right with
$\hvphs\xz\mo$ and on the left with $\hvphs\xy\mo$ to obtain
\[\hvphs\yz\rp\hvphs\xz\mo=\hvphs\xy\mo\rp\ttrip x y z\per
\]
Observe that
\[\hvphs\xy\mo\rp\ttrip x y z =\hvphs\xy\mo\rp\ttrip x y z \rp\hvphs\xy\rp\hvphs\xy\mo
=\hvphs\xy\mo\rp\hvphs\xy\rp\ttrip y z x \rp\hvphs\xy\mo
= \ttrip y z x \rp\hvphs\xy\mo\comma\] by the properties of
isomorphism composition and  \refEq{simpsfr21}.  It follows from
these computations and from the validity of semi-frame condition
(ii) in $\mc F$ that
\begin{equation*}\tag{6}\labelp{Eq:simpsfr25}
\vphih\yz\rp\vphih\zx=\hvphs\yz\rp\hvphs\xz\mo=\hvphs\xy\mo\rp\ttrip x y z=\ttrip y z x
\rp\hvphs\xy\mo= \ttrip y z x \rp\hvphs\yx\comma
\end{equation*}
which is a permuted version of \refEq{simpsfr23} in which the
indices have been shifted one to the left  modulo 3, so that $x$,
$y$, and $z$ have been replaced by $y$, $z$, and $x$ respectively.
This argument shows that the two permuted versions of
\refEq{simpsfr23}, the first obtained by transposing the last two
indices $y$ and $z$ of the triple $\trip xyz$ to arrive at
\refEq{simpsfr24}, and the second by shifting each of the indices
$x$, $y$, and $z$ of the triple to the left by one modulo 3 to
arrive at \refEq{simpsfr25}, are valid in $\mc F$.  All
permutations of the triple $\trip xyz$ may be obtained by
composing these two permutations.  For example,   transpose the
last two indices of \refEq{simpsfr23}, permuting $\trip xyz$ to
$\trip xzy$, to  obtain  \refEq{simpsfr24}, and then use
\refEq{simpsfr25} to shift the indices of \refEq{simpsfr24} to the
left by one modulo 3, permuting $\trip xzy$ to $\trip zyx$, to
arrive at
\[\vphih \zy\rp\vphih\yx=\ttrip zyx\rp\vphih\zy\per\] It follows
that semi-frame condition (iv) is valid in $\mc F$\per

The next step in the proof  is the verification of the coset
conditions for the identity law, the second  involution law, and
the cycle law under the assumption of conditions (i)--(viii) of
the theorem. Certainly, $\mc F$ will satisfy the coset condition
for the identity law, since this is just the equality of the last
two cosets in condition (v) of the theorem. In order to verify the
coset conditions for the second involution law and the cycle law,
which coincide with conditions (ii) and (iv) in \refL{cysim},  it
suffices to show that conditions (i) and (v) of that  lemma,
namely
\begin{align*}
\cc x y z\mo&=\cc x z y\tag{10}\labelp{Eq:simcos1}
\\ \intertext{and}
\vphi\xy[\cc x y z]&=\cc y z x\comma\tag{11}\labelp{Eq:simcos2}
\end{align*} hold
for all triples $\trip x y z$ in $\ez 3$\per If two of the
indices, say
 $x$ and $y$, are equal, then
\begin{alignat*}{4}
  \cc xyz&=\cc xxz&&=\h\xz&&=\{e_\wx\}\scir\h\xz&&=\h\xy\scir\h\xz\comma\tag{12}\label{Eq:simcos2.1}\\
  \cc xzy&=\cc
  xzx&&=\h\xz&&=\h\xz\scir\{e_\wx\}&&=\h\xz\scir\h\xy\comma\\
  \cc zxy&=\cc zxx&&=\h\zx&&=\h\zx\scir\h\zx&&=\h\zx\scir\h\zy\comma
\end{alignat*}
by the assumption on $x$ and $y$, condition (v) (with $z$ in place
of $y$),  condition (i), which implies that $\h\xy=\{e_\wx\}$, and,
for the second to the last equality in the last line, the assumption
that $\h\zx$ is a subgroup of $\G\wz$ and therefore closed under
composition. It is clear from this argument that \refEq{simcos2.1}
holds  for all permutations of the indices $x$, $y$, and $z$. Apply
\refL{idcy} to arrive at \refEq{simcos1}. The cases $y=z$ and  $x=z$
are handled in a similar fashion.

As regards the verification of \refEq{simcos2}, if two of the
indices, say $x$ and $y$ are equal, then \refEq{simcos2.1} holds for
all permutations of the variables $x$, $y$, and $z$, and therefore
\refL{idcy} yields \refEq{simcos2}. A similar argument applies if
$y=z$ or $x=z$.

Assume now that all three indices $x$, $y$, and $z$ are distinct.
If $x<y<z$, then \refEq{simcos2} holds, by condition (vii) of the
theorem.  To derive the permuted version of \refEq{simcos2} in
which the indices $x$ and $y$ are transposed, use condition (vi),
the isomorphism properties of $\vphi\xy$, condition (vii), and
condition (vi) (with $y$, $z$, and $x$ in place of $x$, $y$, and
$z$ respectively) to obtain
\begin{equation*}
\vphi\xy[\cc x z y] =\vphi\xy[\cc x y z\mo]=\vphi\xy[\cc x y
z]\mo=\cc y z x\mo=\cc y x z\per
\end{equation*}
Apply $\vphi\yx$ to  the first and last terms in this string of
equalities, and use the fact that $\vphi\yx$ is the inverse of
$\vphi\xy$, by semi-frame condition (ii), to arrive at
\begin{equation*}\tag{13}\labelp{Eq:simcos3}
\vphi\yx[\cc yxz] =\cc xzy\per
\end{equation*} To derive the permuted version of \refEq{simcos2}
in which $x$, $y$, and $z$ are shifted one  to the right modulo 3
to obtain the equation for $z$, $x$, and $y$ respectively,
\begin{equation*}\tag{14}\labelp{Eq:simcos4}
\vphi\zx[\cc z x y]=\cc x y z\comma
\end{equation*}
apply $\vphi\zx$ to  both sides  of  condition (viii), and use
semi-frame condition (ii) (with $z$ in place of $y$).

The permutation of the triple $\trip xyz$ implicit in
\refEq{simcos3} that is obtained by transposing the first two
indices to obtain $\trip yxz$, and the permutation of the triple
implicit in \refEq{simcos2} that is obtained by shifting each index
to the right by one modulo 3 to obtain $\trip zxy$, together
generate all permutations of $\trip xyz$, and hence all permutations
of \refEq{simcos2}. For example, use \refEq{simcos3} to shift all
the indices of \refEq{simcos2} to the right by one modulo 3,
permuting $\trip xyz$ to $\trip zxy$ and arriving at
\refEq{simcos4}, and then repeat this process on \refEq{simcos4},
permuting $\trip zxy$ to $\trip yzx$, to arrive at
\[\vphi \yz[\cc yzx]=\cc zxy\per\]  From these observations, it is clear  that
\refEq{simcos2} holds for all permuted versions  of a given triple
of distinct elements in $\ez 3$. Combine this with the arguments
following \refEq{simcos2} to see that \refEq{simcos2} holds for all
triples  in $\ez 3$. Use \refEq{simcos1}, \refEq{simcos2}, and
\refL{cysim} to conclude the coset conditions for the second
involution law and the cycle law hold in $\mc F$. This completes the
derivation of the coset conditions for the identity law, the second
involution law, and the cycle law from conditions (i)--(viii) above.

To establish the reverse implication,   assume  $\mc F$ is a
semi-frame  satisfying the coset conditions for the identity law,
the second involution law, and the cycle law. Certainly, $\mc F$
satisfies conditions (i)--(iv) of the theorem, because these
conditions are special cases of the semi-frame conditions. To see
that $\mc F$ satisfies condition (v), use the coset condition for
the identity law  for the pair $\pair yx$, which says that $\cc
yxx=\h\yx$, use the definition of $\vphi\yx$\comma and use
semi-frame condition (ii) in the form of \refCo{con11} (with $x$
and $y$ interchanged), to obtain
\begin{equation*}\tag{15}\label{Eq:sico16}
  \vphi\yx[\cc y x x]=\vphi\yx[\h\yx]=\k\yx=\h\xy\per
\end{equation*}
The coset conditions for the second involution law and the cycle
law are conditions (ii) and (iv) of \refL{cysim}, so they imply
all of the other conditions of the lemma.  In particular, they
imply  (v) (with $y$, $x$, and $x$ in place of  $x$, $y$, and $z$
respectively), so
\begin{equation*}\tag{16}\label{Eq:sico17}
\vphi\yx[\cc y x x]=\cc x x y\per
\end{equation*}
 Combine \refEq{sico16} and \refEq{sico17} to arrive at
\begin{equation*}\tag{17}\label{Eq:sico18}
  \cc xxy=\h\xy\per
\end{equation*} Invoke \refL{cysim} again, this time using (i)
(with $x$ and $y$ in place of $y$ and $z$ respectively), to obtain
\[\cc xxy\mo=\cc xyx\per\] Combine this equation
with \refEq{sico18}, and use the fact that $\h\xy$ is a subgroup
of $\G\wx$ and therefore closed under inverse, to arrive at
\begin{equation*}\tag{18}\label{Eq:sico19}
  \cc xyx=\cc xxy\mo=\h\xy\mo=\h\xy\per
\end{equation*} Together, the coset condition for the identity
law, \refEq{sico18}, and \refEq{sico19} imply condition (v) of the
theorem. To derive conditions (vi), (vii), and (viii) of the
theorem, use \refL{cysim} again, and in fact parts (i), (v), and
(iii) respectively. This completes the proof of the first
assertion of the theorem.

To prove the second assertion of the theorem, suppose that $\mc F$
satisfies conditions (i)--(viii) of  the theorem. It follows from
the first part of the theorem that $\mc F$ must be a semi-frame
that satisfies the first three coset conditions.  The key step in
the argument is showing that $\mc F$ satisfies the coset condition
for  the associative law for one quadruple of elements in $\ez 4$
if and only if it satisfies the   condition for every permutation
of that quadruple.

Fix a quadruple $\quadr x y z w$ in $\ez 4$ of not necessarily
distinct elements\comma and suppose that
\begin{equation*}\tag{19}\labelp{Eq:assoc1}
\cc x y z\scir\cc x z w = \vphi\yx[\cc y z w\scir\h\yx]\scir\cc x
y w\per
\end{equation*}
The immediate goal is to derive a permuted version of
\refEq{assoc1} in which the indices $z$ and $w$ have been
transposed. Form the coset inverses of both sides of
\refEq{assoc1}, and apply the second involution law for cosets, to
obtain
\begin{equation*}\tag{20}\labelp{Eq:assoc2}
\cc x z w\mo\scir\cc x y z\mo = \cc x y w\mo\scir\vphi\yx[\cc y z
w\scir\h\yx]\mo\per
\end{equation*}
Conditions (ii) and (iv) in \refL{cysim} hold for all triples of
indices in $\ez 3$, because $\mc F$ satisfies the coset conditions
for the second involution law and the cycle law.  Consequently,
part (i) of the lemma holds for all such triples.  Use it
repeatedly on different triples to  obtain
\begin{equation*}\tag{21}\labelp{Eq:assoc3}
\cc x y z\mo=\cc x  z y,\qquad\cc x y w\mo=\cc x  w y,\qquad \cc x
z w\mo=\cc x  w z,\qquad\cc y z w\mo=\cc y w z\per
\end{equation*} Expand the second term on the right side of
\refEq{assoc2} as follows:
\begin{multline*}\tag{22}\labelp{Eq:assoc4.1}
\vphi\yx[\cc y z w\scir\h\yx]\mo=\vphi\yx[(\cc y z
w\scir\h\yx)\mo]=\vphi\yx[\h\yx\mo\scir\cc y z
w\mo]\\=\vphi\yx[\h\yx\scir\cc y z w\mo]=\vphi\yx[\cc y z
w\mo\scir\h\yx]=\vphi\yx[\cc y w z\scir\h\yx]\comma
\end{multline*}
by the isomorphism properties of $\vphi\yx$\comma the second
involution law for cosets,  the assumption that $\h\yz$ is a
normal subgroup of $\G\wx$, and hence   is closed under inverses
and   commutes with all elements in $\G\wx$, and the final
equation in \refEq{assoc3}. Combine \refEq{assoc4.1} with
\refEq{assoc2} and the first three equations in \refEq{assoc3} to
arrive at
\begin{multline*}\tag{23}\labelp{Eq:assoc4}
\cc x w z\scir\cc x z y =\cc x z w\mo\scir\cc x y z\mo = \cc x y
w\mo\scir\vphi\yx[\cc y z w\scir\h\yx]\mo\\= \cc x w
y\scir\vphi\yx[\cc y w z\scir\h\yx]\per
\end{multline*}
Multiply the first and last expressions in  \refEq{assoc4} on the
left by $\cc x w y\mo$ and on the right by $\cc x z y\mo$, and use
the inverse law for cosets, to  obtain
\begin{equation*}\tag{24}\labelp{Eq:assoc6}
\cc x w y\mo\scir\cc x w z = \vphi\yx[\cc y w z\scir\h\yx]\scir\cc
x z y\mo\per
\end{equation*}  In more detail, the inverse law for cosets, the
assumption that $\cc xzy$ is a coset
of $\h\xz\scir\h\xy$,  and the assumption that the subgroup
$\h\xy$ is normal yield
\begin{multline*}
  \cc xwy\mo\scir\cc xwz\scir\cc xzy\scir\cc xzy\mo=\cc xwy\mo\scir\cc
  xwz\scir\h\xz\scir\h\xy\\=\cc xyw\mo\scir\h\xy\scir\cc xwz\scir \h\xz=\cc
  xyw\mo\scir\cc xwz\per
\end{multline*} The final equality is justified because
$\cc x w z$ is a coset of the normal subgroup
$\h\xw\scir\h\xz$\comma and therefore \textit{absorbs} the factor
$\h\xz$ in the sense that
\[\cc x w z\scir\h\xz=\cc x w z\scir(\h\xw\scir\h\xz)\scir\h\xz=\cc x w
z\scir\h\xw\scir\h\xz=\cc x w z\comma
\]  by the identity law for groups of cosets, the assumption that $\cc xwz$ is
a coset of $\h\xw\scir\h\xz$, and the assumption that $\h\xz$ is a
subgroup of $\G\wx$ and  therefore closed under composition.
Similarly, the coset $\cc x w y\mo$ of $\h\xw\scir\h\xy$ absorbs
the factor $\h\xy$\per An analogous argument shows that the
product of $\cc x z y$ with its inverse is absorbed by the term
$\vphi\yx[\cc y w z\scir\h\yx]$ on the right side of equation
\refEq{assoc6}. This completes the justification of the
computation in \refEq{assoc6}.  Combine  the first and second
equations in \refEq{assoc3} with \refEq{assoc6} to conclude that
\begin{equation*}\tag{25}\labelp{Eq:assoc5}
\cc x y w\scir\cc x w z = \vphi\yx[\cc y w z\scir\h\yx]\scir\cc x
y z\per
\end{equation*} This is just the desired
permuted version of \refEq{assoc1} in which the indices $z$ and
$w$ have been transposed.

The next goal is to derive a permuted version of \refEq{assoc1} in
which the indices $y$ and $w$ have been transposed. Begin with an
application of \refL{sfimage}(iii) (with  $\ww$  and  $\wy$ in place
of $\wy$  and  $\wz$ respectively, and with   $\cc y w z\scir\h\yx$
in place of $Q$) to obtain
\begin{equation*}\tag{26}\label{Eq:assoc8.1}
\cc x w y\scir\vphi\xy\mo[\cc y w
z\scir\h\yz]=\vphi\xw\mo[\vphi\wwy\mo[\cc y w
z\scir\h\yx]]\scir\cc x w y\per
\end{equation*}
Notice in this connection that  $\cc y w z$ is a coset of
$\h\yw\scir\h\yz$\comma so the product $\cc y w z\scir\h\yx$ is a
coset of $\h\yw\scir\h\yz\scir\h\yx$\comma and therefore  a union of
cosets of $\h\yw\scir\h\yx$\per This latter group coincides with
$\k\xy\scir\k\wwy$\comma by semi-frame condition (ii) and
\refCo{con11}, so the hypotheses of \refL{sfimage}(iii) are indeed
satisfied. Use semi-frame condition (ii)  to rewrite
\refEq{assoc8.1} as
\begin{equation*}\tag{27}\labelp{Eq:assoc8}
\cc x w y\scir\vphi\yx[\cc y w z\scir\h\yz]=\vphi\wwx[\vphi\yw[\cc
y w z\scir\h\yx]]\scir\cc x w y\per
\end{equation*}
The argument of $\vphi\wwx$ on the right  side of \refEq{assoc8}
may be rewritten as
\begin{align*}\tag{28}\labelp{Eq:assoc9}
\vphi\yw[\cc y w z\scir\h\yx]&=\vphi\yw[\cc y w
z\scir\h\yw\scir\h\yx]\\ &=\vphi\yw[\cc y w
z]\scir\vphi\yw[\h\yw\scir\h\yx]\per
\end{align*}
The first equality uses the fact that $\cc y w z$ is a coset of
$\h\yw\scir\h\yz$ and therefore absorbs $\h\yw$, and the second
uses the isomorphism properties of $\vphi\yw$\per The function
$\vphi\yw$ maps the group $\k\xy\scir\h\yw$ to the group
$\k\xw\scir\k\yw$, by the second equation in condition (iii) of
the theorem (with $w$ in place of $z$), which has been shown to
hold for all triples in $\ez 3$. The first of these groups
coincides with $\h\yx\scir\h\yw$, and the second with
$\h\wwx\scir\h\wwy$\comma by semi-frame condition (ii) and
\refCo{con11}, so (using also the assumption that the subgroups
involved are normal)
\begin{equation*}\tag{29}\labelp{Eq:assoc10}
\vphi\yw[\h\yw\scir\h\yx]=\h\wwy\scir\h\wwx\per
\end{equation*}
Also, parts (ii) and (iv) of \refL{cysim} hold for all triples in
$\ez 3$, because $\mc F$ satisfies the coset conditions for the
second involution law and the cycle law.  Apply part (v) of the
lemma (with $y$ and $w$ in place of $x$ and $y$ respectively) to
obtain
\begin{equation*}\tag{30}\labelp{Eq:assoc11}
\vphi\yw[\cc y w z]=\cc w z y\per
\end{equation*}  Combine
\refEq{assoc9}--\refEq{assoc11}, and use the fact that the coset
$\cc wzy$ of $\h\wwz\scir\h\wwy$ absorbs the subgroup $\h\wwy$,
to arrive at
\begin{multline*}\tag{31}\labelp{Eq:assoc12}
\vphi\yw[\cc y w z\scir\h\yx]=\vphi\yw[\cc y w
z]\scir\vphi\yw[\h\yw\scir\h\yx]\\=\cc w z
y\scir\h\wwy\scir\h\wwx=\cc w z y\scir\h\wwx\per
\end{multline*}
Replace the occurrence in \refEq{assoc8} of the left side of
\refEq{assoc12} with the right side of \refEq{assoc12} to get
\begin{equation*}
\cc x w y\scir\vphi\yx[\cc y w z\scir\h\yz]=\vphi\wwx[\cc w z
y\scir\h\wwx]\scir\cc x w y\per
\end{equation*} Combine this with \refEq{assoc4} to conclude that
\begin{equation*}\tag{32}\labelp{Eq:assoc7}
\cc x w z\scir\cc x z y = \vphi\wwx[\cc w z y\scir\h\wwx]\scir\cc
x w y\comma
\end{equation*} which is the permuted version of \refEq{assoc1}
in which the indices $\wy$ and $\ww$ have been transposed.

Finally, we derive a permuted version of \refEq{assoc1} in which
the indices $x$ and $y$ have been transposed. Apply $\vphi\xy$ to
both sides of \refEq{assoc1} to obtain
\begin{equation*}\tag{33}\labelp{Eq:assoc14}
\vphi\xy[\cc x y z\scir\cc x z w] = \vphi\xy[\vphi\yx[\cc y z
w\scir\h\yx]\scir\cc x y w]\per
\end{equation*} The  left side of \refEq{assoc14} may
be   rewritten as
\begin{multline*}\tag{34}\labelp{Eq:assoc35.1}
\vphi\xy[\cc x y z\scir\cc x z w] =\vphi\xy[\cc x y
z\scir\h\xy\scir\cc x z w]=\vphi\xy[\cc x y z\scir\cc x z
w\scir\h\xy]\\=\vphi\xy[\cc x y z]\scir\vphi\xy[\cc x z
w\scir\h\xy]= \cc  y z x\scir\vphi\xy[\cc x z w\scir\h\xy]\per
\end{multline*}
The first equality uses the fact that the coset $\cc x y z$ of
$\h\xy\scir\h\xz$ absorbs the subgroup $\h\xy$\comma the second uses
the assumption that $\h\xy$ is normal, the third uses the
isomorphism properties of $\vphi\xy$ (which is why it is necessary
to insert a copy of $\h\xy$ to compose with $\cc xzw$) \comma and
the fourth uses \refL{cysim}(v). The right side of \refEq{assoc14}
may be rewritten as
\begin{align*}\tag{35}\labelp{Eq:assoc36.1}
\vphi\xy[\vphi\yx[\cc y z w\scir\h\yx]\scir\cc x y w]&=
\vphi\xy[\vphi\yx[\cc y z w\scir\h\yx]]\scir\vphi\xy[\cc x y w]\\
&= \cc y z w\scir\h\yx\scir\vphi\xy[\cc x y w]\\ &= \cc y z
w\scir\h\yx\scir\cc y w x\\ &= \cc y z w\scir\cc y w x\comma
\end{align*}
by isomorphism properties of $\vphi\xy$\comma semi-frame condition
(ii),  \refL{cysim}(v) (with $w$ in place of $z$), and the  fact
that the coset $\cc y w x$ absorbs the group $\h\yx$\per Combine
\refEq{assoc14}---\refEq{assoc36.1} to arrive at
\begin{equation*}
\cc y z w\scir\cc y w x=\cc  y z x\scir\vphi\xy[\cc x z
w\scir\h\xy]\per
\end{equation*}
Multiply both sides of the preceding equation by $\cc y z x\mo$ on
the left and by $\cc y w x\mo$ on the right, and use the inverse
law for groups of cosets, to obtain
\begin{equation*}\tag{36}\labelp{Eq:assoc15}
\cc y z x\mo\scir\cc y z w=\vphi\xy[\cc x z w\scir\h\xy]\scir\cc y
w x\mo\per
\end{equation*}
From \refL{cysim}(i), it follows that
\begin{equation*}\tag{37}\labelp{Eq:assoc38.1}
\cc y z x\mo=\cc y x z\qquad\text{and}\qquad\cc y w x\mo=\cc y x
w\per
\end{equation*} Combine \refEq{assoc15} and \refEq{assoc38.1} to conclude that
\begin{equation*}\tag{38}\labelp{Eq:assoc13}
\cc y x z\scir\cc y z w=\vphi\xy[\cc x z w\scir\h\xy]\scir\cc y x
w\comma
\end{equation*} which is the desired permuted version of \refEq{assoc1} obtained
by transposing the indices $x$ and $y$.

It has been shown that the three permuted versions of \refEq{assoc1}
obtained by transposing the indices $z$ and $w$, the indices $y$ and
$w$, and the indices $x$ and $y$, are all derivable from
\refEq{assoc1}. These three transpositions generate all permutations
of the quadruple $\quadr xyzw$, so it follows that every version of
\refEq{assoc1} in which the indices $x$, $y$, $z$, and $w$ have been
permuted is derivable from \refEq{assoc1}.

The next step is to derive all instances of the coset condition for
the associative law on the basis of condition (ix) of the theorem
and the assumption that $\mc F$ is a semi-frame satisfying
conditions (i)--(viii) of the theorem, or equivalently, satisfying
the first three coset conditions. Suppose that the first two indices
of an arbitrary quadruple in $\ez 4$, say $\quadr xyzw$, are equal,
with the goal of deriving \refEq{assoc1}. This derivation does not
require the use of condition (ix) at all. Observe that
\begin{equation*}\tag{39}\label{Eq:assoc16.1}
  \cc xyz=\cc xxz=\h\xz\qquad\text{and }\qquad\cc xyw=\cc
  xxw=\h\xw\comma
\end{equation*}
by the assumption on $x$ and $y$, and condition (v) of the
theorem.  Also, $\vphi\yx$ and $\h\yx$ coincide  with $\vphi\xx$
and $\{e_x\}$ respectively, and $\vphi\xx$ is the identity
function on $\G\wx/\{e_x\}$, by condition (i) of the theorem, so
\begin{equation*}\tag{40}\label{Eq:assoc17.1}
  \vphi\yx[\cc yzw\scir\h\yx]=\vphi\xx[\cc xzw\scir\{e_x\}]=\cc
  xzw\scir\{e_x\}=\cc xzw\per
\end{equation*} Consequently,
\begin{equation*}\tag{41}\label{Eq:assoc16}
  \cc xyz\scir\cc xzw=\h\xz\scir\cc xzw=\cc xzw\comma
\end{equation*}
by the first part of \refEq{assoc16.1} and the fact that the coset
$\cc xzw$ absorbs the subgroup $\h\xz$.  Therefore,
\begin{equation*}\tag{42}\label{Eq:assoc17}
  \vphi\yx[\cc yzw\scir\h\yz]\scir\cc xyw=\cc xzw\scir\cc xyw= \cc xzw\scir\h\xw=\cc
  xzw\comma
\end{equation*}
by \refEq{assoc17.1}, the second part of \refEq{assoc16.1}, and
the fact that the coset $\cc xzw$ absorbs the subgroup $\h\xw$.
Combine \refEq{assoc16} and \refEq{assoc17} to arrive at
\refEq{assoc1}.

Consider next the case of an arbitrary quadruple $\quadr xyzw$ in
$\ez 4$ in which at least two of the indices are equal. Form a
permutation of this quadruple in which two of the equal indices
are moved to the first and second positions of the quadruple.  The
resulting quadruple   satisfies the hypotheses of the preceding
paragraph, so the version of \refEq{assoc1} that is associated
with this quadruple is valid in $\mc F$, by the observations of
the previous paragraph. It follows that \refEq{assoc1}  must hold
for the given quadruple $\quadr xyzw$, since every permuted
version of a valid instance of the coset condition for the
associative law is also valid.

Turn finally to the case when the indices in a quadruple $\quadr
xyzw$ in $\ez 4$ are distinct.  If $x<y<z<w$, then \refEq{assoc1}
holds by the assumed condition (ix). Consequently, every permuted
version of \refEq{assoc1} also holds, so \refEq{assoc1} is valid
in $\mc F$ in all cases in which the indices of the given
quadruple are mutually distinct. Combine the observations of this
and the preceding paragraph  to conclude that if condition (ix) of
the theorem is true in a semi-frame $\mc F$ satisfying conditions
(i)--(viii), then the coset condition for the associative law
holds in $\mc F$.  The reverse implication is trivially true.
\end{proof}

The following special case of the second part of \refT{simcos} is
quite useful in verifying the coset condition for the associative
law in basic examples of semi-frames.

\begin{cor}\labelp{C:tripsub} Suppose  $\mc F$ is a  semi-frame satisfying the coset
conditions  for the identity law\co the second involution law\co
and the  cycle law\po  If
\[\h\xy\scir\h\xz\scir\h\xw=\G\wx\]
 for all quadruples $\quadr x  y z w$   in $\ez 4$\comma then $\mc F$
satisfies the coset conditions for the associative law\po
\end{cor}
\begin{proof}Consider a quadruple $\quadr x  y z w$   in $\ez 4$, with the
intention of showing that
\begin{equation*}\tag{1}\labelp{Eq:tripsub1}
\cc x y z\scir\cc x z w = \vphi\yx[\cc y z w\scir\h\yx]\scir\cc x
y w\per
\end{equation*}
Since $\cc x y z$ and $\cc x z w$ are cosets  of $\h\xy\scir\h\xz$
and $\h\xz\scir\h\xw$\comma the complex product
\[\cc x y z\scir\cc x z w
\]
is a coset  of the triple product
\[\h\xy\scir\h\xz\scir\h\xw\comma
\] which is $\G\wx$\comma  by assumption.  There is only one
coset of the improper subgroup $\G\wx$, namely   itself, so
\begin{equation*}\tag{2}\labelp{Eq:tripsub2}
\cc x y z\scir\cc x z w=\G\wx\per
\end{equation*}
As regards the right side  of \refEq{tripsub1}\comma because $\cc
y z w$ is a coset of $\h\yz\scir\h\yw$\comma the product
\[\cc y z w\scir\h\yx
\]
is  a coset of the triple  product
\[\h\yz\scir\h\yw\scir\h\yx\comma
\] which is $\G\wy$\comma by assumption.  Therefore,
\[\cc y z w\scir\h\yx=\G\wy\per
\]
Apply  the mapping $\vphi\yx$ to  both sides of the previous
equation to obtain
\begin{equation*}
\vphi\yx[\cc y z w\scir\h\yx]=\vphi\yx[\G\wy]=\G\wx\per
\end{equation*} Multiply the first and
last terms  of this equation on the right by $\cc x y w$ to arrive
at
\begin{equation*}\tag{3}\labelp{Eq:tripsub3}
\vphi\yx[\cc y z w\scir\h\yx]\scir\cc x y w=\G\wx\scir\cc x y
w=\G\wx\per
\end{equation*}
Combine  \refEq{tripsub2} and  \refEq{tripsub3} to see that
  \refEq{tripsub1} holds in this case. Apply \refT{simcos} to conclude that
   coset conditions  for the associative law are valid in $\mc F$.
\end{proof}

There are a number of  other special cases in which the
verification of the coset conditions for a given semi-frame
simplify.  For instance, in many of the examples of group triples,
most of cosets $\cc xyz$ in the coset shifting system are the
identity coset in the sense that they are the identity element of
the corresponding quotient group,
\[\cc xyz=\h\xy\scir\h\xz\per\] The next corollary is perhaps the simplest  example
of such a special case.     Call two cosets $\cc xyz$ and $\cc uvw$
\textit{associated} if $\trip uvw$ is a permutation of $\trip xyz$.

\begin{cor}\labelp{C:trithm}Let
$\mc F$ be a semi-frame\comma and $\trip p q r$  a
 triple in $\ez 3$ with $p<q<r$\per  If every coset   not associated
 with $\cc pqr$ is the identity coset\comma then $\mc F$   satisfies  the four coset conditions
 if and only if the following conditions hold\per
 \begin{enumerate}
 \item[(i)]$\cc p q r\mo=\cc p r q$\comma and $\cc q r p\mo=\cc q p r$\comma  and $\cc
r p q\mo=\cc r q p$\per
\item[(ii)]$\vphi\pq[\cc p q r]=\cc q r p$\per
\item[(iii)]$\vphi\pr[\cc p q r]=\cc r p q$\per
\item[(iv)] $\cc p q r\seq\tbigcap\{\h\pq\scir\h\pr\scir\h{ps}:\pair ps\in\mc E
\text{ and } s\neq  p, q, r\}$\per
\end{enumerate}
\end{cor}

\begin{proof} Assume the conditions of the corollary, with the goal of verifying the conditions of
\refT{simcos}. The assumption that $\mc F$ is a semi-frame implies
that conditions (i)---(iv) of \refT{simcos} are satisfied. Also,
condition (v) of the theorem holds.  To see this, consider an
arbitrary pair $\pair x y$   in $\mc E$.  The  cosets
\[\cc x x y,\qquad\cc x y y,\qquad\cc x y x\comma\]
are  identity cosets, by assumption, so
\[\cc x x y
=\h\xx\scir\h\xy=\{e_x\}\scir\h\xy=\h\xy=\h\xy\scir\{e_x\}=\h\xy\scir\h\xx=\cc x y x \]
and
\[\cc x y
y=\h\xy\scir\h\xy=\h\xy\per
\]
The second and fifth equalities use   semi-frame condition (i).

To verify that condition (vi)  of the theorem is equivalent to condition (i) of the
corollary (under the basic assumption of the corollary), let
$\trip x y z$ be a triple in
$\ez 3$ of pairwise distinct elements.  If
$\trip x y z$ is not associated with $\trip pqr$, then
\begin{equation*}
\cc x y z\mo=(\hh)\mo=\h\xz\mo\scir\h\xy\mo=\h\xz\scir\h\xy=\cc x z y\per
\end{equation*}
The first and last equality use the basic assumption of the
corollary, the second uses the second involution law for group
complexes, and the third uses the fact that $\h\xy$ and $\h\xz$ are
subgroups, and hence closed under the operation of forming inverses.
If $\trip x y z $ is an associate of $\trip pqr$, then condition
(vi) of the theorem holds by  condition (i) of the corollary, and
vice versa.

The next step is to check that conditions (vii) and (viii) of the
theorem are respectively equivalent to conditions (ii) and (iii) of
the corollary. Let $\trip x y z$ be a triple in $\ez 3$ with
$x<y<z$. If this triple is not $\trip pqr$, then it cannot be an
associate of $\trip pqr$, because of the ordering, and therefore
\[\vphi\xy[\cc x y z]=\vphi\xy[\hh]=\kh=\h\yx\scir\h\yz=\h\yz\scir\h\yx=\cc
y z x\per
\]
The first and fifth equalities hold by the basic assumption of the corollary, the second by
semi-frame condition (iii), the third by semi-frame condition  (ii) (and semi-frame
condition (i) in the case when $x=y$), and the
fourth by the fact that the subgroups are normal and hence commute with one another.  A completely analogous argument shows
that
\[\vphi\xz[\cc x y z]=\cc z x y\per
\] Thus, in  this case, conditions (vii) and (viii) of the theorem hold.
If the triple $\trip x y z$ is $\trip pqr$, then conditions (vii) and (viii) of the
theorem are exactly conditions (ii) and (iii) of the corollary.

The associative law coset  conditions will hold for  all permutations  of  a
quadruple $(x, y, z, w)$ just in case
\begin{equation*}\tag{1}\labelp{Eq:onetri1}
\cc x y z\scir\cc x z w=\vphi\yx[\cc y z w\scir\h\yx]\scir\cc x y w\comma
\end{equation*}
by Associative Law \refT{alawc}.  By assumption,
\begin{equation*}
\cc x y w=\h\xy\scir\h\xw,\qquad\cc x z w=\h\xz\scir\h\xw,\qquad
\cc y z w=\h\yz\scir\h\yw
\end{equation*} (under the hypothesis that $w$ is different from $p$, $q$, and $r$)\comma
so equation \refEq{onetri1} can equivalently be rewritten as
\begin{equation*}\tag{2}\labelp{Eq:onetri2}
\cc x y
z\scir\h\xz\scir\h\xw=\vphi\yx[\h\yz\scir\h\yw\scir\h\yx]\scir\h\xy\scir\h\xw\per
\end{equation*}
It is a consequence of semi-frame condition (iii) that
\begin{equation*}
\vphi\yx[\h\yz\scir\h\yw\scir\h\yx]=\h\xy\scir\h\xz\scir\h\xw\comma
\end{equation*}
so the  right-hand side of \refEq{onetri2} reduces to
$\h\xy\scir\h\xz\scir\h\xw$\per
On the other hand,
\[\cc x y z\scir\h\xz=\cc x y z\comma
\]
since $\cc x y z$ is a coset of $\h\xy\scir\h\xz$, so the left-hand
side of \refEq{onetri2} reduces to $\cc x y z\scir\h\xw$\per Thus,
\refEq{onetri2}  is equivalent to
\begin{equation*}\tag{3}\labelp{Eq:onetri3}
\cc x y z\scir\h\xw=\h\xy\scir\h\xz\scir\h\xw\per
\end{equation*}
Finally, since $\cc x y z$ is a coset of $\h\xy\scir\h\xz$\comma
equation \refEq{onetri3} will hold just  in case $\cc x  y z$ is a subset
of  $\h\xy\scir\h\xz\scir\h\xw$\per

If $\trip xyz$ is not an associate of  $\trip pqr$, then $\cc xyz$
is the identity coset $\h\xy\scir\h\xz$, and so the desired
inclusion is trivial.  If $\trip xyz$ is an associate of $\trip
pqr$, then of course $\trip pqr$ is an associate of $\trip xyz$, and
for $\trip pqr$, the desired inclusion holds by condition (iv) of
the corollary.  This means that condition \refEq{onetri1} holds for
$\trip pqr$, and hence also for the original triple $\trip xyz$,
since the validity of \refEq{onetri1} for one triple implies its
validity for all associates of the triple.

The remaining parts of the proof are trivial and are left to the reader.
\end{proof}

The final observation we wish to make is that in a coset relation
algebra     $\cra CF$,  the operation $\,\otimes\,$ reduces to
relational composition in all those cases in which the indices $x$,
$y$, and $z$ of the coset $\cc xyz$ used to define the relative
product $\r\xy\a\otimes\r\yz\b$ are not mutually distinct.

\begin{cor}\labelp{C:otimesrp} If $\mc F$ is a group triple
satisfying conditions \textnormal{(i)--(viii)} of Theorem
\textnormal{\ref{T:simcos}}\comma then
\[\r\xy\lph\otimes\r\yz\bt=\r\xy\lph\rp\r\yz\bt\]
for every triple $\trip x y z$  in $\ez 3$ in which at least two
of the indices $x$\co $y$\co  $z$ are equal\per
\end{cor}
\begin{proof} According to \refL{otimes.1},
\[\r\xy\lph\otimes\r\yz\bt=\r\xy\lph\rp\r\yz\bt
\]
if and only if
\begin{equation*}\tag{1}\label{Eq:trithm.1}
\cc x y z=\hh\per
\end{equation*}
The verification that \refEq{trithm.1} follows from conditions
(i)--(viii) of \refT{simcos}  is nearly identical to the argument
establishing \refEq{simcos2.1} in the proof of \refT{simcos}.  The
details are left to the reader.
\end{proof}

\section{Example}\labelp{S:sec5}

In this section, and example of  a coset relation algebra that is
not representable is constructed. Start with a group pair
\[\mc F=\pair G \vp=\pair{\langle\grp x:x\in I\,\rangle}{\langle\ho x
y:\pair x y\in I\times I\,\rangle}\]   in which the index set $I$
has five elements, say
\[I=\{p,q,r,s,t\}\per
\]  Each of the groups
$\grp x$ is assumed to be a copy of the Cartesian product
${\mathbb Z}_2\times {\mathbb Z}_2\times {\mathbb Z}_2$\comma
where ${\mathbb Z}_2=\{0,1\}$ denotes the cyclic group of order
two, and these copies are assumed to be mutually disjoint.
\begin{figure}[tbh]
\includegraphics[scale=.60]{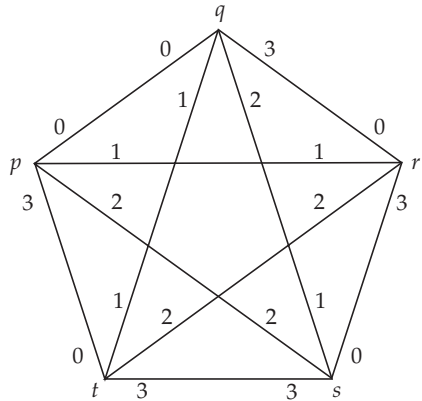}
\caption{Normal subgroup diagram.}\labelp{F:fig5}
\end{figure}
To describe the subgroups $\hll x y$ and $\hr x y$ for distinct
indices $x$ and $y$ in $I$, consider the following four subgroups
of ${\mathbb Z}_2\times {\mathbb Z}_2\times {\mathbb Z}_2$\,:
\begin{alignat*}{2} &\lz 0 =
{\mathbb Z}_2\times\{\ze\}\times\{\ze\}&&,\qquad \lz 1 =
\{\ze\}\times {\mathbb Z}_2\times\{\ze\}\comma\\ &\lz 2 = \{\ze\}
\times\{\ze\}\times {\mathbb Z}_2&&,\qquad \lz 3 =
\{(\ze,\ze,\ze), (\on,\on,\on)\}\per
\end{alignat*} Take $\hll x y$\comma respectively $\hr x y$\comma to
be the copy of one of these four subgroups in $\grp x$\comma
respectively $\grp y$\comma according to the prescriptions given
in \refF{fig5}. For example, the subgroup $\hll p {\myspace t}$ is
the copy of $\lz 3$ in $\grp p$ and the subgroup $\hr p {\myspace
t}$ is the copy of $\lz 0$ in $\grp t$\comma because the edge
between the vertices $p$ and $t$ in the diagram is labeled with
$3$ and $0$. Similarly, the subgroup $\hll q s$ is the copy of
$\lz 2 $ in $\grp q$ and the subgroup $\hr q s $ is the copy of
$\lz 1$ in $\grp s$\comma because the edge from $q$ to $s$ is
labeled with $2$ and $1$.

The quotient isomorphisms $\ho x y$ when $x$ and $y$ are equal are
of course taken to be the appropriate identity automorphisms of
$\G x/\{e_x\}$ for every $x$ in $I$.   For distinct $x$ and $y$,
they are completely determined by the requirement that
$\hvphs\xy\relprodd\hvphs\yz=\hvphs\xz$\per For instance, according to
the diagram in \refF{fig5}, we must have
\[\ho p q[\lz 0\scir\lz 3]=\lz 0\scir\lz 1,\qquad
\ho p q[\lz 0\scir\lz 1]=\lz 0\scir\lz 3,\qquad \ho p q[\lz
0\scir\lz 2]=\lz 0\scir\lz 2\]
\begin{figure}[tbh]
\includegraphics[scale=.60]{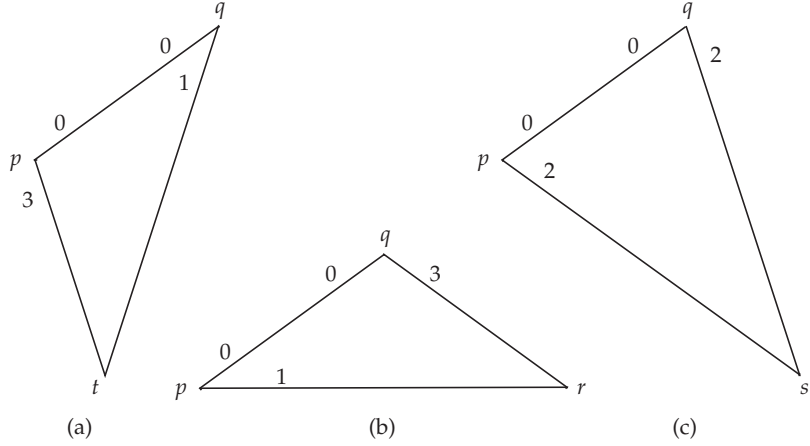}
 \caption{The triangles from the pentagon that determine
$\ho p q$\per}\labelp{F:fig6}
\end{figure}(see (a), (b), and (c) respectively in  \refF{fig6}). (The composite
subgroups  on the left, inside the brackets, should actually be
interpreted as denoting their copies in $\grp p$\comma and the
composite subgroups on the right should be interpreted as denoting
their copies in $\grp q$.) These three requirements determine $\ho p
q$ in the following way.  According to the pentagon, the copy of the
subgroup $\lz 0$ in $\G p$ is mapped by $\ho pq$ to the copy of the
subgroup $\lz 0$ in $\G q$.  The subgroup $\lz0$ has four cosets in
${\mathbb Z}_0\times{\mathbb Z}_0\times{\mathbb Z}_0$, namely
\begin{alignat*}{2}
C_0&=\trip 000\scir\lz 0=\{\trip 000, \trip 100\}\comma&\quad
C_1&=\trip 010\scir\lz 0 =\{\trip 010, \trip 110\}\comma\\
C_2&=\trip 001\scir\lz 0 =\{\trip 001, \trip 101\}\comma&\quad
C_3&=\trip 011\scir\lz 0 =\{\trip 011, \trip 111\}\per
\end{alignat*}
Observe that
\begin{align*}
\lz 0\scir\lz 3&=\{\trip 000,\trip 100\}\scir\{\trip 000,\trip
111\}\\&=\{\trip 000,\trip 100, \trip 011,\trip 111\}= C_0\cup
C_3\comma\\ \lz 0\scir\lz 1&=\{\trip 000,\trip 100\}\scir\{\trip
000,\trip 010\}\\&=\{\trip 000,\trip 100, \trip 010,\trip 110\}=
C_0\cup C_1\comma\\ \lz 0\scir\lz 2&=\{\trip 000,\trip
100\}\scir\{\trip 000,\trip 001\}\\&=\{\trip 000,\trip 100, \trip
001,\trip 101\}= C_0\cup C_2\per
\end{align*} Because $\ho pq$  maps the copies of $\lz0\scir\lz 3$ and $C_0$ in
$\G p$ respectively to the copies of $\lz 0\scir\lz 1$ and $C_0$
in $\G q$, it must map the copy of $C_3$ in $\G p$ to the copy of
$C_1$ in $\G q$, by the preceding observations.  Similarly, it
must map the copies of $C_1$ and $C_2$  in $\G p$ respectively to
the copies of $C_3$ and $C_2$   in $\G q$.

The resulting group pair $\mc F$ is easily seen to be a frame, so
the group relation algebra $\cra G F$ exists. The  next step is to
modify the operation of relative multiplication  in $\cra GF$ by
introducing a coset system
\[C=\langle\cc xyz:\trip xyz\in I\times I\times I\rangle\per\]  If a triple
of indices $\trip xyz$ is not a permutation of the triple $\trip
pqr$, take $\cc xyz$ to be the identity coset,
\[\cc xyz=\h\xy\scir\h\xz\per\]  Suppose now that $\trip xyz$
is a permutation of $\trip pqr$.  As is clear from \refF{fig5}, two
different edges emanating from a given vertex $x$ are labeled with
distinct numbers, so the subgroup $\h\xy\scir\h\xz$ is a composition
of two distinct subgroups of $\G\wx$ of order $2$, and therefore has
order $4$.  It follows that the quotient group
\begin{equation*}\tag{1}\label{Eq:qgr}
\G\wx/(\h\xy\scir\h\xz)
\end{equation*}  has order $2$, so it has exactly
   two cosets, the identity coset and the non-identity coset. Take
   $\cc xyz$ to be the non-identity coset,
   \[\cc xyz=
   \G\wx\sim(\h\xy\scir\h\xz)\per\] It is not difficult to check
   that the resulting group triple
   \[\bar{\mc F}=\trip G \vp C\]
is a  coset semi-frame that satisfies the coset conditions.  For
example, the quotient group in \refEq{qgr} is abelian, so the inner
automorphism of \refEq{qgr} determined by the coset $\cc xyz$ must
be the identity automorphism.  Use, in addition, the fact that $\mc
F$ is a group frame to verify semi-frame condition (iv) for
$\bar{\mc F}$,
\[\vphih\xy\rp\vphih\yz=\vphih\xz=\tau\rp\vphih\xz\per\]

The proof that $\bar{\mc F}$ satisfies the coset conditions is
based on \refC{trithm}.  It suffices to check that conditions
(i)--(iv) of that corollary are satisfied.  As regards condition
(i), the quotient group in \refEq{qgr} has order $2$, so every
coset is its own inverse. Consequently,
\[\cc pqr\mo=\cc pqr=\cc prq=\G p\sim(\h\pq\scir\h\pr)\comma\]
and similarly \[\cc qrp\mo=\cc qpr\qquad\text{and}\qquad \cc
rpq\mo=\cc rqp\per\] As regards conditions (ii) and (iii), the
quotient isomorphisms $\vphih\pq$ and $\vphih\pr$ induced   by
$\vphi\pq$ and $\vphi\pr$ respectively map the identity coset to
the identity coset, and consequently they map the non-identity
coset to the non-identity coset. It follows that
\begin{multline*}
\vphi\pq[\cc pqr]=\vphi\pq[\G p\sim(\h\pq\scir\h\pr)]=\G
q\sim(\k\pq\scir\h\qr)\\=\G q\sim(\h\qp\scir\h\qr)=\cc qrp\comma
\end{multline*} and similarly, $\vphi\pr[\cc pqr]=\cc rpq$.
Finally, to verify condition
(iv) of the corollary, observe that each of the four edges
emanating from vertex $p$ in \refF{fig5} is labeled with a
different number. Consequently, the composite subgroups
\[\h\pq\scir\h\pr\scir\h\pw\] for $w=s,t$ have order 8, that is to say,
they coincide with $\G p$. The coset $\cc pqr$
is trivially included in their intersection, since
\[(\h\pq\scir\h\pr\scir\h\pps)\cap(\h\pq\scir\h\pr\scir\h\ppt)=\G
p\per\]  Apply \refC{trithm} to arrive at the following
conclusion.

\begin{theorem}\labelp{T:theorem7} The group triple $\bar{\mc F}$
is a coset semi-frame that satisfies the coset conditions\per
Consequently\comma the corresponding  algebra $\crah CF$ is a full
coset relation algebra and hence an example of a finite\comma
 measurable relation algebra\per
\end{theorem}

It is instructive to look somewhat closer at the operation
$\,\otimes\,$ of relative multiplication in the  algebra $\crah CF$
just constructed, and to compare it with the corresponding operation
in $\cra GF$.   On atoms, $\,\otimes\,$ is determined   by
\begin{align*}\rs\xy \al\otimes\rs \wwz \beta &= \rs \xy
\al\relprodd\rs \wwz \beta \intertext{whenever $y\neq w$, or $y=w$
and $\{x,y,z\}\neq\{p,q,r\}$, and} \rs\xy \al\otimes\rs \yz \beta
&= \grp x\times \grp y\sim(\rs \xy \al\relprodd\rs \yz \beta)
\end{align*} whenever $\{x,y,z\}=\{p,q,r\}$.   Thus, the operation of relative
multiplication in $\crah CF$ is obtained by changing only slightly
the operation of relational composition in $\cra G F$ as it affects
atomic relations, namely, for those pairs of atomic relations
$\rs\xy\al$ and $\rs\yz\beta$ that are indexed, in some order, by a
permutation $\trip xyz$ of the triple $\trip pqr$, the  relative
product has been shifted to the complement of what it is in $\cra G
F$ .

\renewcommand\pr{{pr}}
\newcommand\px{{px}}
\newcommand\xq{{xq}}
\newcommand\py{{py}}
\newcommand\yq{{yq}}
\renewcommand\pq{{pq}}
\renewcommand\qr{{qr}}
\newcommand\xt{{xt}}
\newcommand\sq{{sq}}
\newcommand\sx{{sx}}
\newcommand\tq{{tq}}
\renewcommand\ps{{ps}}
\renewcommand\pt{{pt}}
\newcommand\rqq{{rq}}

It turns out that the full coset relation algebra of the theorem
is not representable as a set relation algebra, and in particular,
it is not isomorphic to a full group relation algebra.
\begin{theorem}\labelp{T:theorem70} The finite measurable relation algebra
$\cra C {\bar{\mc F}}$ is not representable\per
\end{theorem}
\begin{proof}
Write $\f A=\cra C {\bar{\mc F}}$.  The argument that $\f A$ is not
representable proceeds by contradiction.  Assume that it is
representable, say $\vth$ is a representation of $\f A$  over a base
set  $V$.  Because $\f A$ is simple in the algebraic sense of the
word (see the remarks preceding \refT{simple1} below), it may be
assumed that the unit of the representation is the Cartesian square
$V\times V$ (see, for example, Theorem 16.18 in \cite{giv18b}).  We
identify $R_{xx,0}$ with $x$ in the proof, so that the set $I$
becomes the set of measurable atoms of $\f A$. This permits some
simplification in the  notation.

The first step is to use the representation $\vth$ for constructing
a \emph{scaffold} in $\f A$, that is to say, a system
 of atoms $\langle \cs a\xy:x,y \in I\rangle$ satisfying the following three conditions
 for all measurable atoms $x$, $y$, and $z$ in $I$.\begin{align*} \cs a \xx&=x\per\tag{1}\labelp{Eq:70.1}\\
\cs a \yx&=\cs a \xy\ssm\per\tag{2}\labelp{Eq:70.2}\\
\cs a \xz &\le\cs a\xy\otimes\cs a\yz\per\tag{3}\labelp{Eq:70.3}
\end{align*} Each element $x$ in $I$ is a subidentity atom, so its image $\vth(x)$ must be $\id{\cs Vx}$
for some non-empty subset $\cs Vx$ of $V$, these sets are mutually disjoint for distinct $x$, and
because $\f A$ is finite, \[\tbigcup\{\id{\cs Vx}:x\in I\} = \tbigcup\{\vth(x):x\in I\}=\vth(\tsum I)=\vth(\ident)=\id V\per\]
For each $x$ in $I$, choose an element $\cs vx$ in $\cs Vx$, and for each pair of elements $x$, $ y $, let $\cs a\xy$
be the unique atom in $\f A$ such that \[\pair {\cs vx}{\cs vy}\in \vth(\cs a\xy)\per\]
Since $\vth(x)$ is the unique atom containing $\pair {\cs vx}{\cs
vx}$, property \refEq{70.1} follows. Since $\vth(\cs a\xy\ssm)$ is
an atom (the converse of an atom is an atom) that contains $\pair
{\cs vy}{\cs vx}$, by the representation properties of $\vth$,
property \refEq{70.2} follows. Since $\pair {\cs vx}{\cs vy}$ is in
$\vth(\cs a\xy)$ and $\pair {\cs vy}{\cs vz}$ is in $\vth(\cs
a\yz)$, it follows from the definition of relational composition
that $\pair {\cs vx}{\cs vz}$ is in $\vth(\cs a\xy)\rp \vth(\cs
a\yz)$.  The representation properties of $\vth$ imply that
\[\vth(\cs a\xy)\rp \vth(\cs a\yz)=\vth(\cs a\xy\otimes\cs a\yz)\per\]
Thus, $\vth(\cs a\xz)$ and $\vth(\cs a\xy\otimes\cs a\yz)$ have a
non-empty intersection---they both contain the pair $\pair {\cs
vx}{\cs vz}$---so the former, which is an atom,  must be below the
latter. Use the representation properties of $\vth$ one more time to
conclude that \refEq{70.3} holds.  This completes the proof of the
three scaffold conditions.

Here are some further properties of the elements $\cs a\xy$ that we
shall need.  Notice that each such atom is actually one of the
atomic binary relations of $\f A$ on the base set $U=\tbigcup\{\G
x:x\in I\}$, so it makes sense to speak of the pairs in $\cs a\xy$.
The converse of each atom is the set-theoretic relational inverse,
in symbols,
\begin{align*}
\cs a\yx&=\cs a\xy\mo\per\tag{4}\label{Eq:70.4}\\ \intertext{Second,
the relative product of two elements is the set-theoretic relation
composition of the elements as long as the set of indices
$\{x,y,z\}$ does not coincide with the set $\{p,q,r\}$,} \cs
a\xy\otimes\cs a\yz&=\cs a\xy\rp\cs a\yz\per\tag{5}\label{Eq:70.5}
\end{align*} Third, the relative product is disjoint from the relational
composition when the two sets of indices $\{x,y,z\}$ and $\{p,q,r\}$ are equal,
 \begin{equation*}\tag{6}\label{Eq:70.6}
\cs a\xy\otimes\cs a\yz=\G x\times\G y\sim\cs a\xy\rp\cs a\yz\per
\end{equation*} Fourth, the intersection of certain relative products that share a
common ``edge" is an atom when that common edge is $pq$ or $qr$ or $pr$.  Specifically, 
 \begin{equation*}\tag{7}\label{Eq:70.7}
(\cs a\ps\otimes\cs a\sq)\cap (\cs a\pt\otimes\cs a\tq)=\cs
a\pq\comma
\end{equation*}
and similarly if $pq$ is replaced by either $qr$ or $pr$.

Choose elements $\cs us$ and $\cs ut$ in $U$ so that
\begin{equation*}\tag{8}\label{Eq:70.8}
\pair{\cs us}{\cs ut}\in\cs a\st\per
\end{equation*}
Such a choice is possible because $\cs a\st$ is a non-empty binary relation.  Since for each $x=p,q,r$
\begin{equation*}\tag{9}\label{Eq:70.9}
\cs a \st\le \cs a\sx\rp\cs a\xt\comma
\end{equation*} by \refEq{70.3} and \refEq{70.5}, the pair in \refEq{70.8} must also belong to the the
right side of \refEq{70.9}, so that there must be an element $\cs ux$ in $U$ for which
\[\pair{\cs us}{\cs ux}\in\cs a\sx\qquad\text{and}\qquad \pair{\cs ux}{\cs ut}\in\cs a\xt\comma\]by \refEq{70.8}.
In particular, take $x=p,q$, and use \refEq{70.4}, to obtain
\[\pair{\cs up}{\cs us}\in\cs a\ps\qquad\text{and}\qquad \pair{\cs us}{\cs uq}\in\cs a\sq\comma\] so that
\[\pair{\cs up}{\cs uq}\in\cs a\ps\rp\cs a \sq=\cs a\ps\otimes\cs a\sq, \] and also to obtain
\[\pair{\cs up}{\cs ut}\in\cs a\pt\qquad\text{and}\qquad \pair{\cs ut}{\cs uq}\in\cs a\tq\comma\] so that
\[\pair{\cs up}{\cs uq}\in\cs a\pt\rp\cs a \tq=\cs a\pt\otimes\cs a\tq\per \]  Apply \refEq{70.7} to arrive at
\begin{equation*}\tag{10}\label{Eq:70.10}
\pair{\cs up}{\cs uq}\in\cs a\pq\per
\end{equation*} Similar arguments applied to $p$ and $r$ and to $r$ and $q$ lead to
\begin{equation*}\tag{11}\label{Eq:70.11}
\pair{\cs up}{\cs ur}\in\cs a\pr\qquad\text{and}\qquad
\pair{\cs ur}{\cs uq}\in\cs a\rqq\per
\end{equation*}  In view of the definition of relational composition,  \refEq{70.11} implies that
\begin{equation*}\tag{12}\label{Eq:70.12}
\pair{\cs up}{\cs uq}\in\cs a\pr\rp \cs a\rqq\per
\end{equation*}

Together, \refEq{70.10} and \refEq{70.12} show that the intersection
\[\cs a\pq\cap(\cs a\pr\rp\cs a \rqq)\] is not empty, since both factors contain the pair $\pair{\cs up}{\cs uq}$\per
The left-hand factor is an atom, so
\begin{equation*}\tag{13}\label{Eq:70.13}
\cs a\pq \seq \cs a\pr\rp \cs a\rqq\per
\end{equation*}  On the other hand,
\[\cs a\pq \seq \cs a\pr\otimes \cs a\rqq = \G p\times\G q\sim\cs a\pr\rp \cs a\rqq\comma\] by
\refEq{70.3} and \refEq{70.6}.  This is a direct contradiction to \refEq{70.13}, so the assumption
that $\f A$ is representable cannot be tenable.
\end{proof}

The group  ${\mathbb Z}_2$ can be replaced everywhere in the
preceding construction by an arbitrary non-trivial abelian group.
The mappings $\ho x y$ are no longer uniquely determined, and the
definition of relative multiplication is slightly more involved. In
each case we get an atomic, measurable relation algebra that is not
representable.  These are new examples of non-representable relation
algebras, with a completely different underlying motivation than the
examples  that have appeared so far in the literature.

\section{A Decomposition Theorem}\labelp{S:sec6}

The  isomorphism index set $\mc E$ of a coset semi-frame $\mc
F=\trip G\vp C$ satisfying the coset conditions is  an equivalence
relation on the group index set $I$, and the unit
\[E=\tbigcup\{\G\wx\times\G\wy:\pair xy\in \mc E\}\] of the
corresponding  full coset relation algebra $\cra C F$ is an
equivalence relation on the base set $U=\tbigcup_{x\in I}\cs Gx$.
Call the semi-frame $\mc F$ \textit{simple} if the group index set
$I$ is not empty, and if $\mc E$ is the universal relation on the
index set $I$. It turns out that $\mc F$ is simple in this sense of
the word if and only if the algebra $\cra CF$ is simple in the
algebraic sense of the word, namely, it has more than one element
and every non-constant homomorphism on the algebra is injective; or,
equivalently, the algebra  has exactly two ideals, the trivial ideal
and the improper ideal.

\begin{theorem}\labelp{T:simple1} Let $\mc F$ be a semi-frame satisfying
the coset conditions\po  The  coset relation algebra $\cra C F$ is
simple if and only if the semi-frame $\mc F$ is simple\po
\end{theorem}
\begin{proof} We begin with a preliminary observation: for all
triples $\trip xyz$ in $\ez 3$,
\begin{equation*}\tag{1}\label{Eq:simple1.1}
\tbigcup\{\r\xy\a\otimes\r\yz\b:\a<\kai\xy\text{ and
}\b<\kai\yz\}=\G\wx\times\G\wz\per
\end{equation*}
For the proof, suppose that $\trip xyz$ is in $\ez 3$.  The
definition of $\,\otimes\,$ implies  that
\begin{equation*}\tag{2}\labelp{Eq:squaresum1}
\r\xy\a\otimes\r\yz\b=\tbigcup\{\r\xz\g:
\hs\xz\g\seq\vphi\xy\mo[\ks\xy\a\scir\hs\yz\b]\scir\cc x y z\}\per
\end{equation*}
Each relation $\r\xz\g$ is included in
\begin{equation*}\tag{3}\label{Eq:simple1.3}
\G\wx\times\G\wz\comma
\end{equation*} by
Partition \refL{i-vi}, so each product of the form
\refEq{squaresum1} is included in \refEq{simple1.3},
 and therefore the left side of \refEq{simple1.1}
is included in the right side.

To establish the reverse inclusion, notice that as the indices
$\a$ and $\b$ vary, the complex products $\ks\xy\a\scir\hs\yz\b$
run through all cosets of the subgroup $\kh$. The function
$\vphi\xy$ induces an isomorphism from the quotient group
$\G\wx/(\hh)$ to the quotient group $\G\wy/(\kh)$, so the inverse
images $\vphi\xy\mo[\ks\xy\a\scir\hs\yz\b]$ must run through all
of the cosets of $\hh$\per  It follows that, as $\a$ and $\b$
vary, the complex products
\begin{equation*}\tag{4}\labelp{Eq:squaresum2}
\vphi\xy\mo[\ks\xy\a\scir\hs\yz\b]\scir\cc x y z
\end{equation*}
must also run through all cosets of $\hh$, because $\cc x y z$ is
a fixed element of the quotient group $\G\wx/(\hh)$.  Thus, for
each index $\g<\kai\xz$\comma there are indices $\a<\kai\xy$ and
$\b<\kai\yz$ such that the coset $\hs\xz\g$ of $\h\xz$ is included
in \refEq{squaresum2}. The relation $\r\xz\g$  is therefore
included in $\r\xy\a\otimes\r\yz\b$\comma by \refEq{squaresum1}.
The union of all of the relations $\r\xz\g$  is \refEq{simple1.3},
by Partition \refL{i-vi}, so the right side of \refEq{simple1.1}
must be included in the   left side.

Turn now to the proof of the theorem, and assume first that the
semi-frame $\mc F$ is simple. The isomorphism index set $\mc E$ is
the universal relation on the group index set $I$, by assumption,
so
\begin{multline*}\tag{5}\label{Eq:simple1.4}
U\times U=(\tbigcup_{\wx\in I}\G\wx)\times (\tbigcup_{\wy\in
I}\G\wy)=\tbigcup\{\G\wx\times\G\wy:\wx,\wy\in U\}\\
=\tbigcup\{\r\xy\a:x,y\in U\text{ and
}\a<\kai\xy\}=\tbigcup\{\G\wx\times\G\wy:\pair\wx \wy\in \mc
E\}=E\comma
\end{multline*}
 by the definition of $U$, the distributivity of Cartesian products over
 arbitrary unions, Partition \refL{i-vi}, the assumption on $\mc E$, and the
 definition of $E$.  The index set $I$ is assumed to be non-empty,
 and the groups are non-empty, so the unit $U\times U$
of $\cra CF$ is non-empty and therefore different from the zero
element $\varnot$. In particular, the relation algebra $\cra CF$
has more than one element.

In order to show that a non-degenerated, atomic relation algebra
is simple, it suffices to show that the equation $1;r;1=1 $ holds
for every subidentity atom $r$ (see, for example, Givant\,\cite{giv18}, Theorem 9.2).
A subidentity atom of $\cra C F$
has the form $\r \yy0$ for some $\wy$ in $I$, so it must be shown
that
\begin{equation*}\tag{6}\labelp{Eq:simple13} (U\times U)\otimes \r
\yy 0\otimes (U\times U)=U\times U
\end{equation*} for every $\wy$ in $I$.
 Use \refEq{simple1.4} and    the distributivity of
$\,\otimes\,$ over arbitrary unions to rewrite the left side of
\refEq{simple13} as the union of the relations
\begin{equation*}\tag{7}\labelp{Eq:simple14}
\r\xu\a\otimes\r\yy 0\otimes\r\vz\b
\end{equation*}  over all $x,u,v,z$ in $I$, with $\a<\kai\xu$ and
$\b<\kai\vz$. If  $\wu\neq\wy$ or $\wv\neq\wy$, then the relation
in \refEq{simple14} reduces to the empty relation, by the
definition of $\,\otimes\,$. The left side of \refEq{simple13} is
therefore equal to the union of the relations
\begin{equation*}\tag{8}\labelp{Eq:simple15}
\r\xy\a\otimes\r\yy 0\otimes\r\yz\b
\end{equation*}
over all $x$ and $ z$ in  $I$, with $\a<\kai\xy$ and $\b<\kai\yz$.
The coset condition for the identity law, which $\mc F$ is assumed
to satisfy, and Identity Law \refT{identthm2}, imply that
\begin{equation*}
\r\xy\a\otimes\r\yy 0=\r\xy\a\per
\end{equation*}    Consequently,
\refEq{simple15} reduces to
\begin{equation*}\tag{9}\labelp{Eq:simple16}
\r\xy\a\otimes\r\yz\b\per
\end{equation*} For fixed $\wx$ and $\wz$, the union, over all
$\a$ and $\b$, of the relations in \refEq{simple16}
 is \refEq{simple1.3},
 by the preliminary observation in
\refEq{simple1.1}.  The union of all relations of the form
\refEq{simple14} therefore coincides with the union of all
relations of the form \refEq{simple1.3}, and this latter union is
just $U\times U$, by \refEq{simple1.4}. Conclusion: the equation
in \refEq{simple13} holds in $\cra CF$ for all $\wy$ in $I$, as
was to be shown.

We postpone the proof of the reverse implication of the theorem
until after the next theorem.
\end{proof}

It turns out that every full coset relation algebra can be
decomposed into the direct product of simple, full coset relation
algebras, or equivalently,  full coset relation algebras on simple
frames. We sketch briefly how this decomposition may be
accomplished. Given an arbitrary coset semi-frame
\[\mc F=(\langle \G x:x\in
I\,\rangle\smbcomma\langle\vph_\xy:\pair xy\in \mc E\rangle\smbcomma
\rangle\smbcomma\langle\cc x y z:\trip xyz\in \ez3 \rangle)\comma\]
consider an equivalence class $J$ of the isomorphism index set $\mc
E$. The universal relation $J\times J$ on $J$ is a subrelation of
$\mc E$, and in fact it is a maximal connected component of $\mc E$
in the graph-theoretic sense of the word. The \textit{restriction}
of $\mc F$ to $J$  is defined to be the group triple
\[\mc F_J=(\langle \G x:x\in
J\,\rangle\smbcomma\langle\vph_\xy:\pair xy\in J\times
J\rangle\smbcomma \rangle\smbcomma\langle\cc x y z: \trip xyz\in
J\times J\times J\rangle)\]    Each such restriction of $\mc F$ to
an equivalence class of the index set $\mc E$ inherits the coset
semi-frame properties of $\mc F$, and is therefore a simple
semi-frame. Call these restrictions the \textit{components} of $\mc
F$. Clearly, $\mc F$ is the disjoint union of its components in the
sense that the group system, the isomorphism system, and the coset
system of $\mc F$ are obtained by respectively forming the unions of
the group systems, the isomorphism systems, and the coset systems
of the components of $\mc F$. It is also easy to see that $\mc F$
satisfies the coset conditions if and only if each component
satisfies the coset conditions, because these conditions are
formulated only for cosets $\cc xyz$ such that the elements $x$,
$y$, and $z$ all belong to  the same equivalence class of $\mc E$.

If $\mc F$ is a semi-frame satisfying the coset conditions, then
so is each component $\cs{\mc F}J$, and consequently $\cras C {\cs
{\mc F}J }$ is a full coset relation algebra that is simple, with
base set and unit
\[\cs UJ=\tbigcup_{x\in J}\cs Gx\qquad\text{and}\qquad\cs Ej=\cs
UJ\times\cs UJ\] respectively. The coset relation algebra $\cra
CF$ is isomorphic to the direct product of the simple coset
relation algebras $\cras C{\cs {\mc F}J}$ constructed from the
components of $\mc F$ (so $J$ varies over the equivalence classes
of $\mc E$). In fact, if internal direct products are used instead
of Cartesian direct products, then $\cra C F$ is actually equal to
the internal direct product of the full coset relation algebras
constructed from its component semi-frames.
\begin{theorem}[Decomposition Theorem]\labelp{T:cosetdecomp} Every  full
coset relation algebra is isomorphic to a direct product of full
coset relation algebras on simple frames\per\end{theorem}

The details of the proof of this theorem are left to the reader.

Return now to the proof of the reverse implication in
\refT{simple1}. Assume that the given semi-frame $\mc F$ is not
simple. If the group index set $I$ is empty, then the base set $U$
is also empty, and in this case $\cra C F$ is a one-element
relation algebra with the empty relation as its only element. In
particular, $\cra CF$ is not simple. On the other hand, if the
group index set $I$ is non-empty, then the isomorphism index set
$\mc E$ has at least two equivalence classes, by the definition of
a simple semi-frame.  The coset relation algebra $\cra C F$ is
isomorphic to the direct product of the coset relation algebras on
the component semi-frames of $\mc F$, by Decomposition
\refT{cosetdecomp}, and there are at least two such components.
Each of these components is a simple semi-frame that satisfies the
coset conditions, so the corresponding coset relation algebra must
be  simple, by the first part of  the proof of \refT{simple1}.  It
follows that $\cra CF$ is isomorphic to a direct product of at
least two simple relation algebras, so $\cra CF$ cannot be simple.
For example, the projection of $\cra C F$ onto one of the factor
algebras is a non-constant homomorphism that  is not injective.

\end{document}